\documentclass[12pt, leqno]{amsart}

\usepackage{graphicx}
\usepackage{lscape}
\usepackage{amssymb,amsmath,amsthm,amscd}
\usepackage{mathrsfs}
\usepackage{enumerate}
\usepackage{oldgerm}
\usepackage[usenames,dvipsnames]{color}
\usepackage[colorlinks=true, pdfstartview=FitV,
 linkcolor=blue,citecolor=blue,urlcolor=blue]{hyperref}

\usepackage[all]{xy}

\allowdisplaybreaks[4]
\numberwithin{equation}{section}
\newcommand{\nc}{\newcommand}

\newtheorem{theorem}{Theorem}[section]
\newtheorem{proposition}[theorem]{Proposition}
\newtheorem{corollary}[theorem]{Corollary}
\newtheorem{conjecture}{Conjecture}
\newtheorem{maintheorem}[conjecture]{Main Theorem}
\newtheorem{Itheorem}[conjecture]{Theorem}

\newtheorem{lemma}[theorem]{Lemma}

\newtheorem{sublemma}[theorem]{Sublemma}

\theoremstyle{definition}
\newtheorem{definition}[theorem]{Definition}

\newtheorem{remark}[theorem]{Remark}

\nc{\Th}{\begin{theorem}}
\nc{\enth}{\end{theorem}}
\nc{\Lemma}{\begin{lemma}}
\nc{\enlemma}{\end{lemma}}
\nc{\Cor}{\begin{corollary}}
\nc{\encor}{\end{corollary}}
\nc{\Prop}{\begin{proposition}}
\nc{\enprop}{\end{proposition}}
\nc{\Sub}{\begin{sublemma}}
\nc{\ensub}{\end{sublemma}}

\numberwithin{equation}{section}
\nc{\hs}{\hspace*}

\newcommand{\M}{\mathsf{M}}
\newcommand{\D}{\mathrm{D}}
\nc{\N}{\Z_{\ge0}}
\newcommand{\Z}{\mathbb{Z}}
\newcommand{\B}{\mathbf{B}}
\newcommand{\Q}{\mathbb{Q}}

\newcommand{\A}{{\Z[q^{\pm1}]}}

\newcommand{\g}{\mathfrak{g}}

\newcommand{\n}{\mathfrak{n}}

\newcommand{\Uqg}{U_q(\g)}
\newcommand{\tUqg}{\widetilde{U}_q(\g)}

\newcommand{\seteq}{\mathbin{:=}}

\newcommand{\Hom}{\mathrm{Hom}}

\newcommand{\gmod}{\mbox{-$\mathrm{gmod}$}}

\newcommand{\supp}{{\rm supp}}
\newcommand{\wt}{{\rm wt}}
\newcommand{\ch}{{\rm ch}}

\newcommand{\Seed}{\mathscr{S}}

\newcommand{\rl}{\mathsf{Q}}   
\newcommand{\wl}{\mathsf{P}}   
\newcommand{\cmA}{\mathsf{A}}  
\nc{\on}{\operatorname}

\newcommand{\comp}{\Delta_+}
\newcommand{\comm}{\Delta_-}

\newcommand{\tens}{\mathop\otimes}

\newcommand{\soplus}{\mathop{\mbox{\normalsize$\bigoplus$}}\limits}
\newcommand{\sodot}{\mathop{\mbox{\normalsize$\bigodot$}}\limits}
\newcommand{\snconv}{\mbox{\scriptsize$\odot$}}
\newcommand{\shconv}{\mathbin{\large\diamond}}

\newcommand{\conv}{\mathop{\mathbin{\mbox{\large $\circ$}}}\limits}

\newcommand{\hconv}{\mathbin{\mbox{\large $\diamond$}}}

\nc{\tensp}{\otimes_{_+}\mspace{-1mu}}
\nc{\tensm}{\otimes_{_-}\mspace{-1mu}}

\newcommand{\ko}{\mathbf{k}}
\newcommand{\te}{\tilde{e}}
\newcommand{\tf}{\tilde{f}}
\newcommand{\ve}{\varepsilon}
\newcommand{\vs}{\overline{p}}
\newcommand{\vp}{\varphi}

\newcommand{\lan}{\langle}
\newcommand{\ran}{\rangle}
\newcommand{\ra}{\rangle}
\nc{\la}{\lambda}

\newcommand{\laa}{\langle}
\newcommand{\raa}{\rangle}

\newcommand{\vpi}{{\varpi_i}}

\newcommand{\low}{\mathrm{low}}
\newcommand{\up}{\mathrm{up}}

\newcommand{\ex}{\mathrm{ex}}
\newcommand{\fr}{\mathrm{fr}}

\newcommand{\Oint}{\mathcal{O}_{{\rm int}}}

\newcommand{\ri}{{\mspace{1mu}\rm r}}
\newcommand{\li}{{\rm l}}

\newcommand{\rmat}[1]{{\mathbf{r}}_{\mspace{-2mu}\raisebox{-.6ex}{${\scriptstyle{#1}}$}}}
\newcommand{\soc}{{\rm soc}}
\newcommand{\hd}{{\rm hd}}

\newcommand{\wB}{\widetilde{B}}
\newcommand{\La}{\Lambda}
\newcommand{\tLa}{\widetilde{\Lambda}}
\newcommand{\Lto}{\longrightarrow}

\nc{\de}{\on{\textfrak{d}}}

\newcommand{\eqn}{\begin{eqnarray*}}
\newcommand{\eneqn}{\end{eqnarray*}}

\newcommand{\Dv}{\mathbf{D}_\varphi}

\newcommand{\oi}{\overline{\iota}}
\newcommand{\sym}{\mathfrak{S}}

\newcommand{\K}{{J}}
\nc{\Kfr}{\K_{\mathrm{fr}}}
\nc{\Kex}{\K_{\mathrm{ex}}}

\newcommand{\gMod}{\mbox{-$\mathrm{gMod}$}}
\newcommand{\smod}{\mbox{-$\mathrm{mod}$}}
\newcommand{\Mod}{\mbox{-$\mathrm{Mod}$}}

\newcommand{\isoto}[1][]{\mathop{\xrightarrow%
[{\raisebox{.3ex}[0ex][.3ex]{$\scriptstyle{#1}$}}]%
{{\raisebox{-.6ex}[0ex][-.6ex]{$\mspace{2mu}\sim\mspace{2mu}$}}}}}

\nc{\cl}{\colon}
\nc{\ol}{\overline}
\nc{\Um}{U_q^-(\g)}
\nc{\U}[1][{\g}]{U_q(#1)}
\nc{\ro}{{\rm(}}
\nc{\rf}{{\rm)}}
\nc{\bio}{{\bar{\iota}}}
\nc{\bg}{{\bio_\g}}
\nc{\tU}[1][\g]{\widetilde{U}_q(#1)}
\nc{\An}{A_q(\n)}
\nc{\set}[2]{\left\{{#1}\mid{#2}\right\}}
\nc{\be}{\begin{enumerate}}
\nc{\ee}{\end{enumerate}}
\nc{\bnum}{\be[{\rm(i)}]}
\nc{\bnam}{\be[{\rm(a)}]}
\nc{\bl}{\bigl(}
\nc{\br}{\bigr)}
\newcommand{\To}[1][{\hs{2ex}}]{\xrightarrow{\,#1\,}}
\nc{\shc}{\mathcal{C}}
\nc{\ba}{\begin{array}}
\nc{\ea}{\end{array}}
\nc{\eq}{\begin{eqnarray}}
\nc{\eneq}{\end{eqnarray}}
\nc{\nn}{\nonumber}
\newcommand{\monoto}{\rightarrowtail}
\nc{\epito}{\twoheadrightarrow}
\nc{\noi}{\noindent}
\nc{\eps}{\varepsilon}
\nc{\tEs}{\widetilde{E}^*}
\def\max{{\mathop{\mathrm{max}}}}
\nc{\al}{\alpha}
\nc{\rtl}{\rl}
\nc{\Proof}{\begin{proof}}
\nc{\QED}{\end{proof}}
\nc{\pn}{p_\n}
\nc{\wb}[1]{\mbox{$\rule[-1.1ex]{0ex}{2ex}#1$}}
\nc{\vphi}{\varphi}
\nc{\dP}{\mathrm{E}^*}
\nc{\Up}{U_q^+(\g)}
\nc{\Ag}{A_q(\g)}
\nc{\QA}{\mathbf{A}}
\nc{\id}{\mathrm{id}}
\nc{\Pd}{\wl^+}
\nc{\tLt}{\widetilde{L}}
\nc{\Po}{\wl}
\nc{\De}[1]{\Delta(#1)}
\nc{\rt}{\ri}
\newcommand{\scbul}{{\,\raise1pt\hbox{$\scriptscriptstyle\bullet$}\,}}
\nc{\cor}{\ko}
\nc{\prtl}{\rtl_+}
\nc{\nrtl}{\rtl_-}
\nc{\Rm}{R^{\mathrm{ren}}}
\nc{\lt}{\mathrm{l}}
\nc{\tE}{\widetilde{E}}
\nc{\tF}{\widetilde{F}}
\nc{\bc}{\begin{cases}}
\nc{\ec}{\end{cases}}
\nc{\one}{\mathbf{1}}
\nc{\wtl}{\wt_\lt}
\nc{\wtr}{\wt_\rt}
\nc{\qtext}[1][{and}]{\quad\text{#1}\quad}
\nc{\Cmp}{\comp}
\nc{\Cmm}{\comm}
\nc{\Cm}{\Delta}
\nc{\Uq}{\U}

\newlength{\mylength}
\begin{document}

\title[Monoidal categorification of cluster algebras II]
{Monoidal categorification of cluster algebras II}

\author[S.-J. Kang, M. Kashiwara, M. Kim, Se-jin Oh]{Seok-Jin Kang$^{1}$, Masaki Kashiwara$^{2}$, Myungho Kim, Se-jin Oh$^3$}

\address{Department of Mathematical Sciences
         and
         Research Institute of Mathematics \\
         Seoul National University \\ Seoul 151-747, Korea}
         \email{sjkang@math.snu.ac.kr}

\address{Research Institute for Mathematical Sciences \\
          Kyoto University \\ Kyoto 606-8502, Japan \\
          \& Department of Mathematical Sciences
         and
         Research Institute of Mathematics \\
         Seoul National University \\ Seoul 151-747, Korea}
         \email{masaki@kurims.kyoto-u.ac.jp}

\address{School of Mathematics, Korea Institute for Advanced Study \\ Seoul 130-722, Korea}
         \email{mhkim@kias.re.kr}

\address{Department of Mathematical Sciences
         and
         Research Institute of Mathematics \\
         Seoul National University \\ Seoul 151-747, Korea}
         \email{sj092@snu.ac.kr}

\thanks{$^1$ This work was supported by NRF grant \# 2014021261
and NRF grant \# 2013055408
}
\thanks{$^2$ This work was supported by Grant-in-Aid for
Scientific Research (B) 22340005, Japan Society for the Promotion of
Science.}
\thanks{$^3$ This work was supported by BK21 PLUS SNU Mathematical Sciences Division}


\date{February 23, 2015}

\maketitle

\begin{abstract}
We prove that the quantum unipotent coordinate algebra $A_q(\n(w))$
associated with a symmetric Kac-Moody algebra and
its Weyl group element $w$
has a monoidal categorification as a quantum cluster algebra.
As an application of our earlier work,  we achieve it by showing the
existence of a quantum monoidal seed of $A_q(\n(w))$
which admits the first-step mutations in all the directions.
As a consequence, we solve the conjecture that
any cluster monomial is a member of the
upper global basis up to a power of $q^{1/2}$.

\end{abstract}

\tableofcontents

\section*{Introduction}
The quantum unipotent coordinate ring $A_q(\mathfrak n)$,
which is isomorphic to the negative  half  $U^-_q(\g)$ of a quantized enveloping algebra $U_q(\g)$ associated with a  symmetrizable Kac-Moody algebra $\g$,
has  very interesting bases so called
{\em upper global basis} and {\em lower global basis} (\cite{Kash91}).
In particular, the upper global basis $\bold B^\up$ has been studied
 emphasizing on its multiplicative structure.
For example, Berenstein and Zelevinsky (\cite{BZ93})
conjectured that,
 in the case $\g$ is of type $A_n$,
the product $b_1 b_2$ of two elements $b_1$ and
$b_2$ in $\bold B^\up$ is again an element of $\bold B^\up$ up to a
multiple of a power of $q$
 if and only if they are $q$-commuting;
 i.e., \ $b_1b_2=q^mb_2b_1$ for some $m\in\Z$.
This conjecture turned out to be not true in general,
because Leclerc (\cite{L03}) found examples of
an {\em imaginary}  element $b \in \bold B^\up$
such that $b^2$ does not belong to $\bold B^\up$.
Nevertheless, the idea of considering  subsets of $\bold B^\up$
whose elements are $q$-commuting with each other and
 studying the relations between those subsets has survived
 and becomes one of the motivations  of {\em \ro quantum\rf\ cluster algebras} (\cite{FZ02}).

 A cluster algebra is a $\Z$-subalgebra of a rational function field
given by a set of generators,
called the {\em cluster variables}. They are grouped
into overlapping subsets, called {\em clusters} and
there is a procedure called {\em mutation} which produces new clusters successively from a given  {\em
initial} cluster.
A quantum cluster algebra is a non-commutative  $q$-deformation of a cluster algebra,
and a {\em quantum cluster} is a family of  mutually $q$-commuting elements of it (\cite{BZ05}).
There are many examples of algebras which turned out to  be  (quantum) cluster algebras.
In particular,
 Gei\ss, Leclerc and
Schr{\"o}er showed that the quantum unipotent coordinate algebra
$A_q(\n(w))$
has a quantum cluster algebra structure (\cite{GLS11}).
Here $A_q(\n(w))$ is a subalgebra of the $\Q(q)$-algebra $A_q(\mathfrak n)\simeq\Um$ , which is  associated with a symmetric Kac-Moody algebra $\g$
and its Weyl group element $w$,

It is shown by Kimura \cite{Kimu12} that $ \bold B^\up \cap A_q(\n(w))$ is a basis of $A_q(\n(w))$.
 Then, in terms of quantum cluster algebras, Berenstein-Zelevinsky's ideas
   can be generalized and reformulated in the following form:
\begin{conjecture} [{\cite[Conjecture 12.9]{GLS11}, \cite[Conjecture 1.1(2)]{Kimu12}}] \label{conj:intro}
When $\g$ is of symmetric type,  every quantum cluster monomial in
$A_{q^{1/2}}(\n(w))\seteq\Q(q^{1/2})\otimes_{\Q(q)}A_q(\n(w))$ belongs to the upper global basis up to a power of $q^{1/2}$.
\end{conjecture}
There are some partial results of this conjecture. It is proved
 for $\g= A_2$, $A_3$, $A_4$ and $ A_q(\n(w))=\An$ in \cite{BZ93} and
 \cite[Section 12]{GLS05},  for $\g=A^{(1)}_1$, $A_n$ and $w$ is a square of  a Coxeter element in \cite{Lampe11} and \cite{Lampe14},  when $\g$ is symmetric and $w$ is a square of  a Coxeter element
 in \cite{KQ14}.

In this paper, we  prove the above conjecture completely by showing that there exists a
{\em  monoidal categorification of $A_{q^{1/2}}(\n(w))$} along the lines of our previous work \cite{KKKO15}.
Note that Nakajima proposed a geometric approach of this conjecture via  quiver varieties (\cite{Nak13}).

Let us briefly recall the notion of monoidal categorifications of quantum cluster algebras.
Let  $\mathcal C$  be an
abelian monoidal category equipped with
 an auto-equivalence  $q$ and
a tensor product which is compatible with a decomposition
$\shc=\soplus\nolimits_{\beta \in \rtl}\shc_\beta$.
Fix a finite index set $\K=\K_\ex \sqcup \K_\fr$ with a decomposition into
the exchangeable part and the frozen part.
Let $\Seed$ be a quadruple $(\{M_i\}_{i \in\K}, L,\wB, D)$
of a family of simple objects $\{M_i\}_{i \in\K}$
in $\mathscr C$, an  integer-valued skew-symmetric $\K \times \K$-matrix $L=(\lambda_{i,j})$,
an integer-valued $\K \times \K_\ex$-matrix $\wB = (b_{i,j})$ with skew-symmetric principal part, and a family of elements  $D=\{d_i\}_{i \in\K}$ in $\rl$.
If those data  satisfy the conditions in
Definition \ref{def:quantum monoidal seed} below, then we call it a {\em quantum monoidal seed} in $\mathcal C$.
For each $k\in\K_\ex$, we have mutations $\mu_k(L),\mu_k(\widetilde B)$ and $\mu_k(D)$ of $L,B$ and $D$, respectively.
We say that a quantum monoidal seed
$\mathscr S =(\{M_i\}_{i\in\K}, L,\widetilde B, D)$
 \emph{admits a mutation in direction $k\in\K_\ex$,} if
there exists  a simple object  $M_k' \in \shc_{\mu_k(D)_k}$
which fits into two short exact sequences \eqref{eq:intro} below
in $\mathcal C$
{\em reflecting} the mutation rule in quantum cluster algebras, and thus obtained
 quadruple $\mu_k(\Seed)\seteq(\{M_i\}_{i\neq k}\cup\{M_k'\},\mu_k(L), \mu_k(\widetilde B), \mu_k(D))$
is again a quantum monoidal seed in $\shc$.
We call $\mu_k(\Seed)$ the mutation of $\Seed$ in direction $k\in\K_\ex$.

Now the category $\shc$ is called a {\em monoidal categorification of a quantum cluster algebra $A$ over $\Z[q^{\pm1/2}]$}
if
{\rm (i)} the Grothendieck ring $\Z[q^{\pm1/2}]\tens_{\Z[q^{\pm1}]} K(\shc)$ is isomorphic to $A$,
{\rm (ii)} there exists a quantum monoidal seed
$\mathscr S =(\{M_i\}_{i\in\K}, L,\widetilde B, D)$ in $\shc$ such that
$[\mathscr S]\seteq(\{q^{m_i}[M_i]\}_{i\in\K}, L, \widetilde B)$
 is a quantum seed of $A$ for some $m_i \in \frac{1}{2}\Z$,
and {\rm (iii)} $\mathscr S$ admits successive mutations in all directions in $\K_\ex$.
Note that if $\shc$ is a monoidal categorification of $A$,  all the quantum cluster monomials in $A$ are the
classes of simple objects in $\shc$ up to a power of $q^{1/2}$.

In  the case of quantum unipotent coordinate ring $A_q(\n)$, there is a natural candidate for monoidal categorification, the category of finite-dimensional graded modules over a {\em Khovanov-Lauda-Rouquier algebras} (\cite{KL09,KL11}, \cite{R08}).
The Khovanov-Lauda-Rouquier algebras (abbreviated by KLR algebras)
 are a family of $\Z$-graded algebras
 $\{ R(\beta) \}_{\beta \in \rl^+}$ such that the Grothendieck
ring of $R \gmod \seteq \bigoplus_{\beta \in \rl^+}R(\beta)\gmod$, the direct sum
of the categories of finite-dimensional graded $R(\beta)$-modules, is
isomorphic to the integral form $A_q(\n)_{\Z[q^{\pm 1}]}$ of
$A_q(\n)$.
 Here $\rl^+$ denotes the positive root lattice of the corresponding
 symmetrizable Kac-Moody algebra $\g$.
The multiplication of $K(R \gmod)$ is given by
the {\em convolution product} of modules, and the action of $q$ is given by
the grading shift functor.
In \cite{VV09, R11},
Varagnolo--Vasserot and Rouquier
proved that the upper global basis $\bold B^\up$ of $A_q(\n)$
corresponds to the
set of the classes of all {\em self-dual} simple modules of $R
\gmod$ under the assumption that $R$ is a symmetric KLR algebra
with the base field of characteristic $0$.
For each  Weyl group element $w$, let $\mathcal C_w$
be the full subcategory of $R\gmod$
consisting of objects $M$ such that
$[M]$ belongs to $A_q(\n(w))$.
Then, $\shc_w$ is an abelian monoidal category
whose Grothendieck ring is isomorphic to $A_q(\n(w))$ and its simple objects correspond to the basis $ \bold B^\up \cap A_q(\n(w))$ of $A_q(\n(w))$.
In particular,  Conjecture \ref{conj:intro} is equivalent to saying that
when $\g$ is of symmetric type,  every quantum cluster monomial in
$\Q(q^{1/2})\otimes_{\Z[q^{\pm 1}]}K(\mathcal C_w)$ belongs to the class of self-dual simple modules up to a power of $q^{1/2}$.

In \cite{KKKO15}, to simplify the conditions of quantum monoidal seeds and their mutations, we introduce
the notion of {\em admissible pairs} in $\mathcal C_w$.
 A pair $(\{M_i\}_{i \in\K}, \widetilde B)$ is called admissible in $\mathcal C_w$ if
(i)  $\{M_i\}_{i \in\K}$ is a commuting family
of self-dual real simple  objects of $\shc_w$,
(ii) $\widetilde B$ is an integer-valued $\K \times \K_\ex$-matrix with skew-symmetric principal part,
and (iii)
 for each $k \in\K$, there exists a  self-dual  simple object $M'_k$ in $\shc_w$
such that $M_k'$ commutes with $M_i$  for all  $i\in\K\setminus\{k\}$
and there are exact sequences in $\shc_w$
\eq&&\ba{l}
 0 \to q \sodot_{b_{i,k} >0} M_i^{\snconv b_{i,k}} \to q^{\tLa(M_k,M_k')} M_k \conv M_k' \to
 \sodot_{b_{i,k} <0} M_i^{\snconv (-b_{i,k})} \to 0\\
 0 \to q \sodot_{b_{i,k} <0} M_i^{\snconv(- b_{i,k})} \to q^{\tLa(M_k',M_k)} M'_k \conv M_k
 \to \sodot_{b_{i,k} >0} M_i^{\snconv b_{i,k}} \to 0
\ea \label{eq:intro}
\eneq
 where $\tLa(M_k,M_k')$ and $\tLa(M_k',M_k)$ are prescribed integers
and $\sodot$ is a tensor product up to a power of $q$.

For an admissible pair  $(\{M_i\}_{i \in\K}, \widetilde B)$, let
$\La=(\La_{i,j})_{i,j \in\K}$
be the skew-symmetric matrix
where $\La_{i,j}$ is the homogeneous degree of $\rmat{{M_i},{M_j}}$, the \emph{r-matrix} between $M_i$ and $M_j$,
and let $D=\{d_i\}_{ i \in\K}$ be the family of elements in $\rl$ given by $M_i \in R(-d_i) \gmod$.

Then, together with the result of \cite{GLS11},
the main theorem of \cite{KKKO15} reads as follows:
\begin{Itheorem}[{Theorem 6.3 and Corollary 6.4 in \cite{KKKO15}}]
\label{th:main}
 If there exists an admissible pair  $(\{M_i\}_{i\in K},\widetilde B)$ in $\mathcal C_w$ such that
 $[\Seed]\seteq\bl\{q^{-(\wt(M_i), \wt(M_i))/4}[M_i]\}_{i\in\K},
-\La,\widetilde B, D\br$
  is an initial seed of $A_{q^{1/2}}(\n(w))$,
then $\mathcal C_w$ is a monoidal categorification of $A_{q^{1/2}}(\n(w))$.
\end{Itheorem}

This paper is mainly devoted to show that there exists an admissible pair in $\mathcal C_w$ for every symmetric Kac-Moody algebra $\g$ and its Weyl group element $w$.
In \cite{GLS11},  Gei\ss, Leclerc and
Schr{\"o}er provided an initial quantum seed in $A_q(\n(w))$
whose quantum cluster variables are
\emph{unipotent quantum minors}.
The unipotent quantum minors are elements in $A_q(\n)$, which are
 a $q$-analogue of a generalization of the minors of upper triangular matrices.
In particular, they are elements in $\bold B^\up$.
We define the \emph{determinantial module} $\M(\mu,\zeta)$ to be the simple module in
$R \gmod$ corresponding to the unipotent quantum minor $\D(\mu,\zeta)$
 under the isomorphism $A_q(\n(w))_{\Z[q^{\pm1}]} \simeq K(R \gmod)$.
Here $(\mu,\zeta)$ is a pair of elements in the weight lattice of $\g$ satisfying certain conditions.

Our main theorem is as  follows.
\begin{maintheorem}
Let $( \{ D(k,0) \}_{1 \le k \le r}, \wB, L)$ be the initial quantum seed of $A_q(\n(w))$ in \cite{GLS11} with respect to a reduced expression $\widetilde{w}=s_{i_1}\cdots s_{i_r}$ of $w$.
Let $\M(k,0)$  be the determinantial module corresponding to the unipotent quantum minor $D(k,0)$.
Then the pair
$$( \{ \M(k,0) \}_{1 \le k \le r}, \wB)$$ is admissible in $\mathcal C_w$.
\end{maintheorem}

The most essential condition for an admissible pair is that there exists the \emph{first mutation} $\M(k,0)'$ in \eqref{eq:intro} for each $k \in \K_\ex$.
To obtain it,
 we investigate the properties of determinantial modules and those of their convolution products.
Note that a unipotent quantum minor is the image of a global basis element of  the \emph{quantum coordinate ring} $A_q(\g)$ under the projection $A_q(\g) \to A_q(\n)$.
Since there exists a  bicrystal embedding from the crystal basis $B(A_q(\g))$ of $A_q(\g)$ to
the crystal basis $B(\tUqg)$ of the \emph{modified quantum groups} $\tUqg$,
 this investigation amounts to the  study on the interplays among the
 crystal and global bases of
 $A_q(\g)$, $\tUqg$ and $A_q(\n)$.

 \medskip
 This paper is organized as follows:

 In Section 1, we review on quantum groups, their integrable modules, and crystal and global bases.
 In Section 2, we review the algebras $A_q(\g)$, $\tUqg$ and $A_q(\n)$,
 and study relations among them.
 In Section 3, we study the properties of quantum minors including  T-systems and generalized T-systems.
 In Section 4, we review and study further
about the  modules over Khovanov-Lauda-Rouquier algebras
along the lines of \cite{KKKO14, KKKO15}.
 In Section 5, we review the quantum cluster algebras and their monoidal categorifications by symmetric KLR algebras.
 In the last section, we establish our main theorem.

\bigskip

\noindent
{\bf Acknowledgements.} The authors would like to express their gratitude to
Bernard Leclerc for many fruitful discussions.
The last two authors gratefully acknowledge the hospitality of
Research Institute for Mathematical Sciences, Kyoto University
during their visits in 2014.

\section{Preliminaries}
This is a continuation of \cite{KKKO15}, and
we sometimes omit the definitions and notations employed there.
We refer the reader to loc.\ cit.
In this section, we briefly recall the crystal and global bases theory for
$U_q(\g)$. 
We refer to \cite{Kash91,Kas93,Kas95} for materials in this section.

\subsection{The quantum enveloping algebras  and their integrable modules}
Let $\g$ be a Kac-Moody algebra. We
denote by $I$ the index set which parametrizes the set of simple roots $\Pi=\{ \alpha_i \ | \ i \in I \}$
and the set of simple coroots $\Pi^\vee=\{ h_i \ | \ i \in I \}$.
We also denote by $\wl$ the weight lattice, by
$\wl^\vee \seteq \Hom_\Z(P,\Z)$ the co-weight lattice,
 by $\cmA=\big( \lan h_i,\alpha_j \ra \big)_{i,j \in I}$
the generalized Cartan matrix of $\g$,
by $W$ the Weyl group of $\g$, and
by $( \ , \ )$ a $W$-invariant symmetric bilinear form on $\wl$.
The free abelian
group $\rl \seteq \bigoplus_{i \in I}\Z \alpha_i$ is called the {\em root lattice}. Set $\rl^{+}= \sum_{i \in I} \Z_{\ge 0}
\alpha_i\subset\rl$ and $\rl^{-}= \sum_{i \in I} \Z_{\le 0}
\alpha_i\subset\rl$. For $\beta=\sum_{i\in I}m_i\alpha_i\in\rl$,
we set
$|\beta|=\sum_{i\in I}|m_i|$.

Let $U_q(\g)$ be the quantum enveloping algebra of $\g$, which is the $\Q(q)$-algebra generated by $e_i$, $f_i$ ($i \in I$) and $q^h$ ($h \in \wl^\vee$)
with certain defining relations (see e.g.\ \cite[Definition 1.1]{KKKO15}).
We set $t_i=q^{\frac{( \al_i, \al_i )}{2}h_i}$.
We denote by $U_q^-(\g)$ (resp.\ $U_q^+(\g)$) the subalgebra of $U_q(\g)$ generated by $f_i$ ($i \in I$) (resp.\ $e_i$ ($i \in I$)).

We define the divided powers by
$$e_i^{(n)} = e_i^n / [n]_i!, \quad f_i^{(n)} =
f_i^n / [n]_i! \ \ (n \in \Z_{\ge 0}),$$
where $q_i=q^{(\alpha_i,\alpha_i)/2}$,
$[n]_i=\dfrac{ q_i^n - q_i^{-n} }{ q_i - q_i^{-1} }$ and $[n]_i! = \prod^{n}_{k=1} [k]_i$  for  $n \in \Z_{\ge 0}$, $i \in I$.
Let us denote by $U_q(\g)_\A$ the $\A$-subalgebra of $U_q(\g)$ generated by
$e_i^{(n)}$, $f_i^{(n)}$ ($i\in I$, $n \in \Z_{\ge 0}$) and $q^h$ ($h \in \wl^\vee$),
by $U_q^{-}(\g)_\A$ the
$\A$-subalgebra of $U^-_q(\g)$ generated by $f_i^{(n)}$
($i\in I$, $n \in \Z_{\ge 0}$), and
by $U_q^{+}(\g)_\A$ the
$\A$-subalgebra of $U^+_q(\g)$ generated by $e_i^{(n)}$
($i\in I$, $n \in \Z_{\ge 0}$).

We denote by $\Oint(\g)$ the category of integrable left $U_q(\g)$-modules $M$ satisfying the following conditions: {\rm (i)} $M= \bigoplus_{\eta \in \wl} M_\eta$
where $M_\eta \seteq \{  m \in M \ | \ q^hm=q^{\lan h,\eta \ra} m \}$, {\rm (ii)} $\dim M_\eta < \infty$, and {\rm (iii)}
there exist finitely many weights $\la_1$, \ldots, $\la_m$ such that
$\wt(M)\subset\cup_{j}(\la_j+\rl^-)$. The category $\Oint(\g)$ is semisimple with its simple objects being isomorphic to the
highest weight modules $V(\lambda)$ with highest weight vector $u_\lambda$ of
highest weight $\lambda \in \wl^+$, the set of dominant integral weight.

Let us recall the $\Q(q)$-antiautomorphism $\varphi$ and $^*$ of $U_q(\g)$ given as follows:
\begin{align}
&\varphi(e_i)=f_i, \quad \varphi(f_i)=e_i, \quad \varphi(q^h)=q^h, \label{eq: antiauto v} \\
&e_i^*=e_i, \quad\quad \ f_i^*=f_i, \quad\quad \ (q^h)^*=q^{-h}. \label{eq: antiauto s}
\end{align}

Let us recall also the $\Q$-automorphism $\ol{\phantom{a}}$ of $U_q(\g)$ given by
\begin{align}
& \overline{e}_i=e_i, \quad \overline{f}_i=f_i, \quad \overline{q}^h=q^{-h}, \quad \overline{q}=q^{-1}. \label{eq: auto bar}
\end{align}

For $M \in \Oint(\g)$, let us denote by $\Dv M$ the left $U_q(\g)$-module $\bigoplus_{\eta \in \wl} \Hom(M_\eta,\Q(q))$ with the action of $U_q(\g)$ given by:
$$  (a\psi)(m)=\psi(\vp(a)m)\quad\text{for $\psi\in\Dv M$,
$m\in M$ and $a\in U_q(\g)$.}  $$
Then $\Dv M$ belongs to $\Oint(\g)$.

For a left $U_q(\g)$-module $M$, we denote by $M^\ri$ the right $U_q(\g)$-module
$\{m^\ri \ | \ m\in M \}$ with the right action of $U_q(\g)$ given by
$$\text{$(m^\ri)\, x=(\vp(x)m)^\ri$ for $m\in M$ and $x\in U_q(\g)$.}$$
We denote by $\Oint^\ri(\g)$ the category of right integrable $U_q(\g)$-modules $M^\ri$ such that
$M \in \Oint(\g)$.

There are two comultiplications $\comp$ and $\comm$ on $U_q(\g)$ defined as follows:
\begin{align}
&\comp(e_i)=e_i \tens 1 + t_i \tens e_i,\quad \comp(f_i)=f_i \tens t_i^{-1} + 1 \tens f_i,\quad \comp(q^h)=q^h \tens q^h, \label{eq: comp}  \\
&\comm(e_i)=e_i \tens t_i^{-1} + 1 \tens e_i,\quad \comm(f_i)=f_i \tens 1 + t_i \tens f_i,\quad \comm(q^h)=q^h \tens q^h. \label{eq: comm}
\end{align}

For two $U_q(\g)$-modules $M_1$ and $M_2$, let us denote by $M_1
\tensp M_2$ and $M_1 \tensm M_2$ the vector space $M_1 \tens_{\Q(q)}M_2$ endowed with $U_q(\g)$-module
structure induced by the comultiplications $\comp$ and $\comm$, respectively. Then we have
$$ \Dv(M_1 \otimes_{\pm} M_2 ) \simeq (\Dv M_1) \otimes_{\mp} (\Dv M_2).$$

The simple $U_q(\g)$-module $V(\lambda)$ and the
$B_q(\g)$-module $U_q^-(\g)$)
have a unique non-degenerate  symmetric bilinear form
$( \ , \ )$ such that
\begin{align}
& (u_\lambda,u_\lambda) = 1 \text{ and } (xu,v)=(u,\vp(x)v) \text{ for } u,v \in V(\lambda) \text{ and } x \in U_q(\g), \\
& \ ( \mathbf{1} ,\mathbf{1}) = 1  \text{ and } (xu,v)=(u,\vp(x)v) \text{ for } u,v \in U^-_q(\g) \text{ and } x \in B_q(\g).
\end{align}
Recall that $B_q(\g)$ is the quantum boson algebra generated by $e_i'$ and $f_i$, and $\vp$ is the anti-automorphism of $B_q(\g)$ sending $e_i'$ to $f_i$ and
$f_i$ to $e_i'$.

Note that $( \ , \ )$ induces the non-degenerated bilinear form
$$  \lan \cdot , \cdot \ra \cl  V(\lambda)^\ri \otimes_{U_q(\g)} V(\lambda) \to \Q(q)$$
given by $ \lan u^\ri, v \ra  = (u,v)$,
by which $\Dv V(\la)$ is canonically isomorphic to $V(\la)$.

\subsection{Crystal bases and global bases}
For a  subring $A$ of $\Q(q)$, we say that $L$ is an $A$-{\em lattice} of
a $\Q(q)$-vector space $V$ if
$L$ is a free $A$-submodule of $V$ such that $V=\Q(q)\tens_AL$.

Let us denote by $\QA_0$ (resp.\ $\QA_\infty$) the ring of rational
functions in $\Q(q)$ which are regular at $q=0$ (resp.\ $q=\infty$).
We set $\QA\seteq\QA_0\cap\QA_\infty=\Q[q^{\pm1}]$.

\medskip

Let $M$ be a $\U$-module in $\Oint(\g)$.
Then any $u\in M$ can be uniquely written as
$$u=\sum_{n=0}^{\infty}f_i^{(n)}u_n\quad\text{with $e_iu_n=0$.}$$
We define the {\em lower Kashiwara operators} by
\eqn&&\te^\low_i(u)=\sum_{n=1}^{\infty}f_i^{(n-1)}u_n
\quad\text{and}\quad
\tf^\low_i(u)=\sum_{n=0}^{\infty}f_i^{(n+1)}u_n,
\eneqn
and the {\em upper Kashiwara operators}
by
\eqn&&\te^\up_i(u)=\te_i^{\low}q_i^{-1}t_iu\quad\text{and}\quad
\tf^\up_i(u)=\tf_i^{\low}q_i^{-1}t_i^{-1} u.
\eneqn
We say that an $\QA_0$-lattice of $M$
 is a lower (resp.\ upper) crystal lattice of $M$
if $L$ is $\wl$-graded and invariant by the lower (resp.\ upper) Kashiwara operators.

\Lemma\label{lem:lowup} Let $L$ be a lower crystal lattice of $M\in \Oint(\g)$.
Then we have
\bnum
\item $\soplus\nolimits_{\la\in\wl}q^{-(\la,\la)/2}L_\la$ is an upper crystal lattice of $M$.
\item $L^\vee\seteq\set{\psi\in\Dv M}{\lan\psi,L\ran\in\QA_0}$ is an upper crystal lattice of
$\Dv M$.
\ee
\enlemma
\Proof
(i) Let $\phi_M$ be the endomorphism of $M$ given by $\phi_M\vert_{M_\la}=q^{-(\la,\la)/2}\id_{M_\la}$. Then we have
$\te_i^\up=\phi_M\circ\te_i^\low\phi_M^{-1}$ and
$\tf_i^\up=\phi_M\circ\tf_i^\low\phi_M^{-1}$.

\smallskip
\noi
(ii) follows from the fact that
$\te_i^\up$ and $\tf_i^\up$ are the adjoint operators of
$\tf_i^\low$ and $\te_i^\low$, respectively.
\QED

\begin{definition} A {\em lower} \ro resp.\ {\em upper}\rf\
{\em crystal basis} of $M$ consists of a pair $(L,B)$
satisfying the following conditions:
\bnum
\item $L$ is a lower \ro resp. upper\rf\ crystal lattice of $M$,
\item $B= \sqcup_\eta B_\eta$ is a basis of the $\Q$-vector space $L/qL$, where $B_\eta=B \cap (L_\eta/qL_\eta)$,
\item the induced maps $\tilde{e}_i$ and $\tilde{f}_i$ on $L/qL$ satisfy
$$ \tilde{e}_iB, \tilde{f}_iB \subset B \sqcup \{0\}, \text{ and }  \tilde{f}_ib=b' \text{ if and only if } b=\tilde{e}_ib' \text{ for } b,b' \in B.$$
Here $\te_i$ and $\tf_i$ denote the lower (resp.\ upper) Kashiwara operators.
\end{enumerate}
\end{definition}

It is shown in \cite{Kash91} that $V(\lambda)$
has {\em the lower crystal basis} $(L^\low(\lambda),B^\low(\lambda))$. Using
the non-degenerate  symmetric bilinear form $( \ , \ )$, $V(\lambda)$ has
{\em the upper crystal basis}
$(L^\up(\lambda),B^\up(\lambda))$ where
$$ L^\up(\lambda) \seteq \{ u \in V(\lambda) \
| \ (u,L^{\low} (\lambda)) \subset \QA_0 \},$$
and $B^\up(\lambda)\subset L^\up(\lambda)/qL^\up(\lambda)$ is the dual basis of $B^\low(\lambda)$ with respect to
the induced non-degenerate pairing between
$L^\up(\lambda)/qL^\up(\lambda)$ and $L^\low(\lambda)/qL^\low(\lambda)$.
An (abstract) {\em crystal} is a set $B$ together with maps
$$ \wt\cl B \to \wl, \ \ \ve_i,\vp_i\cl  B \to \Z \sqcup \{ \infty \} \text{ and } \te_i,\tf_i\cl B \to B \sqcup \{ 0 \} \text{ for } i \in I,$$
such that
\begin{itemize}
\item[{\rm (C1)}] $\vp_i(b)=\ve_i(b)+\lan h_i,\wt(b) \ra$ for any $i$,
\item[{\rm (C2)}] if $b \in B$ satisfies $\te_i(b) \ne 0$, then
$$ \ve_i(\te_ib)=\ve_i(b)-1, \ \vp_i(\te_ib)=\vp_i(b)+1, \ \wt(\te_i b)=\wt(b)+\alpha_i, $$
\item[{\rm (C3)}] if $b \in B$ satisfies $\tf_i(b) \ne 0$, then
$$ \ve_i(\tf_ib)=\ve_i(b)+1, \ \vp_i(\tf_ib)=\vp_i(b)-1, \ \wt(\tf_i b)=\wt(b)-\alpha_i, $$
\item[{\rm (C4)}] for $b,b' \in B$, $b'=\tf_i b$ if and only if $b=\te_i b'$,
\item[{\rm (C5)}] if $\vp_i(b)=-\infty$, then $\te_ib=\tf_ib=0$.
\end{itemize}

Recall that, with the notions of {\em morphisms} and {\em tensor product rule} of crystals, the category of crystal becomes a monoidal category
(\cite{Kash94}). If $(L,B)$ is a crystal of
$M$, then $B$ is an abstract crystal.
 Since $B^\low(\lambda) \simeq B^\up(\lambda)$,
we drop the superscripts for simplicity.

Let $V$ be a $\Q(q)$-vector space, and let $L_0$ be an $\QA_0$-lattice of $V$,
$L_\infty$ an $\QA_\infty$-lattice of $V$ and $V_{\QA}$ an
$\QA$-lattice of $V$.
We say that the triple $(V_{\QA},L_0,L_\infty)$ is {\em balanced} if the following canonical map is a $\Q$-linear isomorphism:
$$ E\seteq V_{\QA} \cap L_0 \cap L_\infty \To L_0/qL_0.$$
The inverse of the above isomorphism
$G\cl  L_0/qL_0\isoto E$ is called the
{\em globalizing map}.
If $(V_{\QA},L_0,L_\infty)$ is balanced, then we have
\eqn &&\text{$\Q(q)\tens_\Q E\isoto V$,
$\QA\tens _\Q E\isoto V_{\QA}$, $\QA_0\tens _\Q E\isoto L_0$
and  $\QA_\infty\tens _\Q E\isoto L_\infty$.}
\eneqn
Hence, if $B$ is a basis of $L_0/qL_0$, then $G(B)$ is a basis of $V$,
$V_{\QA}$, $L_0$ and $L_\infty$. We call $G(B)$ a {\em global basis}.

\medskip
We define the two $\A$-forms of $V(\la)$ by
\eqn
&&V^\low(\lambda)_\A\seteq \Um_\A u_\lambda\quad\text{and}\\
&&V^\up(\lambda)_\A\seteq
\set{ u \in V(\lambda)}{\bl u, V^\low(\lambda)_\A\br \subset \A }.\eneqn
Then $\bl\Q\tens V^\low(\lambda)_\A,\; L^\low(\lambda),\; \ol{  L^\low(\lambda)}\br$ and
$\bl\Q\tens V^\up(\lambda)_\A,\; L^\up(\lambda),\; \ol{  L^\up(\lambda)}\br$
are balanced.
Let us denote by $G^\low_\la$ and $G^\up_\la$ the associated globalizing maps, respectively.
Then the sets
$$\B^\low(\lambda) \seteq  \{G^\low_\lambda(b) \ | \ b \in B^\low(\lambda) \} \ \  \text{ and } \ \ \B^\up(\lambda) \seteq  \{G^\up_\lambda(b) \ | \ b \in B^\up(\lambda) \}$$
form $\A$-bases of $V^\low(\lambda)_\A$ and $V^\up(\lambda)_\A$, respectively.
They are called the {\em lower global basis} and the {\em upper global basis} of $V(\lambda)$.

Similarly, there exists a crystal basis $(L(\infty),B(\infty))$ of the simple $B_q(\g)$-module $U_q^-(\g)$ such that
$(\Q\tens \Um_\A, L(\infty),\overline{L(\infty)})$ is balanced.
Let us denote the globalizing map by $G^{\low}$.
Then the set
$$\B^\low(\Um)\seteq  \{ G^\low(b )\ | \ b \in B(\infty) \}$$
forms an $\A$-basis of $\Um_\A$ and is called the {\em lower global basis} of $U_q^-(\g)$.

Let us denote by
$$\B^\up(\Um) \seteq  \{G^\up(b) \ | \ b \in B(\infty) \}$$
the dual basis of $\B^\low(\Um)$ with respect to $( \ , \ )$.
Then it is
a $\A$-basis of $$\Um_\A^\vee \seteq \{ x \in U^-_q(\g) \ | \ (x,\Um_\A) \subset \A \}$$ and  called the {\em upper global basis} of $U_q^-(\g)$.

\section{Quantum coordinate rings and modified quantized enveloping algebras}

\subsection{Quantum coordinate ring} Let $U_q(\g)^*$ be $\Hom_{\Q(q)}(U_q(\g),\Q(q))$. Then the comultiplication $\comp$ induces the
 multiplication $\mu$ on $U_q(\g)^*$ as follows:
\begin{align*}
 \mu \cl  & U_q(\g)^* \tens U_q(\g)^* \to (\U\tens\U)^*\To[\;(\comp)^*\;]
U_q(\g)^* .
\end{align*}
Later on, it will be convenient to use Sweedler's notation
$\comp(x)=x_{(1)} \tens x_{(2)}$.
With this notation,
$$\text{$\bl fg\br(x)=f(x_{(1)})\,g(x_{(2)})$\quad
for $f,g \in U_q(\g)^*$ and  $x\in\U$.}$$

The $U_q(\g)$-bimodule structure on $U_q(\g)$ induces
a $U_q(\g)$-bimodule structure on $U_q(\g)^*$. Namely,
\begin{align*}
(x \cdot f) (v) =f(vx) \quad \text{ and } \quad (f \cdot x) (v)
=f(xv) \quad\text{ for } f \in U_q(\g)^* \text{ and } x,v \in U_q(\g).
\end{align*}

Then the multiplication $\mu$ is a morphism of a
$U_q(\g)$-bimodule, where $U_q(\g)^* \tens U_q(\g)^*$ has the
structure of a $U_q(\g)$-bimodule via $\comp$: for $f,g \in U_q(\g)^*$ and  $x,y \in U_q(\g)$,
$$ x(fg)y=(x_{(1)}fy_{(1)})(x_{(2)}gy_{(2)}),$$
where $\comp(x)=x_{(1)}\tens x_{(2)}$ and $\comp(y)=y_{(1)}\tens y_{(2)}$.

\begin{definition}
We define the {\em quantum coordinate ring} $A_q(\g)$ as follows:
$$A_q(\g) = \{ u \in U_q(\g)^* \ | \ \text{ $U_q(\g)u$ belongs to $\Oint(\g)$ and $u U_q(\g)$ belongs to $\Oint^\ri(\g)$}\}. $$
\end{definition}

Then, $A_q(\g)$ is a
subring of $U_q(\g)^*$
because {\rm (i)} $\mu$ is $U_q(\g)$-bilinear, {\rm (ii)}  $\Oint(\g)$ and
$\Oint^\ri(\g)$ are closed under the tensor product.

We have the weight decomposition:
$A_q(\g) = \soplus_{\eta,\zeta \in \wl} A_q(\g)_{\eta,\zeta}$ where
$$A_q(\g)_{\eta,\zeta} \seteq \{  \psi \in A_q(\g) \ | \  q^{h_\li} \cdot \psi \cdot q^{h_\ri} = q^{\langle h_\li,\eta \rangle +\langle h_\ri,\zeta \rangle  } \psi
\text{ for } h_\li,h_\ri \in \wl^\vee \},$$
For $\psi \in A_q(\g)_{\eta,\zeta}$, we write
$$ \wt_\li(\psi)=\eta \quad \text{ and } \quad \wt_\ri(\psi)=\zeta.$$

\medskip
For any $V\in\Oint$, we have the $U_q(\g)$-bilinear homomorphism
$$\Phi_V\colon V\tens (\Dv V)^\ri\to A_q(\g)$$
given by
$$\Phi_V(v\tens \psi^\ri)(a)=\lan \psi^\ri, av\ra=\lan \psi^\ri a, v\ra \quad\text{for $v\in V$,
$\psi\in\Dv V$ and $a\in U_q(\g)$.}$$

\begin{proposition} [{\cite[Proposition 7.2.2]{Kas93}}] We have an isomorphism $\Phi$ of $\Uqg$-bimodules
\begin{equation} \label{eq: Phi}
\Phi\cl  \soplus_{\lambda \in \wl^+} V(\lambda) \tens_{\Q(q)}
V(\lambda)^\ri \overset{\sim}{\longrightarrow} A_q(\g)
\end{equation}
given by $\Phi|_{V(\lambda) \tens_{\Q(q)} V(\lambda)^\ri} = \Phi_\lambda \seteq \Phi_{V(\lambda)}:$ Namely,
$$ \Phi(u \tens v^\ri)(x)= \lan v^\ri,xu \ra = \lan v^\ri x,u \ra =(v,xu) \text{ for any } v,u \in V(\lambda) \text{ and } x \in U_q(\g).$$
\end{proposition}

We introduce the
 crystal basis $\big(L^\up(A_q(\g)), B(A_q(\g)) \big)$ of $\Ag$ as the images by $\Phi$ of
$$ \soplus_{\lambda \in \wl^+} L^\up(\lambda) \otimes L^\up(\lambda)^\ri \text{ and }
\bigsqcup_{\lambda \in \wl^+} B(\lambda) \otimes B(\lambda)^\ri.$$
Hence it is a crystal base with respect to the left action of
$\U$ and also the right action of $\U$.
We sometimes write by
$e_i^*$ and $f_i^*$ the operators  of $\Ag$ obtained by the right actions of
$e_i$ and $f_i$.

\medskip
We define the $\A$-form of $\Ag$ by
$$\Ag_\A\seteq\set{\psi\in\Ag}{\lan \psi,\, \U_\A\ran\subset\A}.$$
We define the bar-involution $-$ of $A_q(\g)$ by
$$ \overline{\psi}(x) = \overline{\psi(\overline{x})} \quad \text{ for } \psi \in A_q(\g), \  x \in U_q(\g).$$
Note that the bar-involution is not a ring homomorphism but it satisfies
\eqn
&&\ol{\psi\;\theta\;}=q^{(\wt_\lt(\psi),\wt_\lt(\theta))
-(\wt_\rt(\psi),\wt_\rt(\theta))}\;\ol{\theta}\;\ol{\psi}\quad
\text{for any $\psi$, $\theta\in\Ag$.}
\eneqn
Since we do not use this formula and
it is proved similarly to Proposition~\ref{prop:phimul} below,
we omit its proof.

The triple $\big(\Q\tens \Ag_\A,\,L^\up(A_q(\g)),\,\overline{L^\up(A_q(\g))}\big)$ is balanced (\cite[Theorem 1]{Kas93}), and hence there exists an upper global basis of $\Ag$
$$ \B^\up(A_q(\g))\seteq \{ G^\up(b) \ | \ b \in B^\up(A_q(\g)) \}.$$

For $\la\in\Pd$ and $\mu\in W\la$, we denote by $u_{\mu}$
a unique member of the upper global basis
of $V(\la)$ with weight $\mu$. It is also a member of the lower global basis.

\begin{proposition}\label{prop:Agglobal}
Let $\lambda\in\wl^+$, $w\in W$ and $b\in B(\lambda)$.
Then, $\Phi(G^\up(b)\tens u_{w\lambda}^\ri)$ is a member of
the upper global basis of $A_q(\g)$.
\end{proposition}
\begin{proof}
The element $\psi\seteq\Phi(G^\up(b)\tens u_{w\lambda}^\ri)$ is
bar-invariant, and a member of crystal basis modulo $L^\up(A_q(\g))$.
For any $P\in U_q(\g)_\A$,
$$\lan\psi, P\ra=\big( u_{w\lambda}, PG^\up(b)\big)$$
belongs to $\A$ because $PG^\up(b)\in V^\up(\lambda)_\A$ and
$u_{w\lambda}\in V^\low(\lambda)_\A$.
Hence $\psi$ belongs to $\Ag_\A$.
\end{proof}

\medskip

The $\Q(q)$-algebra anti-automorphism $\vphi$ of $\U$ induces
a $\Q(q)$-linear automorphism $\vphi^*$
of $\Ag$
by
$$\bl\vphi^*(\psi)\br(x)=\psi\bl\vphi(x)\br\quad\text{for any $x\in \U$.}$$
We have
$$\vphi^*\bl\Phi(u\tens v^\rt)\br=\Phi(v\tens u^\rt),$$
and
$$\wtl(\vphi^*\psi)=\wtr(\psi)\qtext \wtr(\vphi^*\psi)=\wtl(\psi).$$
It is obvious that
$\vphi^*$ preserves
$\Ag_\A$, $L^\up(\Ag)$ and $ \B^\up(A_q(\g))$.
\Prop\label{prop:phimul}
$$\vphi^*(\psi\theta)=
q^{(\wtr(\psi),\wtr(\theta))-(\wtl(\psi),\wtl(\theta))}
(\vphi^*\psi)(\vphi^*\theta).$$
\enprop
In order to prove this proposition, we prepare
a sublemma.

Let $\xi$ be the algebra automorphism of $\Uq$ given by
$$\xi(e_i)=q_i^{-1}t_ie_i, \quad\xi(f_i)=q_if_it_i^{-1},\quad \xi(q^h)=q^h.$$
We can easily see
\eq
&&(\xi\tens\xi)\circ\Cmp=\Cmm\circ\xi.
\label{eq:Dexi}
\eneq

Let $\xi^*$ be the automorphism of $\Ag$ given by
$$(\xi^*\psi)(x)=\psi(\xi(x))\quad\text{for $\psi\in\Ag$ and $x\in\Uq$.}$$

\Sub\label{sub:xi}
We have
\eq
&&\xi^*(\psi)=q^{A(\wt_\lt(\psi),\wt_\rt(\psi))}\psi.\label{eq:xistar}
\eneq
Here $A(\la,\mu)=\dfrac{1}{2}\bl(\mu,\mu)-(\la,\la)\br$.
\ensub
\Proof 
Let us show that, for each $x$, the following equality
\eq
&&\psi(\xi(x))=q^{A(\wt_\lt(\psi),\wt_\rt(\psi))}\psi(x)\label{eq:xi}
\eneq
holds for any $\psi$.

The equality \eqref{eq:xi} is obviously true for $x=q^h$.
If \eqref{eq:xi} is true for $x$, then
\eqn
&&\xi^*(\psi)(xe_i)=\psi\bl\xi(xe_i)\br=\psi\bl\xi(x)e_it_i)q_i\\
&&=q^{(\al_i,\wt_\lt(\psi))+(\al_i ,\al_i)/2}\psi(\xi(x)e_i)\\
&&=q^{(\al_i,\wt_\lt(\psi))+(\al_i , \al_i)/2}\bl\xi^*(e_i\psi)\br(x)\\
&&=q^{(\al_i,\wt_\lt(\psi))+(\al_i , \al_i)/2+A(\wt_\lt(\psi)+\al_i,\wt_\rt(\psi))}
(e_i\psi)(x).
\eneqn
Since $\Vert\la+\al_i\Vert^2
=\Vert\la\Vert^2+2(\al_i,\la)+(\al_i, \al_i)$,
\eqref{eq:xi} holds for $xe_i$.
Similarly if
\eqref{eq:xi} holds for $x$,
then it holds for $xf_i$.
\QED

\Proof[{Proof of Proposition~\ref{prop:phimul}}]
We have
\eq
&&(\vphi\circ\vphi)\circ \Cmm=\Cmp\circ \vphi.
\label{eq:phiC}
\eneq
Hence, we have
\begin{align*}
\lan \vphi^*(\psi\theta), x\ran
&=\lan\psi\theta,\vphi(x)\ran\\
&=\lan\psi\tens\theta,\Cmp(\vphi(x))\ran\\
&=\lan\psi\tens\theta,(\vphi\tens\vphi)\circ\Cmm(x)\ran\\
&=\lan \vphi^*(\psi)\tens\vphi^*(\theta),\;\Cmm(x)\ran.
\end{align*}
Hence we have
\begin{align*}
\lan \xi^*(\vphi^*(\psi\theta)),x\ran
&=\lan\vphi^*(\psi\theta), \xi (x)\ran
=\lan \vphi^*(\psi)\tens\vphi^*(\theta),\Cmm(\xi(x))\ran\\
&=\lan\vphi^*(\psi)\tens\vphi^*(\theta),(\xi\tens\xi)\circ\Cmp x \ran\\
&=\lan \xi^*\vphi^*(\psi)\tens\xi^*\vphi^*(\theta),\Cmp x \ran\\
&=\lan\bl\xi^*\vphi^*(\psi)\br\bl\xi^*\vphi^*(\theta)\br,x\ran\\
&=q^{A(\wtr(\psi),\wtl(\psi))+A(\wtr(\theta),\wtl(\theta))}
\lan  (\vphi^*\psi)\,(\vphi^*\theta) ,\,x\ran.
\end{align*}
Hence we obtain
$$\vphi^*(\psi\theta)=q^c(\vphi^*\psi)\,(\vphi^*\theta)$$
with
\eqn
&&c=A(\wtr(\psi),\wtl(\psi))+A(\wtr(\theta),\wtl(\theta))-
A(\wtr(\psi)+\wtr(\theta),\wtl(\psi)+\wtl(\theta))\\
&&\hs{2ex}=(\wtr(\psi),\wtr(\theta))-(\wtl(\psi),\wtl(\theta)).
\eneqn
\QED

\subsection{Unipotent quantum coordinate ring}
Let us endow $U_q^+(\g) \tens U_q^+(\g)$ with the algebra structure defined by
$$(x_1 \tens x_2) \cdot (y_1 \tens y_2) = q^{-(\wt(x_2),\wt(y_1))}(x_1y_1 \tens x_2y_2).$$
Let $\Delta_\n$ be the algebra homomorphism $U_q^+(\g) \to U_q^+(\g)
\tens U_q^+(\g)$ given by
$$ \Delta_\n(e_i) = e_i \tens 1 + 1 \tens e_i.$$

Set
$$ A_q(\n) = \soplus_{\beta \in \rl^-} A_q(\n)_\beta \quad \text{ where } A_q(\n)_\beta \seteq (U^+_q(\g)_{-\beta})^*.$$

Defining the bilinear form $\langle \ \cdot \ , \ \cdot \ \rangle\cl
(A_q(\n) \tens A_q(\n)) \times (U^+_q(\g) \tens U^+_q(\g))$ by
$$ \langle \psi \tens \theta, x \tens y \rangle =\theta(x) \psi(y),$$
we define the algebra structure on $A_q(\n)$ by
$$(\psi \cdot \theta)(x) = \langle \psi \tens \theta, \Delta_\n(x) \rangle
= \theta(x_{(1)})\psi(x_{(2)})$$
where $\Delta_\n(x)=x_{(1)} \tens x_{(2)}$.

Since $U_q^+(\g)$ has a $U_q^+(\g)$-bimodule structure, so does
$A_q(\n)$.

We define the $\A$-form of $\An$ by
\eq
&&\An_\A=\set{\psi\in\An}{\psi\bl\Up_\A\br\subset\A}.
\eneq

We define the bar-operator $-$ on $\An$ by

$$\ol{\psi}(x)=\ol{\psi(\ol{x})} \quad\text{for $\psi\in\An$ and $x\in\Up$.}
$$

Note that the bar-operator is not a ring homomorphism but it satisfies
\eqn
&&\ol{\psi\;\theta\;}=q^{(\wt(\psi),\wt(\theta))}\;\ol{\theta}\;\ol{\psi}\quad
\text{for any $\psi$, $\theta\in\An$.}
\eneqn

For $i\in I$, we denote by $e_i^*$ the right action of
$e_i$ on $\An$.

\begin{lemma} For $u,v \in A_q(\n)$, we have  $q$-boson relations
\begin{align*}
e_i(uv) = (e_iu)v + q^{(\alpha_i,\wt(u))} u (e_iv) \  \text{ and }
 \ e_i^*(uv) = u(e_i^*v) + q^{(\alpha_i,\wt(v))} (e_i^*u)v.
\end{align*}
\end{lemma}

\begin{proof}
$$
\lan e_i(uv), x\ra=\lan uv, xe_i\ra= \lan u\tens v,\; \Delta_\n(xe_i)\ra.$$
If we set $\Delta_\n x=x_{(1)}\tens x_{(2)}$,
then we have
$$\Delta_\n(xe_i)=(x_{(1)}\tens x_{(2)})(e_i\tens 1+1\tens e_i)
=q^{-(\alpha_i,\wt(x_{(2)}))}(x_{(1)}e_i)\tens x_{(2)}+x_{(1)}\tens (x_{(2)}e_i).$$
Hence, we have
\begin{align*}
&\lan u\tens v,\; \Delta_\n(xe_i)\ra
=q^{-(\alpha_i,\wt(x_{(2)}))}u(x_{(2)})v(x_{(1)}e_i)+u(x_{(2)}e_i)v(x_{(1)})\\
&\hspace{5ex}=q^{(\alpha_i,\wt(u))}u(x_{(2)})\cdot (e_iv)(x_{(1)})+(e_iu)(x_{(2)})
\cdot v(x_{(1)})\\
&\hspace{5ex}=\lan q^{(\alpha_i,\wt(u))}u\tens(e_iv)+(e_iu)\tens v,\,\Delta_\n x\ra.
\end{align*}
The second  identity follows in a similar way.
\end{proof}

We define the map $\iota \cl  U_q^-(\g) \to A_q(\n)$ by
$$  \langle \iota(u),x \rangle = (u,\vp(x)) \quad \text{ for any } u \in U_q^-(\g) \text{ and } x \in U_q^+(\g).$$
 Since $(\ ,\ )$ is a non-degenerate bilinear form on $\Um$,
$\iota$ is injective.
The relation $$\langle \iota(e_i'u), x \rangle = (e_i'u,\vp(x)) =
(u,f_i\vp(x))  = \langle e_i \iota(u),x
\rangle,$$ implies that
$$\iota(e_i'u) = e_i \iota(u).$$

\begin{lemma} $\iota$ is an algebra isomorphism.
\end{lemma}

\begin{proof}
The map $\iota$ is an algebra homomorphism because $e_i'$ and $e_i$
both satisfy the same $q$-boson relation.
\end{proof}

Hence, the algebra $A_q(\n)$ has an upper crystal basis $(L^\up(A_q(\n)),B(A_q(\n)))$ such that
$B(A_q(\n)) \simeq B(\infty)$. Furthermore, $A_q(\n)$ has an upper global basis $$\B^\up(A_q(\n))=\{ G^\up(b) \}_{ \ b \in B(A_q(\n))}$$
induced by the balanced triple
$\bl \Q\tens\An_\A,\,L^\up(A_q(\n)),\,\overline{L^\up(A_q(\n))}\br$.

\begin{remark} 
Note that the multiplication on $\An$ given in \cite{GLS11} is different from ours.
Indeed, by denoting
the product of $\psi$ and $\phi$  in \cite[\S4.2]{GLS11} by $\psi \cdot \phi$,
we have, for $x \in U_q^+(\g)$
\eqn
(\psi \cdot \phi) (x) = \psi(x^{(1)}) \phi(x^{(2)}),
\eneqn
where $\Delta_+(x)=x^{(1)}q^{h_{(1)}} \otimes x^{(2)}q^{h_{(2)}}$  for $x^{(1)}, x^{(2)} \in U_q^{+}(\g),  \ h_{(1)}, h_{(2)} \in \wl^\vee$.
By Lemma \ref{lem:Deltan} below, we have
\eqn
(\psi \cdot \phi) (x) &&= \psi\bl q^{(\wt(x_{(1)}),\wt(x_{(2)}))} (x_{(2)}) \br \phi \bl x_{(1)} \br \\
&&=q^{(\wt(x_{(1)},\wt(x_{(2)}))} \psi(x_{(2)}) \phi(x_{(1)})
= q^{(\wt(\psi),\wt(\phi))} (\psi \phi)(x),
\eneqn
for $x \in U_q^+(\g)$, where $\Delta_\n(x)=x_{(1)} \tens x_{(2)}$.
In particular, we have a $\Q(q)$-algebra isomorphism
from $(\An, \cdot)$ to $\An$ given by
\eq \label{eq:isoAqn}
 x \mapsto q^{-\frac{1}{2}(\beta, \beta)} x \quad \ \text{for} \ x \in \An_{\beta}.
\eneq
Note also that
the bar-operator $-$ is a ring anti-isomorphism
between $\An$ and $(\An, \cdot)$.
\end{remark}

\subsection{Modified quantum enveloping algebra}
For the materials in this subsection we refer the reader to
\cite{Lusz92}, \cite{Kash94}.
We denote by ${\rm Mod}(\g,\wl)$
the category of left $U_q(\g)$-modules with the weight space decomposition. Let $({\rm forget})$ be the functor from
${\rm Mod}(\g,\wl)$ to the category of vector spaces over $\Q(q)$,
forgetting the $U_q(\g)$-module structure.

Let us denote by $\mathscr{R}$ the endomorphism ring of $({\rm forget})$. Note that $\mathscr{R}$ contains $U_q(\g)$. For $\eta \in \wl$, let
$a_\eta \in \mathscr{R}$ denotes the projector $M \to M_\eta$ to the
weight space of weight $\eta$. Then the defining relation of $a_\eta$ (as a left
$\Uqg$-module) is
$$q^h \cdot a_\eta = q^{\langle h,\eta \rangle}a_\eta.$$
We have
\begin{align*}
a_\eta a_\zeta =\delta_{\eta,\zeta} a_\eta, \quad a_\eta P =
P a_{\eta-\xi} \quad \text{for $\xi \in \rl$ and $P \in \Uqg_\xi$.}
\end{align*}
Then $\mathscr{R}$ is isomorphic to $\displaystyle\prod_{\eta \in \wl} \Uqg
a_\eta$. We set
$$ \tUqg \seteq \soplus_{\eta \in \wl} \Uqg a_\eta\subset \mathscr{R}.$$
Then $\tUqg$ is a subalgebra of $\mathscr{R}$.
We call it the {\em modified quantum enveloping algebra}.
Note that any $\U$-module in ${\rm Mod}(\g,\wl)$ has a natural
$\tU$-module structure.

The (anti-)automorphisms $*$, $\vp$ and $\bar{ \ \ }$ of $\Uqg$ extend to the ones of $\tUqg$
by
$$ a_\eta^* = a_{-\eta}, \quad \varphi(a_\eta)=a_\eta, \quad  \overline{a}_\eta=a_\eta.$$

For a dominant integral weight $\lambda \in \wl^+$, let us denote by
$V(\lambda)$
(resp.\ $V(-\lambda)$) the irreducible module with highest (resp.\
lowest) weight $\lambda$ (resp.\ $-\lambda$). Let $u_\lambda$ (resp.\
$u_{-\lambda}$) be the highest (resp.\ lowest) weight vector.

For $\lambda \in \wl^+$, $\mu \in \wl^- \seteq - \wl^+$, we set
$$V(\la,\mu) \seteq V(\lambda) \tensm V(\mu).$$
Then $V(\la,\mu)$ is generated by $u_\lambda \tens
u_\mu$ as a $\Uqg$-module, and the defining relation of
$u_\lambda \tensm u_\mu$ is
\begin{align*}
& q^h(u_\lambda \tensm u_\mu)=q^{\langle h,\lambda+\mu \rangle}(u_\lambda \tensm u_\mu), \\
& e_i^{1-\langle h_i,\mu \rangle}(u_\lambda \tensm
u_\mu)=0, \quad f_i^{1+\langle h_i,\lambda \rangle}(u_\lambda
\tensm u_\mu)=0.
\end{align*}
Let us define the automorphism $\bar{ \ \ }$ of $V(\lambda,\mu)$
by
$$ \overline{P(u_\lambda \tensm u_\mu)} = \overline{P}(u_\lambda \tensm u_\mu).$$

We set
\begin{itemize}
\item[{\rm (i)}] $L^\low(\lambda,\mu) \seteq L^\low(\lambda) \tens_{\QA_0} L^\low(\mu)$,
\item[{\rm (ii)}]  $V(\lambda,\mu)_\A \seteq V(\lambda)_\A \tens_\A V(\mu)_\A$,
\item[{\rm (iii)}] $B(\lambda,\mu) \seteq B(\lambda) \tens B(\mu)$.
\end{itemize}

\begin{proposition}[\cite{Lusz92}] $(L^\low(\lambda,\mu),B(\lambda,\mu))$ is a lower crystal basis of $V(\la,\mu)$. Furthermore,
$\bl\Q\tens V(\lambda,\mu)_\A,\,L^\low(\lambda,\mu),\,\ol{L^\low(\lambda,\mu)}\br$ is balanced, and
there exists a lower global basis $\B^\low(V(\la,\mu))$ obtained from the lower crystal basis  $(L^\low(\lambda,\mu),B(\lambda,\mu))$.
\end{proposition}

\begin{theorem} [\cite{Lusz92}]\label{th:tU}
The algebra $\tUqg$ has a lower crystal basis $(L^\low(\tUqg),B(\tUqg))$ satisfying the following properties:
\bnum
\item $L^\low(\tUqg) = \soplus_{\lambda \in \wl} L^\low(\tUqg a_\lambda)$ and $B(\tUqg) = \bigsqcup_{\lambda \in \wl} B(\tUqg a_\lambda)$ where
\begin{itemize}
\item $L^\low(\tUqg a_\lambda) = L^\low(\tUqg) \cap \Uqg a_\lambda$ and
\item $B(\tUqg a_\lambda) = B(\tUqg) \cap \big( L^\low(\tUqg a_\lambda)/qL^\low(\tUqg a_\lambda) \big)$.
\end{itemize}
\item
Set $\tUqg_\A:=\soplus_{\eta \in \wl} \Uqg_\A  a_\eta$. Then
$\bl\Q\tens\tUqg_\A,\, L^\low(\tUqg),\,\ol{ L^\low(\tUqg)}\br$ is balanced,
and
$\tUqg$ has the lower global basis
$\B^\low(\tUqg)\seteq\{ G^\low(b)\mid b \in B(\tUqg)\}$.
\item For any $\lambda \in \wl^+$ and $\mu \in \wl^-$, let
$$\Psi_{\lambda,\mu}\cl\Uqg a_{\lambda+\mu} \to V(\lambda,\mu)$$
 be the $U_q(\g)$-linear map
$a_{\lambda+\mu} \longmapsto u_\lambda \tens u_\mu$.
Then we have $\Psi_{\lambda,\mu}\bl L(\tUqg a_{\lambda+\mu})\br =
L^\low(\lambda,\mu)$.
\item Let $\overline{\Psi}_{\lambda,\mu}$ be the induced homomorphism
$$L^\low(\tUqg a_{\lambda+\mu})/qL^\low(\tUqg a_{\lambda+\mu})   \longrightarrow  L^\low(\lambda,\mu)/qL^\low(\lambda,\mu) .$$
Then we have
\bnam
\item $\{ b \in B(\tUqg a_{\lambda+\mu}) \ | \ \overline{\Psi}_{\lambda,\mu}b \ne 0\} \overset{\sim}{\longrightarrow} B(\lambda,\mu)$,
\item $\Psi_{\lambda,\mu}\bl G^\low(b)\br =  G^\low(\overline{\Psi}_{\lambda,\mu}(b))$ for any $b \in B(\tUqg a_{\lambda+\mu})$.
\ee
\item $B(\tUqg)$ has a structure of crystal such that
 the injective map induced by {\rm (iv) (a)}
$$ B(\lambda,\mu) \to B(\tUqg a_{\lambda+\mu}) \subset B(\tUqg)$$
is a strict embedding of crystals  for any $\la\in\Pd$ and $\mu\in\Po^-$.
\end{enumerate}
\end{theorem}

For $\lambda \in \wl$, take any $\zeta \in \wl^+$ and $\eta \in \wl^-$ such that
$\lambda = \zeta +\eta$. Then $B(\zeta) \tens B(\eta)$ is
embedded into $B(\tUqg a_\eta)$.

For $\mu \in \wl$, let $T_\mu=\{ t_\mu \}$ be the crystal with
$$\wt(t_\mu) = \mu, \quad \ve_i(t_\mu) = \vp_i(t_\mu)=-\infty, \quad
\te_i (t_\mu) =\tf_i (t_\mu) =0.$$
Since we have
$$ B(\zeta) \hookrightarrow B(\infty) \tens T_\zeta, \  B(\eta) \hookrightarrow  T_\eta \tens B(-\infty) \text{ and }  T_\zeta \tens T_\eta \simeq T_{\zeta+\eta},$$
$B(\zeta) \tens B(\eta)$ is embedded into the crystal $B(\infty)
\tens T_\lambda \tens B(-\infty)$. Taking $\zeta \to \infty$ and $\eta \to -\infty$, we have
\Lemma [\cite{Kash94}]  For any $\la\in\Po$, we have a canonical crystal isomorphism
$$ B(\tUqg a_{\la}) \simeq B(\infty) \tens T_\lambda \tens B(-\infty).$$
\enlemma
Hence we identify
$$ B(\tUqg)=\bigsqcup_{\la\in\Po}B(\infty) \tens T_\lambda \tens B(-\infty).$$

\begin{theorem}  [\cite{Kash94}] \label{Thm: Kas94} \hfill
\bnum
\item $L^\low(\tUqg)$ is invariant under the anti-automorphisms $*$ and $\vp$.
\item $B(\tUqg)^* = \vp(B(\tUqg)) = B(\tUqg)$.
\item $\bl G^\low(b)\br^*=G^\low(b^*)$ and $\vp(G^\low(b))=G^\low(\vp(b))$ for $b \in B(\tUqg)$.
\end{enumerate}
\end{theorem}

\begin{corollary} [\cite{Kash94}] \label{cor: inv cry}
For $b_1 \in B(\infty)$, $b_2 \in B(-\infty)$, we have
\begin{enumerate}
\item[{\rm (1)}] $(b_1 \tens t_\mu \tens b_2)^* = b_1^* \tens t_{-\mu-\wt(b_1)-\wt(b_2)} \tens b_2^*$.
\item[{\rm (2)}] $\vp(b_1 \tens t_\mu \tens b_2) = \vp(b_2) \tens t_{\mu+\wt(b_1)+\wt(b_2)} \tens \vp(b_1)$.
\end{enumerate}
\end{corollary}

We define, for $b \in B$ with $B=B(\tUqg), B(\infty)$ or $B(-\infty)$,
\begin{align*}
\ve_i^*(b) = \ve_i(b^*),  \ \vp_i^*(b) = \vp_i(b^*), \  \wt^*(b)=\wt(b^*), \  \te_i^*(b)
= \te_i(b^*)^* \text{ and } \tf_i^*(b) = \tf_i(b^*)^*.
\end{align*}

This defines another crystal structure on $\tUqg$: For $b_1 \in
B(\infty)$ and $b_2 \in B(-\infty)$ and $\eta \in \wl$, we have
\begin{align*}
 \ve^*_i(b_1 \tens t_\eta \tens b_2) & = \max( \ve^*_i(b_1), \vp^*_i(b_2)+\langle h_i, \eta \rangle ), \\
 \vp_i^*(b_1 \tens t_\eta \tens b_2) & = \max( \ve^*_i(b_1)-\langle h_i, \eta \rangle , \vp^*_i(b_2)), \\
&= \ve_i^*(b_1 \tens t_\eta \tens b_2) + \langle  h_i,\wt^*(b_1 \tens t_\eta \tens b_2)\rangle, \\
 \wt^*(b_1 \tens t_\eta \tens b_2) & = -\eta, \\
 \te^*_i(b_1 \tens t_\eta \tens b_2) & =
 \begin{cases} (\te^*_ib_1) \tens t_{\eta-\alpha_i} \tens b_2 & \text{ if } \ve^*_i(b_1) \ge \vp^*_i(b_2)+\langle h_i,\eta \rangle, \\
               b_1 \tens t_{\eta-\alpha_i} \tens( \te^*_ib_2) & \text{ if } \ve^*_i(b_1) < \vp^*_i(b_2)+\langle h_i,\eta \rangle, \end{cases} \\
 \tf^*_i(b_1 \tens t_\eta \tens b_2) & =
 \begin{cases} (\tf^*_ib_1) \tens t_{\eta+\alpha_i} \tens b_2 & \text{ if } \ve^*_i(b_1) > \vp^*_i(b_2)+\langle h_i,\eta \rangle, \\
 b_1 \tens t_{\eta+\alpha_i} \tens ( \tf^*_ib_2) & \text{ if } \ve^*_i(b_1) \le \vp^*_i(b_2)+\langle h_i,\eta \rangle. \end{cases} \\
\end{align*}

We have
$$\te_i\circ \vphi=\vphi\circ\tf^*_i\quad\text{and}\quad
\tf_i\circ \vphi=\vphi\circ\tf^*_i\quad\text{for every $i\in I$.}$$

For $\xi\in\nrtl$ and $\eta\in\prtl$, we shall denote by
\begin{align*}
\Up_{<\eta}&\seteq\soplus_{\eta'\in\prtl\cap(\eta+\nrtl)\setminus\{\eta\}}\Up_{\eta'}.
\end{align*}
Then for any $\la\in \Po$, $b_-\in B(\infty)_\xi$ and
$b_+\in B(-\infty)_\eta$, we have
\eq
G^\low(b_-\tens t_\la\tens b_+)
-G^\low(b_-) G^\low(b_+)a_\la\in
\Um_{>\xi}\Up_{<\eta}a_\la.
\label{eq:tUc}
\eneq

\subsection{Relationship of $\Ag$ and $\tU$}

There exists a canonical pairing $A_q(\g) \times \tUqg \to \Q(q)$ by
\begin{equation} \label{eq: coupling}
\laa \psi, x a_\mu \raa = \delta_{\mu,\wt_\li(\psi)} \psi(x)
\quad\text{ for any $x\in\Um$ and $\mu\in\Po$.}
\end{equation}

\begin{theorem}[{\cite{Kash94}}] \label{thm: bicrystal embed}
There exists a bi-crystal embedding
\begin{equation} \label{eq: iota}
\oi_\g\cl  B(A_q(\g)) \To B(\tUqg)
\end{equation}
which satisfies:
$$\laa G^\up(b),\vp(G^\low(b')) \raa =  \delta_{\,\oi_\g(b),b'}$$
for any $b \in B(A_q(\g))$ and $b' \in B(\tUqg)$.
\end{theorem}

\subsection{Relationship of $\Ag$ and $\An$}

\begin{definition} \label{def: p_n}
 Let $p_\n\cl A_q(\g) \to A_q(\n)$ be the homomorphism induced by $U_q^+(\g) \to U_q(\g)$:
$$\langle p_\n(\psi),x \rangle = \psi(x) \quad \text{for \ any
$x \in U_q^+(\g)$.}$$
\end{definition}

Then we have
$$ \wt(p_\n(\psi))=\wt_\li(\psi)-\wt_\ri(\psi).$$

Note that the map $p_\n$ sends
the upper global basis of $A_q(\g)$ to the upper global basis of
$A_q(\n)$ or zero. Hence we have a map
$$  \vs_\n \cl  B(A_q(\g)) \to B(A_q(\n)) \bigsqcup \{ 0 \}.$$

It is obvious that $p_\n$ sends all $\Phi(u_{w\lambda}\tens u_{w\lambda}^\ri)$
($\lambda\in\wl^+$ and $w\in W$) to $1$. Note that
$\oi_\g(u_{w\lambda}\tens u_{w\lambda}^\ri)=b_\infty\tens t_{w\lambda}\tens b_{-\infty}\in B(\tUqg)$.

\begin{proposition} \label{prop: nonzero elt}
For $b \in B(A_q(\g))$,
set $\bg(b)=b_1 \tens t_\zeta \tens b_2
\in B(\infty) \tens T_\zeta \tens B(-\infty)\subset B(\tU)$ $(\zeta \in \wl)$.
Then we have
$$ p_\n(G^\up(b)) = \delta_{b_2,\,b_{-\infty}}G^\up(b_1).$$
\end{proposition}

\begin{proof}
Set $\eta \seteq \wt(b_1)+\zeta+\wt(b_2)=\wt_\lt(b)$. Then for any $\tilde{b} \in B(\infty)$, we
have
\begin{align*}
\big\lan p_\n(G^\up(b)), \vp(G^\low(\tilde{b})) \big\ra &= \big\laa G^\up(b), G^\low(\vp(\tilde{b}))a_{\eta} \big\raa \\
& \hspace{-20ex} = \big\laa G^\up(b), G^\low(b_\infty \tens t_\eta \tens \vp(\tilde{b})) \big\raa = \delta(\bg(b) = \tilde{b} \tens t_{\eta-\wt(\tilde{b}) } \tens b_{-\infty} ) \\
& \hspace{-20ex} =\delta(b_2=b_{-\infty},b_1 = \tilde{b}).
\end{align*}
\end{proof}

Although the map $\pn$ is not an algebra homomorphism,
$\pn$ preserves the multiplications up to a power of $q$, as we will see below.

\begin{lemma}  \label{lem:Deltan}
 For $x \in U_q^+(\g)$, if $\Delta_\n(x)=x_{(1)} \tens x_{(2)}$, then
\begin{align} \label{eq: Delta_n comp}
\comp(x)=q^{\wt(x_{(1)})} x_{(2)} \tens x_{(1)}.
\end{align}
\end{lemma}

\begin{proof}
Assume that \eqref{eq: Delta_n comp} holds for $x \in U_q^+(\g)$.
Note that
$$ \Delta_\n(e_ix) = (e_i \tens 1+1 \tens e_i)(x_{(1)} \tens x_{(2)}) =
e_ix_{(1)} \tens x_{(2)} + q^{-(\alpha_i,\wt(x_{(1)}))}x_{(1)} \tens (e_ix_{(2)}).$$
On the other hand, we have
\begin{align*}
\comp(e_ix)&=(e_i  \tens 1+ q^{\alpha_i} \tens e_i)(q^{\wt(x_{(1)})} x_{(2)} \tens x_{(1)}) \allowdisplaybreaks \\
&=(e_i q^{\wt(x_1)}) x_{(2)} \tens x_{(1)} +
( q^{\alpha_i+\wt(x_{(1)})} x_{(2)}) \tens (e_ix_{(1)}) \allowdisplaybreaks \\
& = q^{-(\alpha_i,\wt(x_{(1)}))}(q^{\wt(x_{(1)})}e_ix_{(2)}) \tens
x_{(1)}+ (q^{\wt(e_ix_{(1)})}x_{(2)}) \tens (e_ix_{(1)}).
\end{align*}
Hence \eqref{eq: Delta_n comp} holds for $e_ix$.
\end{proof}

\begin{proposition}\label{prop: p_n}
For $\psi,\theta \in A_q(\g)$, we have
$$ p_\n(\psi\theta)= q^{(\wt_\ri(\psi),\wt_\ri(\theta)-\wt_\li(\theta))}p_\n(\psi)p_\n(\theta).$$
\end{proposition}

\begin{proof}
For $x \in U_q^+(\g)$, set $\Delta_\n(x)=x_{(1)} \tens x_{(2)}$. Then, we
have
\begin{align*}
\lan p_\n(\psi\theta),x \ra & = \lan \psi\theta,x  \ra
= \lan \psi \tens \theta, q^{\wt(x_{(1)})} x_{(2)} \tens x_{(1)} \ra
= \lan \psi , q^{\wt(x_{(1)})} x_{(2)} \ra \lan \theta , x_{(1)} \ra \\
&= q^{(\wt_\ri(\psi),\wt(x_{(1)}))} \lan \psi ,  x_{(2)} \ra \lan \theta , x_{(1)} \ra
= q^{(\wt_\ri(\psi),\wt(x_{(1)}))} \lan p_\n(\psi) ,  x_{(2)} \ra
\lan p_\n(\theta) , x_{(1)} \ra  \\
&\underset{\mathrm{(a)}}{=} q^{(\wt_\ri(\psi),\wt_\ri(\theta)-\wt_\li(\theta))} \lan  p_\n(\psi)\tens p_\n(\theta),\Delta_\n(x) \ra \\
&= q^{(\wt_\ri(\psi),\wt_\ri(\theta)-\wt_\li(\theta))} \lan
p_\n(\psi)p_\n(\theta),x \ra .
\end{align*}
Here, we used $\wt(x_{(1)})=-\wt\bl p_\n (\theta )\br$  in (a).
\end{proof}

There exists an injective map
\eq &&\oi_\lambda\cl  B(\lambda)\to B(\infty)
\eneq
induced by the $\Up$-linear homomorphism
$\iota_\la\cl V(\lambda)\to A_q(\n)$ given by
$$v\longmapsto \big ( \Uqg^+\ni a\mapsto (av,u_\lambda)\big).$$
It commutes with $\te_i$.
We have
$$
G^\low_\lambda(b)=G^\low(\oi_\lambda(b))u_\lambda
 \quad\text{and}\quad
\iota_\la G_\la^\up(b)=G^\up(\oi_\la(b))\quad\text{for any $b\in B(\lambda)$.}
$$

\begin{proposition} [{\cite{Kash94}}]\label{pro:gltens}
Let $\lambda,\mu \in \wl^+$ and $w \in W$. Then for any $b \in
B(\tUqg a_{\lambda+w\mu})$, $G^\low(b)(u_\lambda \tensm
u_{w\mu})$ vanishes or is a member of the lower global basis of
$V(\lambda) \tensm V(\mu)$.
\end{proposition}

Hence we have a crystal morphism
\begin{align} \label{eq: pi la wmu}
\pi_{\lambda,w\mu}\cl  B(\tUqg a_{\lambda+w\mu}) \to B(\lambda)
\tens B(\mu)
\end{align}
by $G^\low(b)(u_\lambda \tensm
u_{w\mu})=G^\low(\pi_{\lambda,w\mu}(b))$.

\section{Quantum minors and T-systems}

{\em Hereafter, we assume that the generalized Cartan matrix $\cmA$
is symmetric,} although many of the results hold also in the non-symmetric case.
Hence we assume that  $\cmA=\bl(\al_i,\al_j)\br_{i,j\in I}$.

\subsection{Quantum minors}
Using the isomorphism $\Phi$ in \eqref{eq: Phi}, for each
$\lambda \in \wl^+$ and $\mu,\xi \in W \lambda$, we define the elements
$$ \Delta(\mu,\zeta)\seteq \Phi(u_\mu\tens u_\zeta^\ri) \in A_q(\g)$$
and
$$\D(\mu,\zeta)\seteq p_\n(\Delta(\mu,\zeta)) \in A_q(\n).$$
The element $\Delta(\mu,\zeta)$ is called a (generalized) {\em quantum minor}
and $\D(\mu,\zeta)$ is called {\em a unipotent quantum minor}.

\begin{lemma} \label{lem: global element}
$\Delta(\mu,\zeta)$ is a member of the upper global basis of $A_q(\g)$. Moreover
$\D(\mu,\zeta)$
is either a member of the upper global basis of $ A_q(\n)$ or zero.
\end{lemma}
\begin{proof}  It follows from Proposition~\ref{prop:Agglobal} and
Proposition~\ref{prop: nonzero elt}.
\end{proof}
\Lemma[{\cite[{(9.13)}]{BZ05}}]
 For $u,v \in W$ and $\lambda,\mu \in \wl^+$, we have
$$ \Delta(u\lambda,v\lambda)\Delta(u\mu,v\mu) =
\Delta\bl u (\lambda+\mu),v(\lambda+\mu)\br.$$
\enlemma

By Proposition \ref{prop: p_n}, we have the
following corollary:

\begin{corollary}\label{cor:Duv}
For $u,v \in W$ and $\lambda,\mu \in \wl^+$, we have
$$ \D(u\lambda,v\lambda)\D(u\mu,v\mu) =
q^{-(v\lambda,v\mu-u\mu)}\D\bl u (\lambda+\mu),v(\lambda+\mu)\br.$$
\end{corollary}

Note that
$$ \D(\mu,\mu)=1 \quad \text{ for $\mu \in W\lambda$ }.$$

Then $\D(\mu,\zeta) \ne 0$ if and only if $\mu \preceq \zeta$. Recall
that for $\mu,\zeta$ in the same $W$-orbit, we say that $\mu \preceq
\zeta$ if there exists a sequence $\{ \beta_k \}_{1 \le k \le l}$ of positive real roots such that, defining $\lambda_0=\zeta$,
$\lambda_k=s_{\beta_k}\lambda_{k-1}$ $(1 \le k \le l)$, we have
$(\beta_k,\lambda_{k-1}) \ge 0$ and $\lambda_l=\mu$.

More precisely, we have

\begin{lemma} \label{lem: D-property}
Let $\lambda \in \wl^+$ and $\mu,\zeta \in W \lambda$. Then the
following conditions are equivalent:
\begin{enumerate}
\item[{\rm (a)}] $\D(\mu,\zeta)$ is an element of upper global basis of $A_q(\n)$,
\item[{\rm (b)}] $\D(\mu,\zeta) \ne 0$,
\item[{\rm (c)}] $u_\mu \in U_q^-(\g)u_\zeta$,
\item[{\rm (d)}] $u_\zeta \in U_q^+(\g)u_\mu$,
\item[{\rm (e)}] $\mu \preceq \zeta$,
\item[{\rm (f)}] for any $w \in W$ such that $\mu = w \lambda$, there exists $u \le w$ \ro in the Bruhat order\rf\ such that $\zeta=u \lambda$,
\item[{\rm (g)}] there exist $u,v \in W$ such that $\mu = w \lambda$, $\zeta=u\lambda$ and $ u \le w$.
\end{enumerate}
\end{lemma}
\Proof
The equivalence of (b), (c) and  (d) is obvious.
The equivalence of (e), (f), (g) is well-known.
The equivalence of (d) and (f) is proved in \cite{Kas93a}.
\QED

For any $u\in\An\setminus\{0\}$ and $i\in I$,
we set
\begin{align*}
\eps_i(u)&\seteq\max\set{n\in\Z_{\ge0}}{e_i^nu\not=0},\\
\eps^*_i(u)&\seteq\max\set{n\in\Z_{\ge0}}{e_i^{*\,n}u\not=0}.
\end{align*}
Then for any $b\in B(\An)$, we have
$$\text{$\eps_i(G^\up(b))=\eps_i(b)$ and $\eps^*_i(G^\up(b))=\eps^*_i(b)$.}$$

\Lemma\label{lem: weyl left right}
Let $\lambda \in \wl^+$, $\mu,\zeta \in W \lambda$
such that $\mu \preceq \zeta$ and $i \in I$.
\bnum
\item If $n \seteq \lan h_i,\mu \ra \ge 0$, then
$$ \ve_i(\D(\mu,\zeta))=0 \ \ \text{ and } \ \ e_i^{(n)}\D(s_i\mu,\zeta)=\D(\mu,\zeta).$$
\item If $\lan h_i,\mu \ra \le 0$ and $s_i\mu \preceq \zeta$, then $\ve_i(\D(\mu,\zeta))=-\lan h_i,\mu \ra$.
\item If $m \seteq -\lan h_i,\zeta \ra \le 0$, then
$$ \ve^*_i(\D(\mu,\zeta))=0 \ \ \text{ and } \ \ e^{*\;(m)}_i \D(\mu,s_i\zeta)=
\D(\mu,\zeta).$$
\item If $\lan h_i,\zeta \ra \ge 0$ and $\mu \preceq s_i\zeta$, then $\ve^*_i(\D(\mu,\zeta))=\lan h_i,\zeta \ra$.
\end{enumerate}
\enlemma
\Proof
We have
$\eps_i\bl\Delta(\mu,\zeta)\br=\max(-\lan h_i,\mu\ran, 0)$
and $\eps_i^*\bl\Delta(\mu,\zeta)\br=\max(\lan h_i,\zeta\ran, 0)$.
Moreover, $\pn$ commutes with $e_i^{(n)}$ and ${e_i^*}^{(n)}$.

Let us show (ii). Set $\ell=-\lan h_i,\mu \ra$.
Then we have
$e_i^{\ell+1}\Delta(\mu,\zeta)=0$, which implies
$e_i^{\ell+1}\D(\mu,\zeta)=0$. Hence $\eps_i(\D(\mu,\zeta))\le \ell$.
We have
$$e_i^{(\ell)}\Delta(\mu,\zeta)=\Delta(s_i\mu,\zeta).$$
Hence we have $e_i^{(\ell)}\D(\mu,\zeta)=\D(s_i\mu,\zeta)$.
By the assumption  $s_i\mu \preceq \zeta$, $\D(s_i\mu,\zeta)$ does not vanish.
Hence we have $\eps_i(\D(\mu,\zeta))\ge \ell$.

\smallskip
The other statements can be proved similarly.
\QED

\begin{proposition}[{\cite[(10.2)]{BZ05}}] \label{prop: BZ form} Let $\lambda,\mu \in \wl^+$ and $s,t,s',t' \in W$ such that $\ell(s's)=\ell(s')+\ell(s)$ and
$\ell(t't)=\ell(t')+\ell(t)$. Then we have
\bnum
\item
$ \Delta(s's\lambda,t'\lambda)\Delta(s'\mu,t't\mu)=q^{(s\lambda,\mu)-(\lambda,t\mu)}\Delta(s'\mu,t't\mu)\Delta(s's\lambda,t'\lambda)$.
\item
If we assume further that $s's\lambda \preceq t'\lambda$ and
$s'\mu \preceq t't\mu$, then we have
\begin{align}
& \D(s's\lambda,t'\lambda)\D(s'\mu,t't\mu) = q^{(s's\lambda+t'\lambda,\;s'\mu-t't\mu)} \D(s'\mu,t't\mu)\D(s's\lambda,t'\lambda),
\label{eq: quantum commuting factor} \end{align}
or equivalently
\begin{align}
& q^{(t'\lambda,\;t't\mu-s'\mu)}\D(s's\lambda,t'\lambda)\D(s'\mu,t't\mu)
=q^{(s'\mu-t't\mu,\;s's\lambda)}\D(s'\mu,t't\mu)\D(s's\lambda,t'\lambda)
\label{eq: quantum commuting factor 2}
\end{align}
\ee
\end{proposition}
Note that (ii) follows from by Proposition \ref{prop: p_n} and (i).
Note also that the both sides of \eqref{eq: quantum commuting factor 2}
are bar-invariant, and hence they are members of the upper global basis
as seen by \cite[Corollary 3.5]{KKKO15}.

\Prop\label{prop:Deprod}
For $\la,\mu\in\Pd$ and $s,t\in W$,
set
$\bg\bl u_{s\la}\tens (u_\la)^\rt\br=b_-\tens t_\la\tens b_{-\infty}$
and $\bg\bl u_{\mu}\tens (u_{t\mu})^\rt\br=b_{\infty}\tens t_{t\mu}\tens b_+$
with $b_{\mp}\in B(\pm\infty)$.
Then we have
$$\De{s\la,\la}\De{\mu,t\mu}=G^\up\bl\bg^{-1} (b_-\tens t_{\la+t\mu}\tens b_+)\br.$$
\enprop
\Proof
Let
$(\scbul,\scbul):
\bl V(\la)\tensm V(\mu)\br\times \bl V(\la)\tensp V(\mu)\br\to \Q(q)$
be the coupling defined
$(u\tensm  v,u'\tensp v')=(u,u')(v,v')$.
Then it satisfies
$$\bl P(u\tensm v), u'\tensp v'\br=\bl u\tensm v, \vphi(P)(u'\tensp v')\br
 \quad \text{for any} \ P \in \Uqg.
$$
For $u,u'\in V(\la)$ and $v,v'\in V(\mu)$,
we have
\begin{align*}
\lan \Phi(u\tens u'^\rt)\Phi(v\tens v'^\rt), P\ran&=
\bl u'\tensm v', P(u\tensp v)\br\\
&=\bl\vphi(P)(u'\tensm v'),u\tensp v\br.
\end{align*}
Hence for $P\in\U$, we have
\begin{align*}\lan \De{s\la,\la}\De{\mu,t\mu}, Pa_{\zeta}\ran
&=\delta(\zeta=s\la+\mu)
\bl\vphi(P)(u_{\la}\tensm u_{t\mu}),u_{s\la}\tensp u_{\mu}\br.
\end{align*}
If $Pa_\zeta=G^\low(\vphi(b))$ for $b\in B(\tU)$,
then
we have
$$\lan \De{s\la,\la}\De{\mu,t\mu},\vphi(G^\low(b))\ran
=\delta(\zeta=s\la+\mu)
\bl G^\low(b)(u_{\la}\tensm u_{t\mu}),u_{s\la}\tensp u_{\mu}\br.$$
The element
$G^\low(b)(u_{\la}\tensm u_{t\mu})$
vanishes or is a global basis of $V(\la)\tensm V(\mu)$
by Proposition~\ref{pro:gltens}.
Since $u_{s\la}\tensp u_{\mu}$ is a member of the upper global basis of
$V(\la)\tensp V(\mu)$,
 we have
$$\lan \De{s\la,\la}\De{\mu,t\mu},\vphi(G^\low(b))\ran
=\delta(\zeta=s\la+\mu)
\delta\bl \pi_{\la,t\mu}(b)=u_{s\la}\tens u_{\mu}\br.$$
Here $\pi_{\la,t\mu}\cl B(\tU a_{\la+t\mu})\to B(\la)\tens B(\mu)$
is the crystal morphism given in \eqref{eq: pi la wmu}.

Hence we obtain
$$\De{s\la,\la}\De{\mu,t\mu}=G^\up( \bg^{-1}( b))$$
where $b\in B(\tU)$ is a unique element such that
$\bl G^\low(b)(u_{\la}\tensm u_{s\mu}), u_{s\la}\tensp u_{\mu}\br=1$.

On the other hand, we have
$G^\low(b_+)u_{t\mu}=u_{\mu}$ and
$G^{\low}(b_-)u_\la=u_{s\la}$.
The last equality implies $\vphi(G^{\low}(b_-))u_{s\la}=u_\la$ because
$(\vphi(G^{\low}(b_-))u_{s\la},u_\la)=(u_{s\la},G^{\low}(b_-)u_\la)=(u_{s\la},u_{s\la})=1$.
As seen in \eqref{eq:tUc}, we have
$$G^{\low}(b_-)G^\low(b_+)a_{\la+t\mu}-G^\low(b_-\tens t_{\la+t\mu}\tens b_+)
\in \Um_{>s\la-\la}\Up_{<\mu-t\mu}a_{\la+t\mu}.$$
Hence we obtain
\eqn
&&\bl G^\low(b_-\tens t_{\la+t\mu}\tens b_+)(u_{\la}\tensm u_{t\mu}),\;
u_{s\la}\tensp u_{\mu}\br\allowdisplaybreaks[1]\\
&&\hs{10ex}
=\bl G^{\low}(b_-)G^\low(b_+)(u_{\la}\tensm u_{t\mu}),\;
u_{s\la}\tensp u_{\mu}\br\allowdisplaybreaks[1]\\
&&\hs{20ex}
=\bl G^\low(b_+)(u_{\la}\tensm u_{t\mu}),\;
\vphi(G^\low(b_-))(u_{s\la}\tensp u_{\mu})\br
=1.
\eneqn
In the last equality, we used
$G^\low(b_+)(u_{\la}\tensm u_{t\mu})
=u_{\la}\tensm (G^\low(b_+)u_{t\mu})=u_{\la}\tensm u_\mu$
and
$\vphi(G^\low(b_-))(u_{s\la}\tensp u_{\mu})
=\bl\vphi(G^\low(b_-))u_{s\la})\tensp u_{\mu}=u_{\la}\tensp u_\mu$.

Hence we conclude that $b=b_-\tens t_{\la+t\mu}\tens b_+$.
\QED

Let
\eq
&&\iota_{\la,\mu}\cl V(\la+\mu)\monoto V(\la)\tens V(\mu)\eneq
be the canonical embedding and
\eq
&&\bio_{\la,\mu}\cl B(\la+\mu)\monoto B(\la)\tens B(\mu)\eneq
the induced crystal embedding.

\begin{lemma}
For $\lambda,\mu \in \wl^+$ and $x,y \in W$ such that $x \ge y$, we have
$$ u_{x\lambda} \otimes u_{y \mu} \in \overline{\iota}_{\lambda,\mu} (B(\lambda+\mu)) \subset B(\lambda) \otimes B(\mu).$$
\end{lemma}

\begin{proof}
Let us show by induction on $\ell(x)$, the length of $x$ in $W$. We may assume that $x \ne 1$. Then there exists $i \in I$ such that
$s_i x < x$. If $s_iy<y$, then $s_ix \ge s_iy$ and $\tilde{e}_i^{\max}(u_{x\lambda} \otimes u_{y\mu})= u_{s_ix\lambda} \otimes u_{s_iy\mu}$.

If $s_iy>y$, then $s_ix\ge y$ and
$\tilde{e}_i^{\max}(u_{x\lambda} \otimes u_{y\mu})= u_{s_ix\lambda} \otimes u_{y\mu}$. In
both cases,
$u_{x\lambda} \otimes u_{y \mu}$  is connected with an element of $\overline{\iota}_{\lambda,\mu}(B(\lambda+\mu))$.
\end{proof}

\Lemma \label{lem:Deprod}
For $\la, \mu\in\Pd$ and $w\in W$, we have
$$\De{w\la,\la}\De{\mu,\mu}=
G^\up\bl\bio_{\la,\mu}^{-1}(u_{w\la}\tens u_\mu)\tens u_{\la+\mu}{}^\rt\br.$$
\enlemma
\Proof We have
\eqn
&&\bg(u_{w\la}\tens u_\la^\rt\br=b_{w\la}\tens t_{\la}\tens b_{-\infty},\\
&&\bg(u_{\mu}\tens u_\mu^\rt\br=b_{\infty}\tens t_{\mu}\tens b_{-\infty},
\eneqn
where $b_{w\la}\seteq\bio_\la(u_{w\la})$.
Hence Proposition~\ref{prop:Deprod}
implies that
$$\De{w\la,\la}\De{\mu,\mu}=G^\up\bl\bg^{-1}(b_{w\la}\tens t_{\la+\mu}\tens b_{-\infty})\br.$$
Then, $\bg\bl\bio_{\la,\mu}^{-1}(u_{w\la}\tens u_\mu)\tens u_{\la+\mu}{}^\rt\br
=b_{w\la}\tens t_{\la+\mu}\tens b_{-\infty}$
implies the desired result.
\QED

\subsection{T-systems} \label{subsec: T-system} In this subsection, we record the {\em $T$-system} among the (unipotent) quantum minors for later use
(see \cite{KNS11} for $T$-system).

\begin{proposition} [{\cite[Proposition 3.2]{GLS11}}] \label{prop: the ses}
Assume that $u,v \in W$ and $i \in I$ satisfy $u < us_i$ and $v <
vs_i$. Then
\begin{align*}
& \Delta(us_i\varpi_i,vs_i\varpi_i)\Delta(u\varpi_i,v\varpi_i) =
q^{-1}\Delta(u  s_i \varpi_i,v\varpi_i)\Delta( u\varpi_i, v s_i\varpi_i)
+ \prod_{j \ne i}\Delta(u\varpi_j,v\varpi_j)^{-a_{i,j}}, \\
& \Delta(u\varpi_i,v\varpi_i)\Delta(us_i\varpi_i,vs_i\varpi_i) =
q\Delta( u\varpi_i\ ,vs_i\varpi_i)\Delta(u  s_i  \varpi_i,v\varpi_i) +
\prod_{j \ne i}\Delta(u\varpi_j,v\varpi_j)^{-a_{i,j}},
\end{align*}
and
\begin{align*}
& q^{(vs_i\varpi_i,v\varpi_i-u\varpi_i)}\D(us_i\varpi_i,vs_i\varpi_i)\D(u\varpi_i,v\varpi_i) \allowdisplaybreaks\\
& \hspace{5ex} = q^{-1+(v\varpi_i,vs_i\varpi_i-u\varpi_i)}\D(us_i\varpi_i,v\varpi_i)\D(u\varpi_i,vs_i\varpi_i) +\D(u\lambda,v\lambda) \allowdisplaybreaks\\
& \hspace{5ex} = q^{-1+(vs_i\varpi_i,v\varpi_i-us_i\varpi_i)}\D(u\varpi_i,vs_i\varpi_i)\D(us_i\varpi_i,v\varpi_i) +\D(u\lambda,v\lambda), \allowdisplaybreaks\\
& q^{(v\varpi_i,vs_i\varpi_i-us_i\varpi_i)}\D(u\varpi_i,v\varpi_i)\D(us_i\varpi_i,vs_i\varpi_i) \allowdisplaybreaks\\
& \hspace{5ex} = q^{1+(vs_i\varpi_i,v\varpi_i-us_i\varpi_i)}\D(u\varpi_i,vs_i\varpi_i)\D(us_i\varpi_i,v\varpi_i) +\D(u\lambda,v\lambda) \allowdisplaybreaks\\
& \hspace{5ex} = q^{1+(v\varpi_i,vs_i\varpi_i-u\varpi_i)}\D(us_i\varpi_i,v\varpi_i)\D(u\varpi_i,vs_i\varpi_i)+\D(u\lambda,v\lambda),
\end{align*}
where $\lambda=s_i\varpi_i+\varpi_i = -\sum_{j \ne i}a_{j,i}\varpi_j$.
\end{proposition}

\subsection{Revisit of crystal bases and global bases}

In order to prove Theorem~\ref{thm: DDD} below, we first investigate
the upper crystal lattice of
$\Dv V$ induced by an upper  crystal lattice of $V\in\Oint(\g)$.

Let $V$ be a $U_q(\g)$-module in $\Oint$. Let $L^\up$ be an upper crystal lattice of $V$. Then we have (see Lemma~\ref{lem:lowup})
$$ \soplus_{\xi \in \wl} q^{(\xi,\xi)/2}(L^\up)_\xi \text{ is a lower crystal lattice of $V$.}$$

Recall that, for $\lambda\in \wl^+$, the upper crystal lattice $L^\up(\lambda)$ and the lower crystal lattice $L^\low(\lambda)$ of $V(\lambda)$ are related by
\begin{align}
L^\up(\lambda) = \soplus_{\xi \in \wl} q^{((\lambda,\lambda)-(\xi,\xi))/2}L^\low(\lambda)_\xi \subset L^\low(\lambda).
\end{align}

Write
$$V \simeq \soplus_{\lambda \in \wl^+} E_\lambda \otimes V(\lambda)$$
with finite-dimensional $\Q(q)$-vector spaces $E_\lambda$ .
Accordingly, we have a canonical decomposition
$$L^\up \simeq \soplus_{\lambda\in \wl^+} C_\lambda \otimes_{\QA_0} L^\up(\lambda),$$
where $C_\lambda \subset E_\lambda$ is an $\QA_0$-lattice of $E_\lambda$.

On the other hand, we have
$$\Dv V \simeq\soplus_{\lambda\in \wl^+}  E^*_\lambda \otimes V(\lambda).$$

Note that we have
$$ \Phi_V \left( (a \otimes u) \otimes (b \otimes v)^\ri \right) = \lan a,b \ra \Phi_\lambda(u \otimes v^\ri)
\quad \text{for $u,v \in V(\lambda)$ and $a \in E_\lambda$, $b \in E^*_\lambda$.}$$

We define the induced upper crystal lattice $\Dv L^\up$ of $\Dv V$ by
$$\Dv L^\up \seteq \soplus_{\lambda \in \wl^+} C_\lambda^\vee \otimes_{\QA_0} L^\up(\lambda)
\subset \Dv V, $$
where $C_\lambda^\vee \seteq  \set{ u \in E_\lambda^*}{\lan u,C_\lambda \ra \subset \QA_0 }$. Then we have
$$ \Phi_V \left(L^\up \otimes (\Dv L^\up)^\ri \right) \subset L^\up(A_q(\g)).$$

Indeed, we have
$$ \Dv L^\up =\set{ u \in \Dv V }{\Phi_V(L^\up \otimes u^\ri) \subset L^\up(A_q(\g)) }. $$

Since $(L^\up(\lambda))^\vee = L^\low(\lambda)$, we have
$$ (L^\up)^\vee = \soplus_{\lambda \in \wl^+} C_\lambda^\vee \otimes_{\QA_0} L^\low(\lambda).$$

The properties $L^\up(\lambda) \subset L^\low(\lambda)$ and $L^\up(\lambda)_\lambda = L^\low(\lambda)_\lambda$ imply the following lemma.

\begin{lemma} \label{lem: largest upper crystal}
$\Dv L^\up$ is the largest upper crystal lattice of $\Dv V$ contained in the lower crystal lattice $(L^\up)^\vee$.
\end{lemma}

Let $\la, \mu\in\Pd$.
Then
$\bl L^\up(\la)\tensp L^\up(\mu)\br^\vee
=L^\low(\la)\tensm L^\low(\mu)$ is a lower crystal lattice of
$\Dv\bl V(\la)\tensp V(\mu)\br\simeq V(\la)\tensm V(\mu)$.
Let
$\Xi_{\la,\mu}\cl V(\la)\tensp V(\mu)\isoto V(\la)\tensm V(\mu)
\simeq \Dv\bl V(\la)\tensp V(\mu)\br$
be the $\U$-module isomorphism defined by
$$\Xi_{\la,\mu}(u\tensp v)=q^{(\la,\mu)-(\xi,\eta)}\bl u\tensm v\br
\quad\text{for $u\in V(\la)_\xi$ and $v\in V(\mu)_\eta$.}
$$

Then
\eqn
\tLt&\seteq&\Xi_{\la,\mu}\bl L^\up(\la)\tensp L^\up(\mu)\br\\
&\;=&\soplus_{\xi,\eta\in\Po}
q^{(\la,\mu)-(\xi,\eta)}L^\up(\la)_\xi\tensm L^\up(\mu)_\eta.
\eneqn
is an upper crystal lattice of $V(\la)\tensm V(\mu)$.
Since we have $(\la,\mu)-(\xi,\eta)\ge0$
for any $\xi\in\wt\bl V(\la)\br$ and $\eta\in\wt\bl V(\mu)\br$,
Lemma~\ref{lem: largest upper crystal} implies that
$$\tLt\subset \Dv\bl L^\up(\la)\tensp L^\up(\mu)\br.$$

\begin{lemma}
Let $\lambda, \mu \in \wl^+$ and $x_1,x_2,y_1,y_2 \in W$ such that $x_k \ge y_k$
 $(k=1,2)$. Then we have
\begin{equation} \label{eq: multi}
\begin{aligned}
& q^{(\lambda,\mu)-(x_2\lambda,y_2\mu)} \Delta(x_1\lambda,x_2\lambda)\Delta(y_1\mu,y_2\mu) \\
& \hspace{10ex} \equiv G^\up( \oi^{-1}_{\lambda,\mu}(u_{x_1\lambda} \otimes u_{y_1 \mu}) \otimes  \oi^{-1}_{\lambda,\mu}(u_{x_2\lambda} \otimes u_{y_2 \mu})^\ri )
\quad {\rm mod} \ q L^\up(A_q(\g)).
\end{aligned}
\end{equation}
\end{lemma}
\Proof
By the definition, we have
$$\De{x_1\la,x_2\la}\De{y_1\mu,y_2\mu}
=\Phi_{V(\la)\tensp V(\mu)}\bl
(u_{x_1\la}\tensp u_{y_1\mu})\tens (u_{x_2\la}\tensm u_{y_2\mu})^\rt\br.$$
Hence we have
\eqn
&&q^{(\la,\mu)-(x_2\la,y_2\mu)}
\De{x_1\la,x_2\la}\De{y_1\mu,y_2\mu}\\
&&\hs{3ex}=\Phi_{V(\la)\tensp V(\mu)}\bl
(u_{x_1\la}\tensp u_{y_1\mu})\tens q^{(\la,\mu)-(x_2\la,y_2\mu)}(u_{x_2\la}\tensm u_{y_2\mu})^\rt\br\\
&&\hs{6ex}=\Phi_{V(\la)\tensp V(\mu)}\bl
(u_{x_1\la}\tensp u_{y_1\mu})\tens
\bl\Xi_{\la,\mu}(u_{x_2\la}\tensp u_{y_2\mu})\br^\rt\br.
\eneqn
The right-hand side of \eqref{eq: multi} can be calculated as follows.
Let us take $v_k\in L^\up(\la+\mu)$ such that $\iota_{\la,\mu}(v_k)-u_{x_k\la}\tensp u_{y_k\mu}
\in qL^\up(\la)\tensp L^\up(\mu) $  for $k=1,2$.
Here $\iota_{\la,\mu}\cl V(\la+\mu)\to V(\la)\tensp V(\mu)$
denotes the canonical $\U$-module homomorphism.
Then we have
\eqn
&&G^\up\Bigl(\bio_{\la,\mu}^{\ -1}(u_{x_1\la}\tens u_{y_1\mu})
\tens\bl  \bio_{\la,\mu}^{\ -1}(u_{x_2\la}\tens u_{y_2\mu})\br^\rt\Bigl)\\
&&\hs{5ex}\equiv \Phi_{\la+\mu}(v_1\tens v_2^\rt)\mod qL^\up(\Ag)\\
&&\hs{6ex}=\Phi_{V(\la)\tensp V(\mu)}\bl \iota_{\la,\mu}(v_1)\tens
\bl\Xi_{\la,\mu}\iota_{\la,\mu}(v_2))^\rt\br
\eneqn
On the other hand, we have
$$\iota_{\la,\mu}(v_1)\equiv
u_{x_1\la}\tensp u_{y_1\mu}\mod qL^\up(\la)\tensp L^\up(\mu)$$
and
$$\Xi_{\la,\mu}\bl\iota_{\la,\mu}(v_2)\br
\equiv \Xi_{\la,\mu}(u_{x_2\la}\tensp u_{y_2\mu})\mod q\tLt.
$$
Hence
\eqn
&&\Phi_{V(\la)\tensp V(\mu)}\bl
(u_{x_1\la}\tensp u_{y_1\mu})\tens \Xi_{\la,\mu}(u_{x_2\la}\tensp u_{y_2\mu})^\rt\br\\
&&\hs{10ex}\equiv
\Phi_{V(\la)\tensp V(\mu)}\bl (\iota_{\la,\mu}(v_1)\tens
(\Xi_{\la,\mu}\iota_{\la,\mu}(v_2))^\rt\br
\mod qL^\up(\Ag).
\eneqn
\QED

\begin{theorem} \label{thm: DDD}
Let $\lambda \in \wl^+$ and $x,y \in W$ such that $x \ge y$. Then we have
$$ \D(x\lambda,y\lambda)\D(y\lambda,\lambda) \equiv \D(x\lambda,\lambda) \quad {\rm mod} \ qL^\up(A_q(\n)).$$
\end{theorem}

\begin{proof}
Applying $p_\n$ to \eqref{eq: multi}, we have
\begin{equation}  \label{eq: DDD}
\begin{aligned}
& \D(x\lambda,y\lambda)\D(y\lambda,\lambda) \\
& \hspace{10ex} \equiv p_\n
\left( G^\up( \oi^{-1}_{\lambda,\lambda}(u_{x\lambda} \otimes u_{y \lambda})  \otimes \oi^{-1}_{\lambda,\lambda}(u_{y\lambda} \otimes u_{\lambda})^\ri ) \right)
\quad {\rm mod} \ q L^\up(A_q(\n)).
\end{aligned}
\end{equation}
Hence the desired result follows from
Proposition~\ref{prop: nonzero elt} and  Lemma \ref{lem: iotas} below.
\end{proof}

\Lemma\label{lem: iotas}
Let $\la\in \Pd$ and $x,y\in W$ such that $x\ge y$.
Then we have
$$
\bio_\g\Bigl(
\bio_{\la,\la}^{\ -1}(u_{x\la}\tens u_{y\la})
\tens\bl  \bio_{\la,\la}^{\ -1}(u_{y\la}\tens u_{\la})\br^\rt\Bigr)
=\bio_\la(u_{x\la})\tens t_{y\la+\la}\tens b_{-\infty}.$$
\enlemma

\Proof
We shall argue by induction on $\ell(x)$.
We set $b_{x\la}=\bio_\la(u_{x\la})$.
Since the case $x=1$ is obvious, assume that $x\not=1$.
Take $i\in I$ such that $x'\seteq s_ix<x$

\medskip
\noi
(a) First assume that $s_iy>y$. Then we have
$y\le  x'$.
Hence by the induction hypothesis,
\eq
&&\bio_\g\Bigl(
\bio_{\la,\la}^{\ -1}(u_{x'\la}\tens u_{y\la})
\tens\bl  \bio_{\la,\la}^{\ -1}(u_{y\la}\tens u_{\la})\br^\rt\Bigr)
=b_{x'\la}\tens t_{y\la+\la}\tens b_{-\infty}.\label{eq:agtu}
\eneq
We have $\vphi_i(u_{x'\la})=\lan h_i,x'\la\ran$ and
$\vphi_i(b_{x'\la}\tens t_{y\la+\la})=\lan h_i,x'\la\ran
+\lan h_i,y\la\ran\ge\lan h_i,x'\la\ran$.
Hence, applying $\tf_i^{\lan h_i,x'\la\ran}$ to \eqref{eq:agtu}, we obtain
\eqn
&&\bio_\g\Bigl(
\bio_{\la,\la}^{\ -1}(u_{x\la}\tens u_{y\la})
\tens\bl  \bio_{\la,\la}^{\ -1}(u_{y\la}\tens u_{\la})\br^\rt\Bigr)
=b_{x\la}\tens t_{y\la+\la}\tens b_{-\infty}.
\eneqn

\medskip
\noi
(b) Assume that $y'\seteq s_iy<y$. Then we have
$y'\le x'$, and the induction hypothesis implies that
$$
\bio_\g\Bigl(
\bio_{\la,\la}^{\ -1}(u_{x'\la}\tens u_{y'\la})
\tens\bl  \bio_{\la,\la}^{\ -1}(u_{y'\la}\tens u_{\la})\br^\rt\Bigr)
=b_{x'\la}\tens t_{y'\la+\la}\tens b_{-\infty}.$$
We shall apply $\te_i^{*\;\lan h_iy'\la\ran}\tf_i^{\lan h_i,x'\la+y'\la\ran}$
to the both sides.
Then the left-hand side yields
$$\bio_\g\Bigl(
\bio_{\la,\la}^{\ -1}(u_{x\la}\tens u_{y\la})
\tens\bl  \bio_{\la,\la}^{\ -1}(u_{y\la}\tens u_{\la})\br^\rt\Bigr).$$
Since $\vphi_i(b_{x'\la}\tens t_{y'\la+\la})=\lan h_i,x'\la\ran
+\lan h_i,y'\la+\la\ran\ge\lan h_i,x'\la+y'\la\ran$,
the right-hand side yields
\begin{align*}
\te_i^{*\;\lan h_i,y'\la\ran}\tf_i^{\lan h_i,x'\la+y'\la\ran}
\bl b_{x'\la}\tens t_{y'\la+\la}\tens b_{-\infty}\br
&=\te_i^{*\;\lan h_i,y'\la\ran}\bl(\tf_i^{\lan h_i,x'\la+y'\la\ran}
 b_{x'\la})\tens t_{y'\la+\la}\tens b_{-\infty}\br\\
&=\te_i^{*\;\lan h_i,y'\la\ran}\bl(\tf_i^{\lan h_i,y'\la\ran}
 b_{x\la})\tens t_{y'\la+\la}\tens b_{-\infty}\br.
\end{align*}
Since
$\eps_i^*(b_{x\la})=-\vphi_i(b_{x\la})=\lan h_i, \la\ran$ and
$\tf_i^{\lan h_i,y'\la\ran}b_{x\la}=\tf_i^{*\;\lan h_i,y'\la\ran}b_{x\la}$,
we have
\begin{align*}\te_i^{*\;\lan h_i,y'\la\ran}\bl(\tf_i^{\lan h_i,y'\la\ran}
 b_{x\la})\tens t_{y'\la+\la}\tens b_{-\infty}\br
&=b_{x\la}\tens t_{y\la+\la}\tens b_{-\infty}.
\end{align*}
\QED

\subsection{Generalized T-systems}

The $T$-system in \S \ref{subsec: T-system} can be interpreted as an equation among the three products of elements
in $\B^\up(A_q(\g))$ or $\B^\up(A_q(\n))$.
In this subsection, we introduce another equation among the three products
of elements in $\B^\up(A_q(\g))$, called {\em a generalized $T$-system}.

\begin{proposition} \label{prop: g T-system D}
Let $\mu \in W \varpi_i$ and set $b = \oi_{\varpi_i}(u_\mu) \in B(\infty)$. Then we have
\begin{equation}\label{eq: G t-sys}
\begin{aligned}
\Delta(\mu,s_i\varpi_i)\Delta(\varpi_i,\varpi_i) & = q^{-1}
G^\up\left( \oi^{-1}_{\varpi_i,\varpi_i}(u_\mu \otimes u_{\varpi_i})
\otimes \big( \oi^{-1}_{\varpi_i,\varpi_i}(u_{s_i\varpi_i} \otimes u_{\varpi_i}) \big)^\ri \right) \\
& \hspace{10ex} +G^\up \left( \oi_{\varpi_i+s_i\varpi_i}^{-1}(\widetilde{e^*_i}b) \otimes u^\ri_{\varpi_i+s_i\varpi_i} \right).
\end{aligned}
\end{equation}
\end{proposition}

Note that if $\mu = \varpi_i$, then $b=1$ and the last term in \eqref{eq: G t-sys} vanishes. If $\mu \ne \varpi_i$, then
$\ve_i^*(b)=1$ \ and $\oi_{\varpi_i+s_i\varpi_i}^{-1}(\te_i^* b) \in B(\varpi_i+s_i\varpi_i)$, and $u_\mu \otimes u_{\varpi_i} \in
\oi_{\varpi_i,\varpi_i}B(2\varpi_i)$.

\begin{proof}
In the sequel, we omit $\oi^{-1}_{\varpi_i,\varpi_i}$ for the sake of simplicity. Set
$$  u = \Delta(\mu,s_i\varpi_i)\Delta(\varpi_i,\varpi_i)- q^{-1}
G^\up\left( (u_\mu \otimes u_{\varpi_i}) \otimes
(u_{s_i\varpi_i} \otimes u_{\varpi_i})^\ri  \right).$$
 Then $\wt_\ri(u)=\lambda\seteq\varpi_i+s_i\varpi_i$.

It is obvious that we have $uf_j=0$ for $j \ne i$. Since $\te_i(u_{s_i\varpi_i} \otimes u_{\varpi_i})= u_{\varpi_i} \otimes u_{\varpi_i}$,
we have
\begin{align*}
G^\up\left( (u_\mu \otimes u_{\varpi_i}) \otimes (u_{s_i\varpi_i} \otimes u_{\varpi_i})^\ri  \right)f_i  & =
G^\up\left( (u_\mu \otimes u_{\varpi_i}) \otimes (u_{\varpi_i} \otimes u_{\varpi_i})^\ri  \right) \\
& = G^\up(u_\mu \otimes u^\ri_{\varpi_i})G^\up(u_{\varpi_i} \otimes u^\ri_{\varpi_i}) \\
& = \Delta(\mu,\varpi_i)\Delta(\varpi_i,\varpi_i).
\end{align*}
Here the second equality follows from Corollary~\ref{lem:Deprod}.
On the other hand, we have
\begin{align*}
\left( \Delta(\mu,s_i\varpi_i)\Delta(\varpi_i,\varpi_i)   \right) f_i & =  \left( \Delta(\mu,s_i\varpi_i)f_i \right)
\left(\Delta(\varpi_i,\varpi_i)t_i^{-1}\right) \\ & = q_i^{-1}\Delta(\mu,\varpi_i)\Delta(\varpi_i,\varpi_i).
\end{align*}
Hence we have $uf_i=0$. Thus, $u$ is a lowest weight vector of wight $\lambda$ with respect to the right action of
$U_q(\g)$. Therefore there exists some $v \in V(\lambda)$ such that
$$u=\Phi(v \otimes u_\lambda^\ri).$$
Hence we have $p_\n(u)=\iota_\lambda(v) \in A_q(\n)$. On the other hand, we have
\begin{align*}
p_\n\left( \Delta(\mu,s_i\varpi_i)\Delta(\varpi_i,\varpi_i) \right) & =
p_\n\left(\Delta(\mu,s_i\varpi_i) \right) p_\n\left(\Delta(\varpi_i,\varpi_i)\right) \\
& = \D(\mu,s_i\varpi_i)=G^{\up}(\te_i^*b) \\
& = \iota_\lambda\left(G_\lambda^\up \big(\bio_\lambda^{-1}(\te_i^* b) \big) \right)
\end{align*}
Note that since $\ve_i^*(\te_i^*b) =0$ and $\ve_j^*(\te_i^*b) \le -\lan h_j,\alpha_i \ra$ for $j \ne i$, we have
$\te_i^*b \in \oi_\lambda(B(\lambda))$.

Hence in order to prove our assertion, it is enough to show that
$$
p_\n\left( G^\up\big( (u_\mu \otimes u_{\varpi_i}) \otimes (u_{s_i\varpi_i} \otimes u_{\varpi_i})^\ri  \big) \right) =0.
$$
This follows from Proposition \ref{prop: nonzero elt} and
\begin{align} \label{eq: kernel}
\oi_\g\left( (u_\mu \otimes u_{\varpi_i}) \otimes (u_{s_i\varpi_i} \otimes u_{\varpi_i})^\ri \right) = b \otimes t_\lambda \otimes \te_ib_{-\infty}.
\end{align}

Let us prove \eqref{eq: kernel}. Since
$$
(u_\mu \otimes u_{\varpi_i}) \otimes (u_{s_i\varpi_i} \otimes u_{\varpi_i})^\ri
= \te_i^*
\bl(u_\mu \otimes u_{\varpi_i}) \otimes (u_{\varpi_i} \otimes u_{\varpi_i})^\ri\br,
$$
the left-hand side of \eqref{eq: kernel} is equal to
$$
\te_i^*\left(\oi_\g\big( (u_\mu \otimes u_{\varpi_i}) \otimes (u_{\varpi_i} \otimes u_{\varpi_i})^\ri \big)\right)
= \te_i^*(b \otimes t_{2\varpi_i}\otimes b_{-\infty}).
$$

Since $\ve_i^*(b)=1 < \lan h_i,2\varpi_i \ra=2$, we obtain
$$ \te_i^*\left(b \otimes t_{2\varpi_i}\otimes b_{-\infty}\right)
= b \otimes t_{2\varpi_i-\al_i}\otimes \te_i^* b_{-\infty}
=b \otimes t_\lambda \otimes \te_ib_{-\infty}.$$

\end{proof}

\section{KLR algebras and their modules}

\subsection{KLR algebras and $R$-matrices} In this subsection,
we briefly recall the basic materials of
symmetric KLR algebras and $R$-matrices following \cite{KKKO15}.
For precise definitions in this subsection, we refer
\cite{KKK13,KKKO14,KKKO15}.
In this paper, for the sake of simplicity,
KLR {\em algebras are always assumed to be symmetric}
 although some of results in this paper hold also in the non-symmetric case.
 We  consider only graded modules over KLR algebras.
Hence {\em we write ``a  module'' instead of ``a graded module''.}
{\em We sometimes omit the grading shift if there is no afraid of confusion.}

\medskip

Let the quintuple $(\cmA, \wl, \Pi, \wl^{\vee}, \Pi^{\vee})$ be a symmetric
Cartan datum and let $\ko$ be a base field.
Hence we assume that the generalized Cartan matrix
$A=(a_{i,j})_{i,j\in I}$ is equal to $\bl(\al_i,\al_j)\br_{i,j\in I}$.
For $i,j\in I$ such that $i \ne j$, let
us take  a family of polynomials $(Q_{i,j})_{i,j\in I}$ in $\ko[u,v]$
such that $Q_{i,i}=0$ and
\begin{equation} \label{eq:Q}
Q_{i,j}(u,v) =c_{i,j}(u-v)^{-a_{i,j}} \quad\text{for $i\not=j$}
\end{equation}
with $c_{i,j}\in \ko^{\times}$. We assume
$$Q_{i,j}(u,v)=Q_{j,i}(v,u).$$

For $n \in \Z_{\ge 0}$ and $\beta \in \rl^+$ such that $|\beta| = n$, we set
$$I^{\beta} = \{\nu = (\nu_1, \ldots, \nu_n) \in I^{n} \ | \ \alpha_{\nu_1} + \cdots + \alpha_{\nu_n} = \beta \}.$$

For $\beta \in \rl^+$, we denote by $R(\beta)$ the KLR algebra at $\beta$ associated
with $(\cmA,\wl, \Pi,\wl^{\vee},\Pi^{\vee})$ and
$(Q_{i,j})_{i,j \in I}$. It is a $\Z$-graded $\ko$-algebra generated by
the generators $\{ e(\nu) \}_{\nu \in  I^{\beta}}$, $ \{x_k \}_{1 \le
k \le n}$, $\{ \tau_m \}_{1 \le m \le n-1}$ with certain defining relations
(e.g., see \cite[Definition 1.2]{KKKO15}). Note that
the isomorphism class of $R(\beta)$ does not depend on the choice of $c_{i,j}$'s.

\medskip
We denote by $R(\beta) \Mod$
the category of $R(\beta)$-modules and by $R(\beta)\smod$
the category of $R(\beta)$-modules $M$ which are finite-dimensional over $\ko$ and the action of $x_k$ on $M$ is  nilpotent for any
$k$.

Similarly,  we  also denote by $R(\beta) \gMod$ and by $R(\beta)\gmod$
the category of graded $R(\beta)$-modules and
the category of graded $R(\beta)$-modules finite-dimensional over $\ko$, respectively.

For $\beta \in \rl^+$ and $M \in R(\beta) \Mod$, we set
$$  \wt(M)=-\beta.$$
We set $$R \gmod=\soplus_{\beta\in\rtl^+}R(\beta)\gmod.$$
Then $R\gmod$ is a $\ko$-linear abelian monoidal category,
whose tensor product is given by the {\em convolution product} $\conv$.

Let us denote by $K(R\gmod)$
the  Grothendieck group of $R\gmod$.
Then $K(R\gmod)$ becomes a $\A$-algebra whose
multiplication is induced by the convolution product
and the $\A$-action is induced by the {\em grading shift functor} $q$.

\medskip
For $M \in R( \beta) \smod$, the dual space
$M^* \seteq  \Hom_{\ko}(M, \ko)$
admits an $R(\beta)$-module structure via
\begin{align*}
(r \cdot  f)(u) \seteq f(\psi(r) u) \quad (r \in R( \beta), \ u \in M),
\end{align*}
where $\psi$ denotes the $\ko$-algebra anti-involution on $R(\beta
)$ which fixes the generators $e(\nu)$, $x_m$ and $\tau_k$ for $\nu
\in I^{\beta}, 1 \leq m \leq  |\beta|$ and $1 \leq k \leq
|\beta|-1$.

It is known that (see \cite[Theorem 2.2 (2)]{LV11})
\begin{align}
&(M_1 \circ M_2)^* \simeq q^{(\wt(M_1),\wt(M_2))}
(M_2^* \circ M_1 ^*) \qquad \text{for } M_1,M_2 \in R\gmod. \label{eq:dualconv}
\end{align}

A simple module $M$ is called \emph{self-dual} if $M^* \simeq M$.
Every simple module is isomorphic to a grading shift of a self-dual simple module (\cite[\S.3.2]{KL09}).

\begin{theorem}[\cite{KL09, R11}]  \label{Thm:categorification}
There exists a $\A$-algebra isomorphism
\begin{align} \label{eq: ch}
{\rm ch} \cl   \soplus_{\beta \in \rl^+} K(R(\beta) \gmod) \isoto A_q(\n)_{\Z[q^{\pm1}]}.
\end{align}
\end{theorem}

\begin{theorem} [\cite{R11,VV09}] \label{thm:categorification 2}
We assume further that $\ko$ has characteristic $0$.
Then under the isomorphism ${\rm ch}$ in \eqref{eq: ch}, the upper global basis $\B^\up(\An)$ corresponds to
the set of the isomorphism classes of self-dual simple $R$-modules.
\end{theorem}

\medskip
Let $z$ be an indeterminate  which is  homogeneous of degree $2$, and
let $\psi_z$ be the algebra homomorphism
\begin{align*}
&\psi_z\colon R( \beta )\to \ko[z]\tens R( \beta )
\end{align*}
given by
$$\psi_z(x_k)=x_k+z,\quad\psi_z(\tau_k)=\tau_k, \quad\psi_z(e(\nu))=e(\nu).$$

For an $R( \beta )$-module $M$, we denote by $M_z$ the
$\big(\ko[z]\tens R( \beta )\big)$-module $\ko[z]\tens M$ with the
action of $R( \beta )$ twisted by $\psi_z$.

For a non-zero $R(\beta)$-module $M$ and a non-zero $R(\gamma)$-module $N$,
\begin{eqnarray}&&\parbox{70ex}{%
let $s$ be the order of zero of $R_{M_z,N}\colon  M_z\conv
N\longrightarrow N\conv M_z$;
i.e., the largest non-negative integer such that the image of
$R_{M_z,N}$ is contained in $z^s N\conv M_z$,}
\label{def:s} \end{eqnarray}
where $R_{M,N}$ is the R-matrix from $M \conv N$ to $N \conv M$ constructed in \cite{KKK13}.

\begin{definition} For a non-zero $R(\beta)$-module $M$ and a non-zero $R(\gamma)$-module
$N$,
\begin{align*}
R^{{\rm ren}}_{M,N} \colon M_z\conv N\to N\conv M_z \quad \text{ and }\quad   \rmat{M,N}\colon M\conv N\to N\conv M
\end{align*}
by \begin{align*}
R^{{\rm ren}}_{M,N} = z^{-s}R_{M_z,N} \quad  \text{ and } \quad  \rmat{M,N} = \big( z^{-s}R_{M_z,N}\big)\vert_{z=0},
\end{align*}
where $s$ is the integer in \eqref{def:s}.
\end{definition}

For $M,N \in R \gmod$, $\La(M,N)$ denotes the homogeneous degree of the morphisms $R^{{\rm ren}}_{M_z,N}$ and $\rmat{M,N}$.
Hence we have a morphism in $R\gmod$:
$$q^{\La(M,N)}M\conv N\To N\conv M.$$

\medskip

We say that a simple $R$-module $M$ is {\em real} 
if $M \conv M$ is simple again.

\medskip

The following theorem is the main result of \cite{KKKO14}
\begin{theorem}[{\cite[Theorem 3.2]{KKKO14}}]
Let $M$ and $N$ be simple modules.
We assume that one of them is real. Then we have:
\bnum
\item $M \conv N$ has a simple head and a simple socle,
\item ${\rm Im}(\rmat{M \circ N})$ coincides with the head of $M \conv N$ and the socle of $N \conv M$ \ro up to grading shifts\rf.
\end{enumerate}
\end{theorem}

We also have the following
\Prop\label{prop:rmat}
 Let $M$ and $N$ be simple modules.
We assume that one of them is real.
Then we have
$$\Hom_{R\smod}(M\conv N,N\conv M)=\cor\,\rmat{M,N}.$$
\enprop
\Proof
 Since the other case can be proved similarly, we assume that $M$ is real.
Let $f\cl M\conv N\to N\conv M$ be a morphism.
Then we have a commutative diagram (up to a constant multiple)

$$\xymatrix@C=10ex{
M\conv M\conv N\ar[r]^{M\circ \rmat{M,N}}\ar[d]_{M\circ f}
&M\conv N\conv M\ar[d]_{f\circ M}\\
M\conv N\conv M\ar[r]^{\rmat{M,N}\circ M}
&N\conv M\conv M
}$$
Note that $\rmat{M,\,M\circ N}=M\conv \rmat{M,N}$
and $\rmat{M,\,N\circ M}=\rmat{M,N}\conv M$.
Hence we have
$$M\conv {\rm Im}(\rmat{M,N})\subset f^{-1}\bl {\rm Im}(\rmat{M,N})\br\conv M.$$
Hence there exists a submodule $K$ of $N$ such that
${\rm Im}(\rmat{M,N})\subset K\conv M$ and
$M\conv K\subset f^{-1}\bl{\rm Im}(\rmat{M,N})\br$.
Since $K\not=0$, we have $K=N$.
Hence
$f(M\conv N)\subset {\rm Im}(\rmat{M,N})$, which means that
$f$ factors as
$M\conv N\to \soc(N\conv M)\monoto N\conv M$.
It remains to remark that
$\Hom_{R\smod}\bl M\conv N,\soc(N\conv M)\br=\cor\,\rmat{M,N}$.
\QED

For simple modules $M$ $N$, let us denote by $M\hconv N$
the head of $M\conv N$.

\subsection{Properties of $\tLa(M,N)$ and $\de(M,N)$}

\begin{lemma} [{\cite[Lemma 2.9]{KKKO15}}]
Let $M$ and $N$ be non-zero modules.
Then we have
\begin{enumerate}
\item[{\rm (1)}] $\Lambda(M,N)+\Lambda(N,M) \in 2 \Z_{\ge 0}.$
\item[{\rm (2)}] Setting $\Lambda(M,N)+\Lambda(N,M) =2m$
with $m \in \Z_{\ge 0}$, we have
$$
R^{{\rm ren}}_{M_z,N} \circ R^{{\rm ren}}_{N,M_z}=z^m {\rm id}_{N \conv M_z} \quad
\text{and} \quad
R^{{\rm ren}}_{N,M_z} \circ R^{{\rm ren}}_{M_z,N}=z^m {\rm id}_{M_z \conv N}
$$  up to constant multiples.
\end{enumerate}
\end{lemma}

By \cite[Lemma 2.5]{KKKO15} and the above lemma, we can define the following
integers:

\begin{definition}[\cite{KKKO15}]
Let  $M$ and $N$ be non-zero modules.
\eqn
&&\tLa(M,N)\seteq\big((\wt(M),\wt(N))+\La(M,N) \big)/2 \in \Z,\\
&&  \de(M,N) \seteq \big( \La(M,N)+\La(N,M) \big)/2 \in \Z_{\ge 0}.
\eneqn
\end{definition}

\begin{lemma} [\cite{KKKO15}] \label{lem: d=0}
Let $M$ and $N$ be simple modules. Assume that one of them is real. Then the following conditions are equivalent:
\begin{itemize}
\item[{\rm (i)}] $\de(M,N)=0$,
\item[{\rm (ii)}] $\rmat{M,N}$ and $\rmat{N,M}$ are inverse to each other up to a constant multiple,
\item[{\rm (iii)}] $M \conv N \simeq N \conv M$ up to a grading shift,
\item[{\rm (iv)}] $M \hconv N \simeq N \hconv M$ up to a grading shift,
\item[{\rm (v)}] $M \conv N$ is simple.
\end{itemize}
\end{lemma}

\begin{definition} Let $M$ and $N$ be simple modules.
\begin{itemize}
\item[{\rm (i)}] We say that $M$ and $N$ \emph{commute} if $\de(M,N)=0$.
\item[{\rm (ii)}] We say that $M$ and $N$ are \emph{simply linked} if $\de(M,N)=1$.
\end{itemize}
\end{definition}

\begin{definition}
Let $M_1,\ldots,M_m$ be simple modules. Assume that they commute with each other. We set
\eqn
&&M_1\sodot M_2\seteq q^{\tLa(M_1,M_2)}M_1\conv M_2,\\
&&\sodot_{1 \le k \le m}M_k
\seteq(\cdots(M_1\sodot M_2)\cdots)\sodot M_{m-1})\sodot M_m\\
&&\hs{20ex}\simeq q^{\sum_{1 \le i <j\le m} \tLa(M_i,M_j)}M_1 \conv \cdots \conv M_m.
\eneqn
\end{definition}
It is invariant by the permutations of $M_1,\ldots,M_m$.

\begin{lemma} [{\cite[Lemma 2.15]{KKKO15}}] \label{lem: tLa}
Let $M_1,\ldots,M_m$ be real simple modules commuting with each other. Then for any $\sigma \in \mathfrak{S}_m$, we have
$$\sodot_{1 \le k \le m}M_k \simeq \sodot_{1 \le k \le m}M_{\sigma(k)} \quad \text{ in } R \gmod.$$
Moreover, if $M_k$'s are self-dual, so is $\sodot_{1 \le k \le m}M_k$.
\end{lemma}

\Lemma\label{lem:MN}
Let $\{M_i\}_{1\le i\le n}$ and $\{N_i\}_{1\le i\le n}$ be a pair of
commuting families of real simple modules.
We assume that
\bnam
 \item $\{M_i\hconv N_i\}_{1\le i\le n}$ is a commuting family of real simple
modules,
\item $M_i\hconv N_i$ commutes with $N_{j}$
for any $1\le i,j\le n$.
\ee
Then we have
$$(\conv_{1\le i\le n}M_i)\hconv (\conv_{1\le j\le n}N_j)
\simeq \conv_{1\le i\le n}(M_i\hconv N_i)\quad
\text{up to a grading shift}
$$
\enlemma
 Since the proof is elementary and similar to the proof of
\cite[Lemma 2.23]{KKKO15}, we omit it.

\begin{proposition} \label{prop:3simple}
Let $L$, $M$ and  $N$ be simple modules.
We assume that $L$ is real and
one of $M$ and $N$ is real.
\begin{enumerate}
\item[{\rm (i)}]
If $\La(L,M\hconv N)=\La(L,M)+\La(L,N)$,
then
$L\conv M\conv N$ has a simple head and
$N\conv M\conv L$ has a simple socle.
\item[{\rm (ii)}]
If $\La(M\hconv N,L)=\La(M,L)+\La(N,L)$,
then
$M\conv N\conv L$ has a simple head and
$L\conv N\conv M$ has a simple socle.
\item[{\rm (iii)}]
If $\de(L,M\hconv N)=\de(L,M)+\de(L,N)$,
then
$L\conv M\conv N$ and $M\conv N\conv L$  have a simple head, and
$N\conv M\conv L$ and $L\conv N\conv M$ have a simple socle.
\end{enumerate}
\end{proposition}

\begin{proof}
(i)\ Denote $k=\La(L,M\hconv N)=\La(L,M)+\La(M,N)$ and $m=\La( M,N)$. Then the diagram
$$\xymatrix@C=12ex@R=3ex
{{L\conv M\conv N}\ar[r]^{\rmat{L,M\circ N}}\ar@{->>}[d]
& q^{-k}{M\conv N\conv L}\ar@{->>}[d]
\\
\wb{L\conv (M\hconv N)}\ar[r]^{\rmat{L,M\shconv N}}\ar@{>->}[d]
&\wb{q^{-k}(M\hconv N)\conv L}\ar@{>->}[d]\\
q^{-m}L\conv N\conv M\ar[r]^{\rmat{L,N\circ M}}&q^{-k-m} N\conv M\conv L
}
$$
commutes. Hence \cite[Proposition 2.1, Proposition 2.2]{KKKO15}
implies that $L\conv M\conv N$ has a simple head and $N\conv M\conv L$ has a simple socle.
(ii) are proved similarly.

\medskip\noindent
(iii)\ If $\de(L,M\hconv N)=\de(L,M)+\de(L,N)$,
the we have $\La(L,M\hconv N)=\La(L,M)+\La(L,N)$
and $\La(M\hconv N,L)=\La(M,L)+\La(N,L)$. Thus the third assertion follows from the first and second assertion.
\end{proof}

\Prop\label{Prop: l2}
Let $M$ and $N$ be simple modules.
Assume that one of them is real and $\de(M,N)=1$.
Then we have an exact sequence
\eqn
&&0\to q^{\La(N,M)}N\hconv M\to M\conv N\to M\hconv N\to 0.
\eneqn
In particular, $M\conv N$ has length $2$.
\enprop
\Proof In the course of the proof, we ignore the grading.

Set $X=M_z\circ N$
and $Y=N\circ M_z$.
By $\Rm_{N,M_z}\cl Y\monoto X$ let us regard $Y$ as a submodule of $X$.
By the condition, we have
$\Rm_{N,M_z}\circ \Rm_{M_z,N}=z\id_X$ up to a constant multiple,
and hence we have
$$zX\subset Y\subset X.$$

We have an exact sequence
$$ 0\To \dfrac{Y}{zX}\To\dfrac{X}{zX}\To\dfrac{X}{Y}\To0.$$
Since
$$M\conv N\simeq\dfrac{X}{zX}\epito \dfrac{X}{Y}\monoto \dfrac{z^{-1}Y}{Y}
\simeq N\conv M,$$
we have $\dfrac{X}{Y}\simeq M\hconv N$.
Similarly
$$N\conv M\simeq\dfrac{Y}{zY}\epito \dfrac{Y}{zX}\monoto \dfrac{X}{zX}
\simeq M\conv N$$
implies that
$\dfrac{Y}{zX}\simeq N\hconv M$.
\QED

\Lemma\label{lem:short}
Let $M$ and $N$ be simple modules.
Assume that one of them is real.
If the equation
$$[M][N]=q^m[X]+q^n[Y]$$
holds in $K(R\gmod)$ for self-dual simple modules $X$, $Y$ and
integers $m$, $n$ such that $m\ge n$,
then we have
\bnum
\item $\de(M,N)=m-n>0$,
\item there exists an exact sequence
$$0\To q^mX\To M\conv N\To q^nY\To0,$$
\item
 $ q^mX$ is a socle of $M\conv N$ and $q^nY$ is
a
head of
$M\conv N$.
\ee
\enlemma
\Proof
First note that $\de(M,N)>0$ since $M\conv N$ is not simple.
By the assumption
there exists either an exact sequence
$$0\To q^mX\To M\conv N\To q^nY\To0,$$
or
$$0\To q^nY\To M\conv N\To q^mX\To0.$$
The last sequence cannot exist by \cite[Corollary 2.24]{KKKO15}
because $\de(M,N)=n-m\le0$.
Hence the first sequence exists, and the assertions follow from
loc.cit.
\QED

\subsection{Chevalley and Kashiwara operators}
 Let us recall the definition of several functors used to categorify $\Um$ by using KLR algebras.

\begin{definition} Let $\beta\in\rtl^+$.
\bnum
\item For $i \in I$ and $1 \le a \le |\beta|$, set
$$ e_a(i)=\sum_{\nu \in I^\beta,\nu_a=i}e(\nu)\in R(\beta).$$
\item We take conventions:
\begin{align*}
&E_iM=e_1(i)M, \allowdisplaybreaks\\
&E^*_iM=e_{|\beta|}(i)M,
\end{align*}
which are functors from $R(\beta) \gmod$ to $R(\beta-\al_i) \gmod$.
\item For a simple module $M$, we set
\begin{align*}
&\eps_i(M)=\max\set{n\in\Z_{\ge0}}{E_i^nM\not=0},\\
&\eps^*_i(M)=\max\set{n\in\Z_{\ge0}}{E_i^{*\;n}M\not=0},\\
& \widetilde{F}_iM = q^{\ve_i(M)} L(i) \hconv M, \allowdisplaybreaks\\
& \widetilde{F}^*_iM = q^{\ve^*_i(M)} M \hconv L(i), \allowdisplaybreaks\\
& \widetilde{E}_iM = q^{1-\ve_i(M)} \soc(E_iM) \simeq q^{\ve_i(M)-1} \hd(E_iM), \allowdisplaybreaks\\
& \tE^*_iM = q^{1-\ve^*_i(M)} \soc(E^*_iM) \simeq q^{\ve^*_i(M)-1}
\hd(E_i^*M),\\
&\tE_i^{\,\max} M=\tE_i^{\eps_i(M)}M\quad\text{and}\quad
\tE_i^{*\;\max} M=\tE_i^{*\;\eps^*_i(M)}M.
\end{align*}
\item For $i \in I$ and $n \in \Z_{\ge 0}$, we set
$$L(i^n)=q^{n(n-1)/2} \underbrace{L(i)\conv \cdots \conv L(i)}_n.$$
Here $L(i)$ denotes the $R(\al_i)$-module $R(\al_i)/R(\al_i)x_1$.
Then  $L(i^n)$ is a self-dual real simple  $R(n\alpha_i)$-module.
\ee
\end{definition}
In the course of the following propositions, we use the following notations.

\begin{align} \label{eq: overQ}
\overline{Q}_{i,j}(x_a,x_{a+1},x_{a+2}) \seteq
\dfrac{Q_{i,j}(x_a,x_{a+1})-Q_{i,j}(x_{a+2},x_{a+1})}{x_a-x_{a+2}}.
\end{align}
Then we have
$$\tau_{a+1}\tau_a\tau_{a+1}-\tau_{a}\tau_{a+1}\tau_{a}
=\sum_{i,j\in I}\overline{Q}_{i,j}(x_a,x_{a+1},x_{a+2})
e_a(i)e_{a+1}(j)e_{a+2}(i).$$
\begin{proposition} \label{prop: La i}
Let $\beta \in \rl^+$ with $n=|\beta|$. Assume that an $R(\beta)$-module $M$ satisfies $E_iM=0$. Then the morphism $R(\alpha_i) \tens M \Lto
q^{(\alpha_i,\beta)}M \conv R(\alpha_i)$ given by
\begin{align} \label{eq: pre-form}
e(i) \tens u \longmapsto \tau_1 \cdots \tau_n(u \tens e(i))
\end{align}
extends uniquely to an $(R(\alpha_i+\beta),R(\alpha_i))$-bilinear homomorphism
\begin{align} \label{eq: ext-form}
R(\alpha_i) \conv M \Lto q^{(\alpha_i,\beta)}M \conv R(\alpha_i)
\end{align}
\end{proposition}

\begin{proof} {\rm (i)} First note that, for $1 \le a \le n$,
\begin{equation} \label{eq: e_a(i) vanish}
\tau_1 \cdots \tau_{a-1}e_a(i)\tau_{a+1}\cdots\tau_n\big( u \tens e(i) \big) =
\tau_{a+1}\cdots\tau_n\big(e_1(i)\tau_1 \cdots \tau_{a-1}(u \tens e(i))  \big)=0
\end{equation}
since $E_iM =0$. \\
\noindent
{\rm (ii)} In order to see that \eqref{eq: ext-form} is a well-defined $R(\alpha_i+\beta)$-linear homomorphism, it is enough
to show that \eqref{eq: pre-form} is $R(\beta)$-linear. \\
\noindent
{\rm (a)} Commutation with $x_a\in R(\beta)$ $(1\le a \le n)$: We have
\begin{align*}
x_{a+1}\tau_1 \cdots \tau_n\big(u \tens e(i)\big) &=  \tau_1 \cdots \tau_{a-1}x_{a+1}\tau_{a} \cdots \tau_n\big(u \tens e(i)\big) \allowdisplaybreaks \\
 &=  \tau_1 \cdots \tau_{a-1}\big( \tau_{a} x_{a}+ e_a(i) \big) \tau_{a+1}\cdots \tau_n\big(u \tens e(i) \big) \allowdisplaybreaks \\
 &=  \tau_1 \cdots  \tau_n  x_{a}\big(u \tens e(i)\big)
\end{align*}
by \eqref{eq: e_a(i) vanish}.  \\
\noindent
{\rm (b)} Commutation with $\tau_a\in R(\beta)$ $(1\le a < n)$:
 Then we have
\begin{align*}
&\tau_{a+1}\tau_1 \cdots \tau_n\big(u \tens e(i)\big) \\
& \hspace{5ex} =  \tau_1 \cdots \tau_{a-1}(\tau_{a+1}\tau_{a}\tau_{a+1})\tau_{a+2} \cdots \tau_n\big(u \tens e(i)\big) \allowdisplaybreaks \\
& \hspace{5ex} =  \tau_1 \cdots \tau_{a-1}\big(\tau_{a}\tau_{a+1}\tau_{a}+\sum_{j}\overline{Q}_{i,j}(x_{a},x_{a+1},x_{a+2})e_a(i)e_{a+1}(j) \big)\tau_{a+2}\cdots \tau_n\big(u \tens e(i) \big) \allowdisplaybreaks \\
& \hspace{5ex} =  \tau_1 \cdots  \tau_n  \tau_{a}\big(u \tens e(i)\big) \allowdisplaybreaks \\
& \hspace{10ex} + \sum_{j}\tau_1 \cdots \tau_{a-1}\overline{Q}_{i,j}(x_{a},x_{a+1},x_{a+2})e_a(i)e_{a+1}(j)\tau_{a+2}\cdots \tau_n\big(u \tens e(i) \big).
\end{align*}
The last term vanishes because $E_iM=0$ implies that
\begin{align*}
& \tau_1 \cdots \tau_{a-1}f(x_{a},x_{a+1})g(x_{a+2})e_a(i)\tau_{a+2}\cdots \tau_n\big(u \tens e(i) \big) \\
& \hspace{18ex} = g(x_{a+2})\tau_{a+2}\cdots \tau_n e_1(i)\tau_1 \cdots \tau_{a-1}f(x_{a},x_{a+1})\big(u \tens e(i) \big)=0
\end{align*}
for any polynomial $f(x_{a},x_{a+1})$ and $g(x_{a+2})$. \\
\noindent
{\rm (iii)} Now let us show that \eqref{eq: ext-form} is right $R(\alpha_i)$-linear. By \eqref{eq: e_a(i) vanish}, we have
\begin{align*}
 \tau_1 \cdots \tau_{a-1} x_a \tau_a \cdots \tau_{n}\big(u \tens e(i) \big) & =
\tau_1 \cdots \tau_{a-1} \big(\tau_a x_{a+1}-e_{a}(i) \big)\tau_{a+1}\cdots \tau_{n} \big(u \tens e(i) \big) \\
& = \tau_1 \cdots \tau_{a} x_{a+1} \tau_{a+1} \cdots \tau_{n}\big(u \tens e(i) \big)
\end{align*}
for $1 \le a \le n$. Therefore we have
$$ x_{1}\tau_1 \cdots \tau_n\big(u \tens e(i)\big) = \tau_1 \cdots \tau_n x_{n+1}\big(u \tens e(i)\big)=\tau_1 \cdots \tau_n \big(u \tens e(i)x_{1}\big).$$
\end{proof}

For $m,n\in\Z_{\ge0}$,
let us denote by $w[{m,n}]$ the element of $\sym_{m+n}$  defined by
\eq
&&w[{m,n}](k)=\begin{cases}k+n&\text{if $1\le k\le m$,}\\
k-m&\text{if $m<k\le m+n$}.\end{cases}
\eneq
Set $\tau_{w[{m,n}]}:=\tau_{i_1} \cdots \tau_{i_r}$, where $s_{i_1} \cdots s_{i_r}$ is a reduced expression
of $w[{m,n}]$.
Note that $\tau_{w[{m,n}]}$ does not depend on the choice of reduced expression (\cite[Corollary 1.4.3]{KKK13}).

\begin{proposition} Let $M \in R(\beta) \gmod$  and $N \in R(\gamma) \gmod$, and set $m=|\beta|$ and
$n=|\gamma|$. If $E_iM=0$ for any $i \in \supp(\gamma)$, then
\begin{align} \label{eq: pre-form beta}
v \tens u \longmapsto \tau_{w[m,n]}(u \tens v)
\end{align}
gives a well-defined $R(\beta+\gamma)$-linear homomorphism $N \conv M \Lto q^{(\beta,\gamma)} M \conv N$.
\end{proposition}

\begin{proof} The proceeding proposition implies that
$$ v \tens u \longmapsto \tau_{w[m,n]}(u \tens v) \quad \text{ for } u \in M, v \in R(\gamma)$$
gives a well-defined $R(\beta+\gamma)$-linear homomorphism $R(\gamma) \conv M \to M \conv R(\gamma)$. Hence
it is enough to show that it is right $R(\gamma)$-linear. Since we know that it commutes with
the right multiplication of $x_k$, it is enough to show that it commutes with the right
multiplication of $\tau_k$. For this, we may assume that $n = 2$ and $k = 1$. Set $\gamma = \alpha_i +\alpha_j$.

Thus we have reduced the problem to  the equality 
\begin{align} \label{eq: reduced}
\tau_1(\tau_2\tau_1)\cdots(\tau_{m+1}\tau_m)\big( u \tens e(i) \tens e(j) \big)
= (\tau_2\tau_1)\cdots(\tau_{m+1}\tau_m) \tau_{m+1}\big( u \tens e(i) \tens e(j) \big)
\end{align}
for $u \in M$. It is a consequence of
\begin{equation} \label{eq: conse}
\begin{aligned}
& (\tau_2\tau_1)\cdots(\tau_{a}\tau_{a-1})\tau_a(\tau_{a+1}\tau_{a})\cdots(\tau_{m+1}\tau_m)\big( u \tens e(i) \tens e(j) \big) \allowdisplaybreaks \\
& \hs{5ex} = (\tau_2\tau_1)\cdots(\tau_{a+1}\tau_{a})
\tau_{a+1}(\tau_{a+2}\tau_{a+1})\cdots(\tau_{m+1}\tau_m)\big( u \tens e(i) \tens e(j) \big)
\end{aligned}
\end{equation}
for $1 \le a \le m$. Note that
\begin{align*}
& \tau_a(\tau_{a+1}\tau_{a})\cdots(\tau_{m+1}\tau_m)\big( u \tens e(i) \tens e(j) \big) \allowdisplaybreaks \\
& \hspace{10ex} = \tau_a(\tau_{a+1}\tau_{a})e_{a+1}(i)e_{a+2}(j)(\tau_{a+2}\tau_{a+1})\cdots(\tau_{m+1}\tau_m)\big( u \tens e(i) \tens e(j) \big)
\end{align*}
and
\begin{align*}
& \tau_a(\tau_{a+1}\tau_{a})e_{a+1}(i)e_{a+2}(j) \allowdisplaybreaks \\
& \hspace{10ex} = (\tau_{a+1}\tau_{a})\tau_{a+1} e_{a+1}(i)e_{a+2}(j)-\overline{Q}_{ji}(x_{a},x_{a+1},x_{a+2})e_{a}(j)e_{a+1}(i)e_{a+2}(j).
\end{align*}
Hence it is enough to show
\begin{align*}
& (\tau_2\tau_1)\cdots(\tau_{a}\tau_{a-1})\overline{Q}_{j,i}(x_{a},x_{a+1},x_{a+2})
e_a(j)
\allowdisplaybreaks\\
& \hspace{20ex}(\tau_{a+2}\tau_{a+1})\cdots (\tau_{m+1}\tau_{m})\big( u \tens e(i) \tens e(j) \big)=0.
\end{align*}
This follows from
\begin{align*}
&(\tau_2\tau_1)\cdots(\tau_{a}\tau_{a-1})f(x_a)g(x_{a+1},x_{a+2}) e_a(j)
(\tau_{a+2}\tau_{a+1})\cdots(\tau_{m+1}\tau_m)
\big( u\tens e(i)\tens e(j)\big)\allowdisplaybreaks\\
& \hspace{10ex}=(\tau_2\cdots\tau_a)
(\tau_1\cdots\tau_{a-1})f(x_a)g(x_{a+1},x_{a+2}) e_a(j) \allowdisplaybreaks\\
&\hspace{35ex}(\tau_{a+2}\tau_{a+1})\cdots(\tau_{m+1}\tau_m)
\big( u\tens e(i)\tens e(j)\big)\allowdisplaybreaks\\
&\hspace{10ex}=(\tau_2\cdots\tau_a)
g(x_{a+1},x_{a+2})(\tau_{a+2}\tau_{a+1})\cdots(\tau_{m+1}\tau_m)\allowdisplaybreaks\\
&\hspace{35ex}
e_1(j)(\tau_1\cdots\tau_{a-1})f(x_a)\big( u\tens e(i)\tens e(j)\big)\allowdisplaybreaks\\
&\hspace{10ex}=0
\end{align*}
for $1 \le a \le m$ and $f(x_a) \in \ko[x_{a}]$, $g(x_{a+1},x_{a+2}) \in \ko[x_{a+1},x_{a+2}]$.
\end{proof}

\Cor
Let $i\in I$ and $M$ a simple module.
Then we have
\eqn
&&\tLa\bl L(i),M\br=\eps_i(M),\\
&&\La\bl L(i),M\br=2\eps_i(M)+\lan h_i,\wt(M)\ran=\eps_i(M)+\vphi_i(M).\\
\eneqn
\encor
\Proof
Set $n=\eps_i(M)$ and $M_0=E_i^{(n)}(M)$.
Then the preceding proposition implies
$\La(L(i),M_0)=\bl \al_i,\wt(M_0)\br$.
Hence we have $\tLa(L(i),M_0)=0$, which implies
$$\tLa(L(i),M)=\tLa(L(i), L(i^n)\circ M_0)
=\tLa\bl L(i), L(i^n)\br+\tLa(L(i),M_0)=n.$$
\QED

Let $P(i^n)$ be a projective cover of
$L(i^n)$.
The functor
$$E_i^{(n)} : R(\beta) \Mod \to R(\beta - n \alpha_i) \Mod$$
is defined by
$$E_i^{(n)}(M) := P(i^n)^\psi \tens_{R(n\al_i)}E_i^nM, $$
where
$P(i^n)^\psi $ denotes the right $R(n\al_i)$-module obtained from the left
$R(\beta)$-module $P(i^n)$  via the anti-automorphism $\psi$.
Similarly we define $E_i^{*\,(n)}$.
We have $$E_i^n\simeq [n]! E_i^{(n)}.$$

\begin{proposition}\label{prop: Econv} Let $M$ and $N$ be modules,
and $m,n \in \Z_{\ge 0}$,
\begin{enumerate}
\item[{\rm (i)}] If $E_i^{m+1}M=0$ and $E_i^{n+1}N=0$, then we have
\begin{align} \label{eq: shuffle 1}
E_i^{(m+n)} (M \conv N) \simeq q^{mn+n \lan h_i, \wt(M) \ran}
 E_i^{(m)} M \conv  E_i^{(n)}N.
\end{align}
\item[{\rm (ii)}] If $E_i^{*\,m+1}M=0$ and $E_i^{*\,n+1}N=0$, then we have
\begin{align} \label{eq: shuffle 2}
E_i^{*\,(m+n)} (M \conv N) \simeq
q^{mn+m \lan h_i, \wt(N) \ran} E_i^{*\,(m)} M \conv E_i^{*\,(n)} N.
\end{align}
\end{enumerate}
\end{proposition}

\begin{proof}
It follows from the shuffle lemma (\cite[Lemma 2.20]{KL09}).
\end{proof}

The following corollary is an immediate consequence of Proposition~\ref{prop: Econv}.
\Cor\label{cor:Ereal} Let $i\in I$ and let $M$ be a real simple module.
Then $\tE_i^{\,\max}M$ is also real simple.
\encor

\begin{proposition} \label{prop: acted on left only}
Let $M$ and $N$ be simple modules.
 We assume that one of them is real.
Assume further that $\ve_i(M \hconv N)=\ve_i(M)$. 
Then we have an isomorphism in $R\gmod$.
$$ \widetilde{E}^{\,\max}_i (M \hconv N) \simeq (\widetilde{E}^{\,\max}_i M) \hconv N.$$
\end{proposition}

\begin{proof}
Set $n=\ve_i(M \hconv N) = \ve_i(M)$ and $M_0=\widetilde{E}^{\,\max}_i M$.
Then $M_0$ or $N$ is real.
Now we have
$$ L(i^n) \otimes M_0 \otimes N \monoto
E_i^n(M\hconv N)\simeq L(i^n) \otimes \widetilde{E}^{\,\max}_i (M \hconv N),$$
which induces a non-zero map $M_0 \otimes N \to
\widetilde{E}^{\,\max}_i (M \hconv N)$. Hence we have a surjective map
$$M_0 \conv N \twoheadrightarrow \widetilde{E}^{\,\max}_i (M \hconv N).$$
Since $M_0$ or $N$ is real by Corollary~\ref{cor:Ereal}, $M_0\conv N$ has a simple head
and we obtain the desired result. 
\end{proof}

\subsection{Determinantial modules and T-systems}
In the rest of this paper,
we assume further that {\em the base field $\ko$ is of characteristic $0$.}
Under this condition, the family of self-dual simple $R$-modules
corresponds to the upper global basis of $\An$
by Theorem \ref{thm:categorification 2}.
\begin{definition}
For $\lambda \in \wl^+$ and $\mu,\zeta \in W \lambda$ such that $\mu \preceq \zeta$, let $\M(\mu,\zeta)$ be a simple $R(\zeta-\mu)$-module
such that $\ch(\M(\mu,\zeta))=\D(\mu,\zeta)$.
\end{definition}
Since $\D(\mu,\zeta)$ is a member of the upper global basis,
such a module exists uniquely due to Theorem \ref{thm:categorification 2}.
The module $\M(\mu,\zeta)$ is self-dual and we call it {\em the determinantial module}. 

\begin{lemma}
$\M(\mu,\zeta)$ is a real simple module.
\end{lemma}

\begin{proof}
It follows from $\ch(\M(\mu,\zeta) \conv \M(\mu,\zeta)) =\ch\big( \M(\mu,\zeta) \big)^2=q^{-(\zeta,\zeta-\mu)}\D(2\mu,2\zeta)$ which is a member of the upper global
basis up to a power of $q$.
Here the last equality follows from Corollary~\ref{cor:Duv}.
\end{proof}

\begin{proposition} \label{prop: commute 1}
Let $\lambda,\mu \in \wl^+$, and $s,s',t,t' \in W$ such that $\ell(s's)=\ell(s')+\ell(s)$,
$\ell(t't)=\ell(t')+\ell(t)$, $s's \lambda \preceq t'\lambda $ and
$s'\mu \preceq t't\mu$. Then
\bnum
\item $\M(s's\lambda,t'\lambda)$ and $\M(s'\mu,t't\mu)$ commute,
\item $\La\bl\M(s's\lambda,t'\lambda),\M(s'\mu,t't\mu)\br
= (s's\lambda+t'\lambda,\;t't\mu-s'\mu)$. 
\item $\tLa\bl\M(s's\lambda,t'\lambda),\M(s'\mu,t't\mu)\br
=(t'\lambda,\;t't\mu-s'\mu)$ and \\
$\tLa\bl\M(s'\mu,t't\mu),\M(s's\lambda,t'\lambda)\br
=(s'\mu-t't\mu,\;s's\lambda)$.
\end{enumerate}
\end{proposition}

\begin{proof}
It is a consequence of Proposition \ref{prop: BZ form}~(ii) and
\cite[Corollary 3.4]{KKKO15}.
\end{proof}

\begin{proposition} \label{prop: weyl left right}
Let $\lambda \in \wl^+$, $\mu,\zeta \in W \lambda$
such that $\mu \preceq \zeta$ and $i \in I$.
\begin{enumerate}
\item[{\rm (i)}] If $n \seteq \lan h_i,\mu \ra \ge 0$, then
$$ \ve_i(\M(\mu,\zeta))=0 \ \ \text{ and } \ \ \M(s_i\mu,\zeta) \simeq \widetilde{F}_i^n\M(\mu,\zeta) \simeq L(i^n) \hconv \M(\mu,\zeta)
\ \text{in $R\gmod$.}$$
\item[{\rm (ii)}] If $\lan h_i,\mu \ra \le 0$ and $s_i\mu \preceq \zeta$, then $\ve_i(\M(\mu,\zeta))=-\lan h_i,\mu \ra$.
\item[{\rm (iii)}] If $m \seteq -\lan h_i,\zeta \ra \le 0$, then
$$ \ve^*_i(\M(\mu,\zeta))=0 \ \ \text{ and } \ \ \M(\mu,s_i\zeta) \simeq
\tF_i^{*\,m}\M(\mu,\zeta) \simeq \M(\mu,\zeta) \hconv L(i^m )
\ \text{ in $R\gmod$.}$$
\item[{\rm (iv)}] If $\lan h_i,\zeta \ra \ge 0$ and $\mu \preceq s_i\zeta$, then $\ve^*_i(\M(\mu,\zeta))=\lan h_i,\zeta \ra$.
\end{enumerate}
\end{proposition}
\Proof
It is a consequence of Lemma~\ref{lem: weyl left right}.
\QED

\begin{proposition}\label{prop:T}
Assume that $u,v \in W$ and $i \in I$ satisfy $u<us_i$ and $v<vs_i \le u$.
\bnum
\item
We have exact sequences
\begin{equation} \label{eq: deteminatial seq1}
\begin{aligned}
& 0 \Lto \M(u\lambda,v\lambda) \Lto q^{(vs_i\varpi_i,v\varpi_i-u\varpi_i)}\M(us_i\varpi_i,vs_i\varpi_i) \conv \M(u\varpi_i,v\varpi_i) \\
& \hspace{14ex} \Lto q^{-1+(v\varpi_i,vs_i\varpi_i-u\varpi_i)}\M(us_i\varpi_i,v\varpi_i) \conv \M(u\varpi_i,vs_i\varpi_i) \Lto 0,
\end{aligned}
\end{equation}
and
\begin{equation} \label{eq: deteminatial seq2}
\begin{aligned}
& 0 \Lto  q^{1+(v\varpi_i,vs_i\varpi_i-u\varpi_i)}  \M(us_i\varpi_i,v\varpi_i) \conv \M(u\varpi_i,vs_i\varpi_i)  \\
& \hspace{3ex} \Lto   q^{(v\varpi_i,vs_i\varpi_i-us_i\varpi_i)}  \M(u\varpi_i,v\varpi_i)\conv \M(us_i\varpi_i,vs_i\varpi_i) \Lto \M(u\lambda,v\lambda)  \Lto 0,
\end{aligned}
\end{equation}
where $\lambda = s_i\varpi_i+\varpi_i = \sum_{j \ne i} -a_{j,i}\varpi_j$.
\item
$\de\bl \M(u\varpi_i,v\varpi_i), \M(us_i\varpi_i,vs_i\varpi_i)\br=1$.
\ee
\end{proposition}

\begin{proof}
Since the proof of \eqref{eq: deteminatial seq1} is similar,
let us only prove \eqref{eq: deteminatial seq2}.
(Indeed, they are dual to each other.)

Set
\eqn
&&X=q^{(v\varpi_i,vs_i\varpi_i-u\varpi_i)}  \M(us_i\varpi_i,v\varpi_i) \conv \M(u\varpi_i,vs_i\varpi_i) ,\\
&&Y=q^{(v\varpi_i,vs_i\varpi_i-us_i\varpi_i)}  \M(u\varpi_i,v\varpi_i)\conv \M(us_i\varpi_i,vs_i\varpi_i),\\
&&Z=\M(u\lambda,v\lambda).
\eneqn
Then Proposition \ref{prop: the ses} implies that
$$\ch(Y)=\ch(qX)+\ch(Z).$$
Since $X$ and $Z$ are simple and self-dual,
the assertion follows from
Lemma~\ref{lem:short}.
\QED

\subsection{Generalized T-system on determinantial modules.}

\begin{theorem} \label{thm: canonical surjection}
Let $\lambda \in \wl^+$ and $\mu_1,\mu_2,\mu_3 \in W \lambda$ such that $\mu_1 \preceq \mu_2 \preceq \mu_3$. Then there exists a canonical epimorphism
$$ \M(\mu_1,\mu_2) \conv \M(\mu_2,\mu_3) \twoheadrightarrow \M(\mu_1,\mu_3),$$
which is equivalent to saying that $\M(\mu_1,\mu_2) \hconv \M(\mu_2,\mu_3) \simeq \M(\mu_1,\mu_3)$.

In particular, we have
$$\tLa(\M(\mu_1,\mu_2),\M(\mu_2,\mu_3))=0\quad\text{and}\quad
\La(\M(\mu_1,\mu_2),\M(\mu_2,\mu_3)) = -(\mu_1-\mu_2,\mu_2-\mu_3). $$
\end{theorem}

\begin{proof}
(a)\ Our assertion follows from Theorem \ref{thm: DDD} and \cite[Theorem 3.6]{KKKO15} when $\mu_3=\lambda$.

\medskip\noi
(b) We shall prove the general case by induction on $|\la-\mu_3|$.
By (a), we may assume that $\mu_3 \ne \lambda$. Then there exists $i$ such that $\lan h_i,\mu_3\ra <0$. The induction hypothesis implies that
$$\M(\mu_1,\mu_2) \hconv \M(\mu_2,s_i \mu_3) \simeq \M(\mu_1,s_i \mu_3).$$
Since $\mu_1\preceq\mu_2\preceq\mu_3\preceq s_i\mu_3$,
Proposition~\ref{prop: weyl left right} (iv) implies that
$$\eps_i^*\bl\M(\mu_2,s_i\mu_3)\br=\eps_i^*\bl\M(\mu_1,s_i\mu_3)\br
=-\lan h_i,\mu_3\ran.$$
Then Proposition~\ref{prop: acted on left only} implies that
$$\tEs_i{}^\max\bl \M(\mu_1,\mu_2)\hconv \M(\mu_2,s_i\mu_3)\br
\simeq
\M(\mu_1,\mu_2)\hconv\bl \tEs_i{}^\max \M(\mu_2,s_i\mu_3)\br,$$
from which we obtain
$$\M(\mu_1,\mu_3)\simeq \M(\mu_1,\mu_2)\hconv  \M(\mu_2,\mu_3).$$

\medskip\noi
 Lemma \ref{lem: tLa} implies that
$\tLa\bl\M(\mu_1,\mu_2),\M(\mu_2,\mu_3)\br=0$. Hence we have
$$\La\bl\M(\mu_1,\mu_2),\M(\mu_2,\mu_3)\br
=-\bl\wt(\M(\mu_1,\mu_2),\wt(\M(\mu_2,\mu_3))\br.$$
\end{proof}

\begin{proposition}\label{prop:dMM}
Let $i \in I$ and $x,y,z \in W$.
\begin{enumerate}
\item[{\rm (i)}] If $\ell(xy)=\ell(x)+\ell(y)$, $zs_i>z$,
 $y>ys_i$, 
and $x \ge z$, then we have
$$ \de(\M(xy\varpi_i,zs_i\varpi_i),\M(x\varpi_i,z\varpi_i)) \le 1.$$
\item[{\rm (ii)}] If $\ell(zy)=\ell(z)+\ell(y)$, $xs_i>z$, $xs_i\ge zy$ and $x \ge z$, then we have
$$ \de(\M(xs_i\varpi_i,zy\varpi_i),\M(x\varpi_i,z\varpi_i)) \le 1.$$
\end{enumerate}
\end{proposition}

\begin{proof}
In the course of proof, we omit $\oi_{\vpi,\vpi}^{-1}$ for the sake of simplicity.
Set $y' \seteq ys_i < y$.

\medskip

Let us show {\rm (i)}. By Proposition \ref{prop: g T-system D}, we have
\begin{equation} \label{eq: ac le 1}
\begin{aligned}
\Delta(y\vpi,s_i\vpi)\Delta(\vpi,\vpi) & =q^{-1}G^\up\left( (u_{y\vpi} \otimes u_{\vpi}) \otimes (u_{s_i\vpi} \otimes u_{\vpi})^\ri \right) \\
& \hspace{10ex} +G^\up(\oi_\lambda^{-1}(\te_i^*b) \otimes u_\lambda^\ri),
\end{aligned}
\end{equation}
where $\lambda=\vpi+s_i\vpi$ and $b=\oi_\vpi(u_{y\vpi}) \in B(\infty)$.
Let $S_z^*$ be the operator on $\Ag$ given by
the multiplication of $e_{j_1}^{(a_1)}\cdots e_{j_t}^{(a_t)}$ from the right,
where $z=s_{j_t}\cdots s_{j_1}$ is a reduced expression of $z$ and
$a_k=\lan h_{j_k},s_{j_{k-1}}\cdots s_{j_1}\la\ran$.
Then applying $S_z^*$  to \eqref{eq: ac le 1}, we obtain
\begin{align*}
\Delta(y\vpi,zs_i\vpi)\Delta(\vpi,z\vpi) &  =q^{-1}G^\up\left( (u_{y\vpi} \otimes u_{\vpi} ) \otimes (u_{zs_i\vpi} \otimes u_{z\vpi})^\ri \right) \\
& \hspace{10ex} +G^\up(\oi_\lambda^{-1}(\te_i^*b) \otimes u_{z\lambda}^\ri).
\end{align*}

Recall that $\mu\in\Po$ is called {\em $x$-dominant}
if $c_k\ge0$.
Here $x=s_{i_r} \cdots s_{i_1}$ is a reduced expression of $x$ and
 $c_k\seteq \lan h_{i_k}, s_{i_{k-1}} \cdots s_{i_1} \mu \ra$ ($1\le k\le r$).
Recall that an element $v\in\Ag$ with $\wtl(v)=\mu$ is called
{\em $x$-highest}
if $\mu$ is $x$-dominant and
$$\text{$f_{i_k}^{1+c_k}f_{i_{k-1}}^{(c_{k-1})}\cdots f_{i_1}^{(c_{1})}v=0$ for any $k$
($1\le k\le r$).}$$
If $v$ is $x$-highest, then
$v$ is a linear combination of
$x$-highest $G^\up(b)$'s.
Moreover, $S_{x,\mu}G^\up(b)\seteq f^{(c_r)}_{i_r}\cdots f^{(c_1)}_{i_1} G^\up(b)$ is either a member of the upper global basis or zero.
Since $\Delta(y\vpi,zs_i\vpi)\Delta(\vpi,z\vpi)$ is $x$-highest
of weight $\mu\seteq\vpi_i+\vpi_i$,
we obtain
\begin{align*}
\Delta(xy\vpi,zs_i\vpi)\Delta(x\vpi,z\vpi) & =q^{-1}G^\up\left( (u_{xy\vpi} \otimes u_{x\vpi} ) \otimes (u_{zs_i\vpi} \otimes u_{z\vpi})^\ri \right) \\
& \hspace{10ex} +S_{x,\mu}G^\up(\oi_\lambda^{-1}(\te_i^*b) \otimes u_{z\lambda}^\ri).
\end{align*}

Applying $p_\n$, we obtain
\begin{align*}
q^c \D(xy\vpi,zs_i\vpi)\D(x\vpi,z\vpi) & =q^{-1}p_\n G^\up\left( (u_{xy\vpi} \otimes u_{x\vpi} ) \otimes (u_{zs_i\vpi} \otimes u_{z\vpi})^\ri \right) \\
& \hspace{10ex} +p_\n S_{x,\mu}G^\up(\oi_\lambda^{-1}(\te_i^*b) \otimes u_{z\lambda}^\ri)
\end{align*}
for some integer $c$. Hence we obtain {\rm (i)} by
Lemma~\ref{lem:short}.

\medskip\noi
{\rm (ii)} is proved similarly. By applying $\varphi^*$  to \eqref{eq: ac le 1}, we obtain
\begin{align*}
q^{(s_i\vpi_i,\vpi_i)-(y\vpi_i,\vpi_i)}\Delta(s_i\vpi,y\vpi)\Delta(\vpi,\vpi) & =q^{-1}G^\up\left( (u_{s_i\vpi} \otimes u_{\vpi}) \otimes (u_{y\vpi} \otimes u_{\vpi})^\ri \right) \\
& \hspace{10ex} +G^\up(u_\lambda  \otimes (\oi_\lambda^{-1}\te_i^*b)^\ri).
\end{align*}
 Here we used Proposition~\ref{prop:phimul}.
Then the similar arguments as above show {\rm (ii)}.
\end{proof}

\begin{proposition}\label{prop:deMM}
Let $x \in W$ such that $xs_i > x$ and $x \vpi \ne \vpi$. Then we have
$$\de(\M(xs_i\vpi,x\vpi),\M(x\vpi,\vpi)) =1.$$
\end{proposition}

\begin{proof}
 By Proposition~\ref{prop:dMM} (ii),
 we have $\de(\M(xs_i\vpi,x\vpi),\M(x\vpi,\vpi)) \le1$.
Assuming $\de(\M(xs_i\vpi,x\vpi),\M(x\vpi,\vpi))=0$, let us derive a contradiction.

By Theorem~\ref{thm: canonical surjection} and the assumption, we have
$$\M(xs_i\vpi,x\vpi) \conv \M(x\vpi,\vpi) \simeq \M(xs_i\vpi,\vpi).$$
Hence we have
$$ \ve_j^*(\M(xs_i\vpi,\vpi))= \ve_j^*(\M(xs_i\vpi,x\vpi))+\ve_j^*(\M(x\vpi,\vpi))$$
for any $j \in I$. Since $xs_i\vpi \preceq x\vpi \preceq s_i\vpi$, Proposition \ref{prop: weyl left right} implies that
$$ \ve_j^*(\M(xs_i\vpi,\vpi))=\ve_j^*(\M(x\vpi,\vpi))=\lan h_j,\vpi \ra.$$
It implies that
$$\ve_j^*(\M(xs_i\vpi,x\vpi))=0 \quad \text{ for any } j \in I.$$
It is a contradiction  since $\wt\bl\M(xs_i\vpi,x\vpi)\br=xs_i\vpi-x\vpi$ does not vanish.
\end{proof}

\section{Quantum cluster algebras
and monoidal categorifications}

\subsection{Quantum cluster algebras} In this subsection, we briefly recall
the definition of a quantum cluster algebra following \cite{BZ05}, \cite[\S 8]{GLS11} and \cite[\S 4.1]{KKKO15}.

\medskip

Let us fix a finite index set $\K$
with a decomposition
$\K=\Kex\sqcup\Kfr$
into the set $\Kex$ of {\em exchangeable indices} and
the set $\K_\fr$ of {\em frozen indices}.
Let $L=(\lambda_{i,j})_{i,j \in\K}$ be a skew-symmetric  integer-valued $\K \times \K$-matrix.
We denote by $\mathscr{P}(L)$ the $\Z[q^{\pm 1/2}]$-algebra generated by $\{X_i\}_{i \in\K}$ subject to the following defining relations:
\begin{align}\label{eq: L-commuting} X_iX_j=q^{\lambda_{i,j}}X_jX_i \quad (i,j \in\K). \end{align}

For $\mathbf{a}\seteq(a_i)_{i\in\K}\in \Z_{\ge0}^\K$, we define the element $X^\mathbf{a}$ of $\mathscr{P}(L)$ as
$$
X^{\bf a}\seteq q^{1/2 \sum_{i > j} a_ia_j\lambda_{i,j}} \overset{\Lto}{\prod}_{i\in\K} X_i^{a_i}
$$
where $<$ is a total order on $J$ and
$\overset{\Lto}{\prod}_{i\in\K} X_i^{a_i}=X_{i_1}^{a_1}\cdots X_{a_r}^{i_r}$
with $\K=\{i_1,\ldots,i_r\}$ such that
$i_1 < i_2 < \cdots < i_r$.
Note that $X^{\bf a}$ does not depend on the choice of a total order on $J$.
We have
\begin{align}
X^{\bf a}X^{\bf b} =  q^{1/2 \sum_{i,j \in\K}a_ib_j\lambda_{i,j}}  X^{\bf a+b} \quad \text{ for all ${\bf a}$, ${\bf b} \in \Z_{\ge0}^\K$.}
\end{align}
It is well-known that $\{ X^{\bf a} \ | \ {\bf a} \in \Z_{\ge 0}^\K \}$ forms a $\Z[q^{\pm 1/2}]$-basis of $\mathscr{P}(L)$.

For a $\Z[q^{\pm 1/2}]$-algebra $A$, we say that a family of elements $\{x_i\}_{i \in\K}$ of $A$ is {\em $L$-commuting} if it satisfies the relation
\eqref{eq: L-commuting}, i.e., $x_ix_j=q^{\lambda_{i,j}}x_jx_i$.

\medskip
Let $\wB=(b_{i,j})$ be an integer-valued $\K \times \K_\ex$ matrix whose {\em principal part} $B\seteq (b_{i,j})_{i,j \in\K_\ex}$ is skew-symmetric.
To the matrix $\wB$ we can associate the quiver $Q_{\wB}$ without loops,  $2$-cycles  and arrows between frozen vertices
such that its vertices are labeled by $J$ and
the arrows are given by
\begin{equation} \label{eq: bij}
 b_{i,j} =  (\text{the number of arrow from $i$ to $j$}) - (\text{the number of arrow from $j$ to $i$}).
 \end{equation}
Here we extend the $\K\times\Kex$-matrix $\wB$ to
the skew-symmetric $\K\times\K$-matrix $\wB'=(b_{i,j})_{i,j\in\K}$
by setting $b_{i,j}=0$ for $i,j\in\Kfr$.

Conversely, whenever we have a quiver with vertices labeled by $J$ and without loops, $2$-cycles and arrows between frozen vertices, we can
associate a $\K \times \K_\ex$-matrix $\wB$ by \eqref{eq: bij}.

We say that the pair of matrices $(L,\wB)$ is {\em compatible} if there exists a positive integer $d$ such that
\begin{align}\label{eq: integer d}
  \sum_{k\in\K} \lambda_{i,k}b_{k,j} = \delta_{i,j}d, \quad ( i \in\K, \  j \in\K_\ex).
  \end{align}

For a $\Z[q^{\pm 1/2}]$-algebra $A$ and a compatible pair $(L,\wB)$, we say that the datum $\Seed=(\{x_i\}_{i \in\K},L,\wB)$ is
a {\em quantum seed} if $\{x_i\}_{i \in\K}$ is an algebraically independent $L$-commuting family of elements in $A$.

The set $\{ x_i\}_{i \in\K}$ is called a {\em cluster} of $\Seed$ and its elements are called {\em cluster variables}. In particular,
the elements  in  $\{x_i\}_{i \in\K_\fr}$ are called {\em frozen variables}. The elements in $\{ x^{\mathbf{a}} \ | \ \mathbf{a} \in \Z_{\ge 0}^r \}$
are called {\em cluster monomials}.

See \cite{KKKO15} for the definition of
\begin{align}\label{eq: mutation}
\mu_k(\Seed)\seteq \big(\{\mu_k(x)_i\}_{i\in\K},\mu_k(L),\mu_k(\widetilde B)\big),
\end{align}
the {\em mutation of $\Seed$ in direction $k\in \Kex$}.

\begin{definition}
Let $\Seed=(\{x_i\}_{i\in\K},L, \widetilde B)$ be a quantum seed in a $\Z[q^{\pm1/2}]$-algebra $A$ which is contained a skew field $K$.
The {\em quantum cluster algebra $\mathscr A_{q^{1/2}}(\mathscr S)$} associated with the quantum seed $\mathscr S$ is
the $\Z[q^{\pm 1/2}]$-subalgebra of the skew field $K$ generated by all the quantum cluster variables
in the quantum seeds obtained from $\Seed$ by any sequence of {\em mutations}.
 \end{definition}

We call $\Seed$ the {\em initial quantum seed}
of the quantum cluster algebra $\mathscr A_{q^{1/2}}(\Seed)$.

\subsection{Monoidal categorifications of quantum cluster algebras}
In this subsection, we shall review  monoidal categorifications of
cluster algebras. In this paper, we shall restrict ourselves to the case when
monoidal categories are full subcategories of $R\gmod$.
Here $R\gmod$ is the category of finite-dimensional graded modules
over a symmetric KLR algebra $R$ with the base field $\cor$ of
characteristic $0$.
We refer to \cite{KKKO14} in a more general setting.

Let $\shc$ be a full subcategory of $R\gmod$ stable by taking
convolution products, subquotients, extensions and grading shifts.
Hence $\shc$ is a $\ko$-linear abelian monoidal category with the decomposition
$$\shc=\soplus_{\beta\in\rtl^-}\shc_\beta\quad
\text{with $\shc_\beta\seteq\shc\cap R(-\beta)\gmod$.}$$

\begin{definition} \label{def:quantum monoidal seed}
We call a quadruple $\Seed = (\{M_i\}_{i \in\K}, L,\wB, D)$
a \emph{quantum monoidal seed  in $\shc$}
 if it satisfies the following conditions:
\begin{enumerate}
\item[{\rm (i)}] $\wB = (b_{i,j})$ is an integer-valued $\K \times \K_\ex$-matrix whose principal part is skew-symmetric.
\item[{\rm (ii)}] $L=(\lambda_{i,j})$ is an integer-valued skew-symmetric $\K \times \K$-matrix.
\item[{\rm (iii)}] $D=\{d_i\}_{i \in\K}$ is a family of elements in $\rl$.
\item[{\rm (iv)}] $\{M_i\}_{i \in\K}$ is a family of
simple objects such that $M_i \in \shc_{d_i}$ for $i \in\K$.
\item[{\rm (v)}] $M_i \conv M_j \simeq q^{\lambda_{i,j}} M_j \conv M_i$ for all $i,j \in\K$.
\item[{\rm (vi)}] $M_{i_1} \conv \cdots \conv M_{i_t}$ is simple for every sequence $(i_1,\ldots,i_t)$ in $J$.
\item[{\rm (vii)}] The pair  $(L,\wB)$  is  compatible in the sense of \eqref{eq: integer d} with $d=2$.
\item[{\rm (viii)}] $\lambda_{i,j} - (d_i,d_j) \in 2\Z$ for all $i,j \in\K$.
\item[{\rm (ix)}] $\displaystyle\sum_{i \in\K} b_{i,k}d_i =0$ for all $k \in\K_\ex$.
\end{enumerate}
\end{definition}
By (vi), every $M_i$ is real simple.
The integer $\la_{i,j}$ in (ii) is given by $-\Lambda(M_i,M_j)$.
Note that (viii) is redundant since it follows from $\tLa(M_i,M_j)\in\Z$.
For a $\K \times \K_\ex$-matrix $\wB$ with skew-symmetric principal part and  $D=\{d_i\}_{i \in\K}$,
we define the \emph{mutation $\mu_k(D) \in \rl^\K$ of $D$ in direction $k$ with respect to $\wB$} by
\begin{align*}
\mu_k(D)_i =d_i \ (i \neq k), \quad \mu_k(D)_k=-d_k+\sum_{b_{i,k} >0}   b_{i,k} d_i.
\end{align*}

Note that
\begin{itemize}
\item $\mu_k(\mu_k(D))=D \quad (\text{for } k \in\K_\ex)$,
\item $\big(\mu_k(L),\mu_k(\wB),\mu_k(D)\big)$ satisfies the conditions {\rm (viii)} and {\rm (ix)} for any $k \in\K_\ex$.
\end{itemize}

\begin{definition} [{\cite[Definition 5.6]{KKKO15}}] \label{def:monoidal mutation}
For $k\in\K_\ex$,
we say that a quantum monoidal seed
$\mathscr S =(\{M_i\}_{i\in\K}, L,\widetilde B, D)$
 \emph{admits a mutation in direction $k$} if
there exists a simple object  $M_k' \in \shc_{\mu_k(D)_k}$
such that
\bnum
  \item
there exist exact sequences in $\shc$
\begin{align}
&0 \to q \sodot_{b_{i,k} >0} M_i^{\snconv b_{i,k}} \to q^{m_k} M_k \conv M_k' \to
 \sodot_{b_{i,k} <0} M_i^{\snconv (-b_{i,k})} \to 0, \label{eq:ses graded mutation1} \\
&0 \to q \sodot_{b_{i,k} <0} M_i^{\snconv(-b_{i,k})} \to q^{m'_k} M_k' \conv M_k \to
  \sodot_{b_{i,k} >0} M_i^{\snconv b_{i,k}} \to 0, \label{eq:ses graded mutation2}
\end{align}
where
\begin{equation} \label{eq: mm}
m_k=\dfrac{1}{2}(d_k,\xi) +\dfrac{1}{2} \displaystyle \sum_{b_{i,k} < 0} \lambda_{k,i}b_{i,k}
\quad \text{and} \quad
m'_k=\dfrac{1}{2}(d_k,\xi) +\dfrac{1}{2} \displaystyle \sum_{b_{i,k} > 0} \lambda_{k,i}b_{i,k}
\end{equation}
with $\xi=-d_k+\sum_{b_{i,k}>0}b_{i,k}d_i$.
\item the quadruple $\mu_k(\Seed)\seteq(\{M_i\}_{i\neq k}\cup\{M_k'\},\mu_k(L),
\mu_k(\widetilde B), \mu_k(D))$ is
a quantum monoidal seed in $\shc$.
\end{enumerate}
\end{definition}
We call $\mu_k(\Seed)$ the {\em mutation} of $\Seed$ in direction $k$.

\begin{definition}[{\cite[Definition 5.8]{KKKO15}}]
The category $\shc$ is called a \em{monoidal categorification of a quantum cluster algebra $A$ over $\Z[q^{\pm1/2}]$}
if
\begin{enumerate}
\item[{\rm (i)}] the Grothendieck ring $\Z[q^{\pm1/2}]\tens_{\Z[q^{\pm1}]} K(\shc)$ is isomorphic to $A$,
\item[{\rm (ii)}] there exists a quantum monoidal seed
$\mathscr S =(\{M_i\}_{i\in\K}, L,\widetilde B, D)$ in $\shc$ such that
$[\mathscr S]\seteq(\{q^{-(d_i,d_i)/4}[M_i]\}_{i\in\K}, L, \widetilde B)$
 is the image of a quantum seed of $A$ by the isomorphism in {\rm(i)}.
\item[{\rm (iii)}] $\mathscr S$ admits successive mutations in all directions,
\end{enumerate}
\end{definition}
Note that if $\shc$ is a monoidal categorification of $A$, then any quantum seed in $A$ obtained by mutations from the initial quantum seed is
of the form $(\{q^{-(d_i,d_i)/4}[M_i]\}_{i\in\K}, L, \widetilde B)$  for some monoidal seed
$(\{M_i\}_{i\in\K}, L,\widetilde B, D)$.
Thus any quantum cluster monomial in $A$ coincides with
$q^{-(\wt(S),\wt(S))/4}[S]$ for some real simple module
$S$ in $\shc$.

\subsection{Main result of \cite{KKKO15}} In this subsection, we briefly recall the main result of \cite{KKKO15}.

\begin{definition}[{\cite[Definition 6.1]{KKKO15}}] \label{def:admissible}
A pair $(\{M_i\}_{i \in\K}, \widetilde B)$ is called \emph{admissible} if
\be
\item  $\{M_i\}_{i \in\K}$ is a commuting family of real simple  self-dual  objects of $\shc$,
\item $\widetilde B$ is an integer-valued $\K \times \K_\ex$-matrix with skew-symmetric principal part, and
\item
 for each $k \in\K$, there exists a  self-dual  simple object $M'_k$ in $\shc$
 such that there is an exact sequence in $\shc$
\begin{align}  \label{eq:eq:ses graded mutation KLR}
0 \to q \sodot_{b_{i,k} >0} M_i^{\snconv b_{i,k}} \to q^{\tLa(M_k,M_k')} M_k \conv M_k' \to
 \sodot_{b_{i,k} <0} M_i^{\snconv (-b_{i,k})} \to 0,
\end{align}
and
$M_k'$ commutes with $M_i$  for all  $i \neq k \in\K$.
  \end{enumerate}
\end{definition}

For an admissible pair  $(\{M_i\}_{i \in\K}, \widetilde B)$, let
$\La=(\La_{i,j})_{i,j \in\K}$
be the skew-symmetric matrix
given by $\La_{i,j}=\Lambda(M_i,M_j)$.
and let $D=\{d_i\}_{ i \in\K}$ be the family of elements of $\rl^-$ given by $d_i=\wt(M_i)$.

\begin{theorem}[{\cite[Theorem 6.3, Corollary 6.4]{KKKO15}}] Let $\mathscr S=(\{M_i\}_{i \in\K},-\La,\wB,D)$ be a quantum monoidal seed.
 We assume the following:
\begin{eqnarray*}&&
\parbox{80ex}{The $\Q(q^{1/2})$-algebra $\Q(q^{1/2})\tens\limits_{\Z[q^{\pm1}]}K(\shc)$
is isomorphic to $\Q(q^{1/2})\tens\limits_{\Z[q^{\pm1/2}]}\mathscr A_{q^{1/2}}([\Seed])$.
}
\end{eqnarray*}
If the pair $(\{M_i\}_{i \in\K},\widetilde B)$ is admissible, then
$\mathscr S$ admits successive mutations in all directions
and the category $\shc$ is
a monoidal categorification of $\mathscr{A}_{q^{1/2}}([\Seed])$.
\end{theorem}

\section{Monoidal categorification of $A_q(\n(w))$}

\subsection{Quantum cluster algebra structure on $A_q(\n(w))$} \label{subsec: Q cluster Anw}
In this subsection, we shall consider the Kac-Moody algebra
$\g$ associated with a symmetric Cartan matrix $\cmA=(a_{i,j})_{i,j\in I}$.
We shall recall briefly the definition of the subalgebra $A_q(\n(w))$ of $A_q(\g)$ and its quantum cluster algebra structure by using
the results of \cite{GLS11} and \cite{Kimu12}.
Remark that we bring the results in \cite{GLS11}
through the isomorphism in \eqref{eq:isoAqn}.
\medskip

For a given $w \in W$, fix a reduced expression $\widetilde{w}=s_{i_r}\cdots s_{i_1}$.
For $s \in \{ 1,\ldots,r \}$ and $j \in I$, we set
\begin{equation} \label{eq: s+ s-}
\begin{aligned}
s_+&\seteq\min( \{k \ | \ s<k\le r,\; i_k=i_s \}\cup\{r+1\}),\\
s_-&\seteq\max(\{k \ | \ 1\le s<k,\; i_k=i_s \}\cup\{0\}),\\
s^-(j)&\seteq\max(\{k \ | \ 1\le k<s,\; i_k=j \}\cup\{0\}).
\end{aligned}
\end{equation}
For $1\le k\le r$, set
\eq
\text{$\la_k\seteq u_k\varpi_{i_k}$ where $u_k\seteq s_{i_1}\cdots s_{i_k}$.}\label{eq:lau}
\eneq
Note that $\la_{k_- }=u_{k-1}\varpi_{i_k}$.
For $0\le t\le s\le r$, we set
\begin{align} \label{eq: D(s,t)}
\D(s,t)&=
\bc
\D(\lambda_s,\lambda_t) & \text{if $0<t$,}\\
\D(\lambda_s,\varpi_{i_s})& \text{if $0=t<s\le r$,}\\
\one&\text{if $t=s=0$.}
\ec
\end{align}

 The $\Q(q)$-subalgebra of  $A_q(\n)$ generated by
$\D(i,i_-)$ ($1\le i\le r$) is independent of the choice of a reduced expression of $w$. We denote it by $A_q(\n(w))$.
Then every $\D(s,t)$ ($0\le t\le s\le r$) is contained in $A_q(\n(w))$
(\cite[Corollary 12.4]{GLS11}).
The set $\B^\up\bl A_q(\n(w))\br\seteq\B^\up\bl A_q(\g)\br\cap A_q(\n(w))$
is a basis of $A_q(\n(w))$ as a $\Q(q)$-vector space
(\cite[Theorem 4.2.5]{Kimu12}).
We call it the upper global basis of $A_q(\n(w))$.
 We denote by $A_q(\n)_\A$ the $\A$-module generated by $\B^\up\bl A_q(\n(w))$.
Then it is a $\A$-subalgebra of $A_q(\n(w))$ (\cite[\S\,4.7.2]{Kimu12}).
We set $A_{q^{1/2}}(\n(w)) \seteq \Q(q^{1/2})  \otimes_{\Q(q)}  A_{q}(\n(w))$.

\medskip
We set $\K=\{1,\ldots,r\}$,
$\Kfr\seteq\set{k\in\K}{k_+=r+1}$ and $\Kfr\seteq\K\setminus\Kfr$.
\begin{definition} We define the quiver $Q$
with the set of vertices $Q_0$ and the set of
arrows $Q_1$ as follows:
\begin{enumerate}
\item[$(Q_0)$] 
$Q_0=\K=\{1,\ldots,r\}$,
\item[$(Q_1)$] There are two types of arrows:
\eqn
&&\ba{lccl}\text{{\em ordinary arrows}}&:&\text{$s\To[\ |a_{i_s,i_t}|\ ] t$}
&\text{if $1\le s<t<s_+<t_+\le r+1$,}\\[1ex]
\text{{\em horizontal arrows}}&:&
\text{$s\longrightarrow s_-$}&
\text{if $1\le s_-<s\le r$.} 
\ea
\eneqn
\end{enumerate}
Let $\widetilde B = (b_{i,j})$ be the integer-valued $\K\times\Kex$-matrix associated to the quiver $Q$ by \eqref{eq: bij}.
\end{definition}

\begin{lemma}
Assume that $0 \le d\le b\le a\le c \le r$ and
\begin{itemize}
\item $i_b=i_a$ if $b \ne 0$,
\item $i_d=i_c$ if $d \ne 0$.
\end{itemize}
Then $\D(a,b)$ and $\D(c,d)$ $q$-commute; that is, there exists $\lambda \in \Z$ such that
$$ \D(a,b)\D(c,d)= q^{\lambda}\D(c,d)\D(a,b).$$
\end{lemma}

\begin{proof}
We may assume $a>0$.
Let $u_k$ be as in \eqref{eq:lau}.
Take $s'=u_a$, $s=u_a^{-1}u_c$, $t'=u_d$ and $t=u_d^{-1}u_b$. Then we have
$$ \D(s'\varpi_{i_a},t't\varpi_{i_a})=\D(a,b)\quad \text{and}\quad
  \D(s's\varpi_{i_c},t'\varpi_{i_c})=\D(c,d).$$
From Proposition \ref{prop: BZ form}, our assertion follows.
\end{proof}

Hence we have an integer-valued skew-symmetric matrix $L=(\lambda_{i,j})_{1 \le i,j \le r}$ such that
$$ \D(i,0)\D(j,0)=q^{\lambda_{i,j}}\D(j,0)\D(i,0).$$

\begin{proposition} [{\cite[Proposition 10.1]{GLS11}}] The pair $(L,\wB)$ is compatible with $d=2$ in \eqref{eq: integer d}.
\end{proposition}

\begin{theorem}[{\cite[Theorem 12.3]{GLS11}}] Let
$\mathscr{A}_{q^{1/2}}([\mathscr{S}])$ be the quantum cluster algebra associated to the initial quantum seed
$[\mathscr{S}]\seteq( \{ q^{-(d_s,d_s)/4}\D(s,0) \}_{1 \le s \le r}, L, \wB )$.
Then we have an isomorphism
$$ \Q(q^{1/2})  \otimes_{\Z[q^{\pm 1/2}]} \mathscr{A}_{q^{1/2}}([\mathscr{S}]) \simeq   A_{q^{1/2}}(\n(w)),$$
where $d_s \seteq \lambda_s-\varpi_{i_s}=\wt(D(s,0))$ and
$A_{q^{1/2}}(\n(w)) \seteq \Q(q^{1/2})  \otimes_{\Q(q)}  A_{q}(\n(w))$.
\end{theorem}

\subsection{Admissible seeds in the monoidal category $\shc_w$}

 For $0\le t\le s\le r$, we set
$\M(s,t)=\M(\la_s,\la_t)$.
It is a real simple module with $\ch(\M(s,t))=D(s,t)$.

\begin{definition} For $w \in W$, let $\shc_w$ be the smallest monoidal abelian full subcategory of $R \gmod$ satisfying the following properties:
\begin{enumerate}
\item[{\rm (a)}] It is stable under the subquotients, extensions and grading shifts.
\item[{\rm (b)}] It contains $\M(s,s_-)$ for all $1 \le s \le \ell(w)$.
\end{enumerate}
\end{definition}

Then by \cite{GLS11}, $M\in R\gmod$ belongs to $\shc_w$ if and only if $\ch(M)$
belongs to $A_q(\n(w))$. Hence we have an $\A$-algebra isomorphism
$$K(\shc_w)\simeq A_q(\n(w))_\A.$$

We set
$$\Lambda \seteq (\La(\M(i,0),\M(j,0)))_{1 \le i,j\le r} \quad \text{ and }
\quad D= (d_i)_{1 \le i \le r} \seteq (\wt(\M(i,0)))_{1 \le i \le r}.$$

Then, by Proposition~\ref{prop: commute 1},
$\Seed \seteq ( \{ \M(k,0) \}_{1 \le k \le r}, -\Lambda, \wB, D )$
is a quantum monoidal seed in $\shc_w$.
We are now ready to state the main theorem in this subsection:

\begin{theorem} \label{thm: main}
The pair
$\bl \{ \M(k,0) \}_{1 \le k \le r}, \wB\br$ is admissible.
\end{theorem}
As we already explained, this theorem implies
Theorem~\ref{th:main} in Introduction.
\begin{theorem}
The category $\shc_w$ is a monoidal categorification of the quantum cluster algebra $A_{q^{1/2}}(\n(w))$.
\end{theorem}

In the course of the proof of Theorem~\ref{thm: main},
{\em we omit grading shifts} if there is no afraid of confusion.

We shall start the proof of Theorem~\ref{thm: main}
by proving that, for each $s\in\Kex$,  there exists a simple module $X$ such that
\begin{eqnarray}&&\left\{
\parbox{70ex}{ \bnam
\item there exists a surjective homomorphism (up to a grading shift)
$$  X \conv \M(s,0) \twoheadrightarrow \conv_{t;\; b_{t,s}> 0 } \M(t,0)^{\circ b_{t,s} },$$
\item there exists a surjective homomorphism (up to a grading shift)
$$ \M(s,0) \conv X \twoheadrightarrow \conv_{t;\;\ b_{t,s}<0} \M(t,0)^{\circ -b_{t,s}},$$
\item $\de(X,\M(s,0))=1$.
\ee }\right. \label{eq: conditions}
\end{eqnarray}

We set
\eqn
&&x\seteq i_s\in I,\\
&&I_s\seteq\set{i_k}{s<k<s_+}\subset I\setminus\{x\},\\
&&A\seteq
  \displaystyle\conv\limits_{ \substack{t<s < t_+ < s_+ }  }
\M(t,0)^{\circ |a_{i_s,i_t}| }
= \displaystyle\conv\limits_{ y\in I_s}
\M(s^-(y),0)^{\circ |a_{x,y}| }.
\eneqn
Then $A$ is a real simple module.

Now we claim that the following simple module $X$ satisfies
the conditions in \eqref{eq: conditions}:
$$X \seteq \M(s_+,s) \hconv  A.$$

Let us show \eqref{eq: conditions} (a).
The incoming arrows to $s$ are
\begin{itemize}
\item $t \To[{|a_{x,i_t}|}]s $ for $1\le t<s<t_+<s_+$,
\item $s_+ \To s$.
\end{itemize}
Hence we have
$$\conv_{t;\;  b_{t,s} > 0 } \M(t,0)^{\circ b_{t,s}}
\simeq A\conv M(s_+,0).$$

\medskip

Then the morphism in (a)
is obtained  as the composition:
\begin{align} \label{eq: surjection 1}
X \conv \M(s,0)  \rightarrowtail A \conv \M(s_+,s) \conv \M(s,0) \twoheadrightarrow A \conv \M(s_+,0).
\end{align}
Here the second epimorphism
is given in Theorem \ref{thm: canonical surjection},  and
\cite[Corollary 3.11]{KKKO14} asserts that the composition
\eqref{eq: surjection 1} is non-zero
and hence an epimorphism.

\medskip

Let us show \eqref{eq: conditions} (b).
The outgoing arrows from $s$ are
\begin{itemize}
\item $s \To[{\;|a_{x,i_t}|\;}]t$\quad for $s<t<s_+<t_+\le r+1$.
\item $s \longrightarrow s_-$\quad if $s_- > 0$.
\end{itemize}
Hence we have
\eq
&&\conv_{t; b_{t,s}<0 } \M(t,0)^{\circ -b_{t,s}}
\simeq\M(s_-,0)\conv \left(\conv_{y\in I_s} \M((s_+)^-(y),0)^{\circ -a_{x,y}}\right).
\eneq

\begin{lemma} \label{lem: step1}
There exists an epimorphism  \ro up to a grading\rf
\begin{align} \label{eq: Omega}
\Omega: \M(s,0) \conv \M(s_+,s) \conv A \twoheadrightarrow \conv_{t;b_{t,s}<0} \M(t,0)^{\circ -b_{t,s}}.
\end{align}
\end{lemma}

\begin{proof}
By the dual of Theorem \ref{thm: canonical surjection} and the $T$-system
\eqref{eq: deteminatial seq2}
with $i=i_s$, $u=u_{s_+-1}$ and $v=u_{s-1}$, we have morphisms
\eq
&& \M(s,0) \rightarrowtail \M(s_-,0) \conv \M(s,s_-), \label{eq: 1} \\
&&\M(s,s_-)\conv \M(s_+,s) \twoheadrightarrow
\displaystyle \conv_{y \in I\setminus\{x\}} \M((s_+)^-(y),s^-(y))^{\circ -a_{x,y}}\\\label{eq: 2}
&&\hs{35ex}\simeq \displaystyle \conv_{y \in I_s}
\M((s_+)^-(y),s^-(y))^{\circ -a_{x,y }}.  \nn
\eneq
Here the last isomorphism follows from the fact that
$(s_+)^-(y)=s^-(y)$ for any $y\not\in \{x\}\cup I_s=\set{i_k}{s\le k<s_+}$.

Thus we have a sequence of morphisms
\eqn
&&\M(s,0) \conv \M(s_+,s) \conv A\; \xymatrix{\ar@{>->}[r]^-{\varphi_1}&}
 \M(s_-,0) \conv \M(s,s_-) \conv \M(s_+,s) \conv A  \\
&&\hs{25ex}\xymatrix@C=6ex{\ar@{->>}[r]^-{\varphi_2} &}\M(s_-,0) \conv
\left(\conv_{y \in I_s} \M((s_+)^-(y),s^-(y))^{\circ -a_{x,y}}\right) \conv A.
\eneqn
By \cite[Corollary 3.11]{KKKO14}, the composition
$\varphi \seteq \varphi_2 \circ \varphi_1$ is non-zero.

Since
$A=  \displaystyle\conv_{ \substack{y\in I_s} }  \M(s^-(y),0)^{\circ -a_{x,y} }$,
Theorem \ref{thm: canonical surjection} gives the  morphisms
\begin{align*}
 \M(s,0) \conv \M(s_+,s) \conv A   & \xymatrix@C=4ex{\ar[r]^-{\varphi}&}
   \M(s_-,0) \conv
\left(\conv_{y\in I_s} \M((s_+)^-(y),s^-(y))^{\circ -a_{x,y}}\right) \conv A  \\
&\xymatrix@C=4ex{\ar@{->>}[r]^-{\phi} &}\M(s_-,0)\conv \left(\conv_{y\in I_s} \M((s_+)^-(y),0)^{\circ -a_{x,y}}\right)\simeq \conv_{t; b_{t,s}<0} \M(t,0)^{\circ -b_{t,s}}.
\end{align*}
Here we have used Lemma~\ref{lem:MN} to obtain the morphism $\phi$.
Note that  the module $\conv_{y \in I_s} \M((s_+)^-(y),s^-(y))^{\circ -a_{x,y}}$
is  simple.
By applying \cite[Corollary 3.11]{KKKO14} once again, $\phi \circ \varphi$ is non-zero, and hence it is an epimorphism.
\end{proof}

\Lemma\label{lem: simply linked}
We have $\de(X,\M(s,0))=1$.
\enlemma

\begin{proof}
Since $A$ and $\M(s,0)$ commute and
$\de\big(\M(s_+,s),\M(s,0)\big)=1$ by Proposition~\ref{prop:deMM}, we have
$$\de\big( X,\M(s,0) \big) \le\de\big(\M(s_+,s),\M(s,0)\big)+
\de\big(A,\M(s,0)\big)\le1,$$
by \cite[Corollary 2.17]{KKKO15} and Lemma \ref{lem: d=0}.
If $X$ and $\M(s,0)$ commute, then \eqref{eq: conditions} (a)
 would imply
that ${\rm ch}\left(\conv_{t;b_{st}<0} \M(t,0)^{\circ -b_{st}}\right)$
belongs to $K(R \gmod)\,{\rm ch}(\M(s,0))$.
It contradicts the result in \cite{GLS13} that all the ${\rm ch}(\M(k,0))$'s are
prime at $q=1$.
\end{proof}

\begin{proposition} \label{prop: step2}
The map $\Phi$ factors through $\M(s,0) \conv X$; that is,
\begin{align} \label{eq: surjection 2}
\xymatrix{
 \M(s,0) \conv \M(s_+,s) \conv A  \ar@{->>}[drr]^{\tau}\ar@{->>}[rrrr]^{\Omega} &&&&  \displaystyle \conv_{t; b_{t,s}<0} \M(t,0)^{\circ  -b_{t,s}}. \\
 && \M(s,0) \conv X \ar@{->>}[urr]^{\overline{\Omega}}
}
\end{align}
Here $\tau$ is the canonical surjection.
\end{proposition}

\begin{proof}
We have $1=\de\big(\M(s,0), \M(s_+,s)\hconv A\big)$
by Lemma~\ref{lem: simply linked}, and
$$\de\big(\M(s,0),\M(s_+,s)\big)+
\de\big(\M(s,0),A\big)=1$$
by Proposition~\ref{prop:deMM} with $x=u_{s_+ -1}, \ i = i_s$.
Hence $ \M(s,0) \conv \M(s_+,s) \conv A
$ has a simple head by Proposition~\ref{prop:3simple} (iii).
\end{proof}

\begin{proof}[End of the proof of Theorem \ref{thm: main}]
By the arguments above, we have proved the existence of $X$ which satisfies
\eqref{eq: conditions}.
By Proposition \ref{Prop: l2} and \eqref{eq: conditions} (c),
$\M(s,0) \conv X$ has composition
length $2$. Moreover, it has a simple socle and simple head.
On the other hand, taking the dual of \eqref{eq: conditions} (a),
we obtain a monomorphism
$$\sodot_{t; b_{t,s}>0} \M(t,0)^{\snconv  b_{t,s}} \monoto \M(s,0)\conv X$$
in $R\smod$.
Together with \eqref{eq: conditions} (b),
there exists a short exact sequence in $R\gmod$:
$$ 0 \to q^c\sodot_{t;b_{t,s}>0} \M(t,0)^{\snconv b_{t,s}} \to q^{\tLa(\M(s,0), X)}  \M(s,0) \conv X \to \sodot_{t;b_{t,s}<0} \M(t,0)^{\snconv(-b_{t,s})}
\to 0,$$ for some $c\in \Z$. By \cite[Corollary 2.24]{KKKO15}, $c$
must be equal to $1$.

It remains to prove that $X$ commutes with $\M(k,0)$ ($k\not=s$).
For any $k\in\K$, we have
\begin{align*}
\La(\M(k,0),X)&=\La(\M(k,0),\M(s,0)\hconv X)-\La(\M(k,0), \M(s,0))\\
&=\sum_{t;\;b_{t,s}<0}\La(\M(k,0),\M(t,0))(-b_{t,s})-\La(\M(k,0), \M(s,0))
\end{align*}
and
\begin{align*}
\La(X,\M(k,0))&=\La(X\hconv\M(s,0),\M(k,0))-\La(\M(s,0), \M(k,0))\\
&=\sum_{t;\;b_{t,s}>0}\La(\M(t,0),\M(k,0))b_{t,s}-\La(\M(s,0), \M(k,0)).
\end{align*}
Hence we have
\begin{align*}
2\de(\M(k,0),X)&  =-2\de(\M(k,0),\M(s,0)) -
\sum_{t;\;b_{t,s}<0}\La(\M(k,0),\M(t,0))b_{t,s} \\
& \hspace{23.5ex}-\sum_{t;\;b_{t,s}>0}\La(\M(k,0),\M( t  ,0))b_{t,s} \\
& =- \sum_{1 \le t \le r}\La(\M(k,0),\M(t,0))b_{t,s} \\ & = 2\delta_{k,s},
\end{align*}
We conclude that $X$  commutes with $M(k,0)$ if $k\not=s$.
Thus we complete the proof of Theorem~\ref{thm: main}.
\end{proof}

As a corollary we obtain the following answer to
the conjecture on the cluster monomials.
\Th
{\rm Conjecture~\ref{conj:intro}} in {\rm Introduction} is true, i.e.,
every cluster variable in  $A_{q}(\n(w))$ is a member of the upper global basis.
\enth
Theorem \ref{thm: main} also implies \cite[Conjecture 12.7]{GLS11} in the refined form as follows:

\begin{corollary}
$\Z[q^{\pm1/2}]\tens_{\Z[q^{\pm1}]} A_{q}(\n(w))_{\Z[q^{\pm1}]}$ has a quantum cluster algebra structure associated with the initial quantum seed
$[\mathscr{S}]=( \{ q^{-(d_i,d_i)/4}\D(i,0) \}_{1 \le i \le r}, L, \wB )$; i.e.,
$$\Z[q^{\pm1/2}]\tens_{\Z[q^{\pm1}]} A_{q}(\n(w))_{\Z[q^{\pm1}]} \simeq \mathscr{A}_{q^{1/2}}([\Seed]).$$
\end{corollary}

\bibliographystyle{amsplain}

\end{document}